\documentclass[12pt,reqno,a4paper, english]{amsart}

\usepackage[top=3.5cm,bottom=3cm,left=3.1cm,right=3.1cm]{geometry}
\usepackage{changepage}
\usepackage{placeins}
\usepackage{amsmath, amsthm, amsfonts, amssymb,enumerate}
\usepackage{bm} % Bold
\usepackage{amsthm} % For theorem without numbering
\usepackage{mathrsfs} % For \mathscr{}
\usepackage{mathtools}%makes typesetting equations easy, and mathtools makes those equations beautiful
\usepackage{slashed}% for slashing operators
\usepackage{dsfont} % pour 1 avec double bar
\usepackage{upgreek}% to use \uppsi,\upphi
\usepackage{nccmath}% to use small size of big fraction. For example \mfrac
\usepackage{stmaryrd}% cross out the arrow of convergence

\usepackage[dvipsnames]{xcolor} % To cancel the error that appears while using tikz with color
\usepackage{tikz}% To draw
\usepackage{microtype} %Alignement des fin de phrases
\usepackage{comment} % instead of the ctrl+right slash /
\usepackage[titletoc,title]{appendix}
\usepackage[font=small]{caption} %Caption--Legende d'une figure
\usepackage{systeme} % Exmaple : \system[x,y,z]{2x+4y+3z=2,x+z=1,y+z=2}
\usepackage{enumitem} %To  adjust the distance between tow item in itemize
\usepackage{stmaryrd}

\usepackage{pifont} % pour utiliser le xmark
\newcommand{\xmark}{\ding{55}}

\usepackage{etoolbox}

\usepackage{hyperref}
\hypersetup{colorlinks={true},linkcolor={awesome},citecolor={greenBleu},urlcolor={britishracinggreen}}
\definecolor{awesome}{rgb}{1, 0.1, 0}
\definecolor{amber}{rgb}{1.0, 0.75, 0.0}
\definecolor{greenBleu}{rgb}{0.0, 0.6, 0.5}
\definecolor{britishracinggreen}{rgb}{0.0, 0.26, 0.15}
\definecolor{caputmortuum}{rgb}{0.7, 0.55, 0.5}

\makeatletter
\newcommand\footnoteref[1]{\protected@xdef\@thefnmark{\ref{#1}}\@footnotemark}
\makeatother



\theoremstyle{plain}
\newtheorem{theorem}{Theorem}[section]
\newtheorem*{theorem*}{Theorem}
\newtheorem{prop}[theorem]{Proposition}
\newtheorem*{prop*}{Proposition}

\newtheorem{corollary}[theorem]{Corollary}
\newtheorem*{corollary*}{Corollary}
\newtheorem{lemma}[theorem]{Lemma}
\newtheorem*{lemma*}{Lemma}

\theoremstyle{definition}
\newtheorem{defi}[theorem]{Definition}

\theoremstyle{remark}
\newtheorem{Rq}{Remark}[section]

\numberwithin{equation}{section}

\DeclareMathOperator{\eee}{e}

\DeclareMathOperator{\Id}{Id}

\newcommand{\be}{\begin{equation}}
\newcommand{\ee}{\end{equation}}
\newcommand{\bes}{\begin{equation*}}
\newcommand{\ees}{\end{equation*}}

\newcommand{\mrm}[1]{\mathrm{#1}}

\newcommand{\R}{\mathbb{R}} 
\newcommand{\C}{\mathbb{C}} 
\newcommand{\N}{\mathbb{N}} 
\newcommand{\D}{\mathbb D} 
\newcommand{\T}{\mathbb{T}}
\newcommand{\Z}{\mathbb{Z}}

\newcommand{\Nzero}{\mathbb{N}_{\geq0}}

\newcommand{\U}{\mathcal{U}}

\newcommand{\G}{\mathcal{G}}
\newcommand{\I}{\mathcal{I}}

\newcommand{\qlq}{\forall}
\newcommand{\e}{\exists}
\newcommand{\bk}{\backslash}
\newcommand{\p}{\partial}

\newcommand{\eps}{\varepsilon}
\newcommand{\la}{\lambda}
\newcommand{\g}{\gamma}

\newcommand{\tauct}{\tau_{ct}}

\newcommand{\thetant}{\theta_n(t)}
\newcommand{\thetapt}{\theta_p(t)}
\newcommand{\thetanMoinsUnt}{\theta_{n-1}(t)}

\newcommand{\eix}{\mathrm{e}^{ix}}
\newcommand{\eiNx}{\mathrm{e}^{iNx}}
\newcommand{\ethetant}{\mathrm{e}^{i\theta_n(t)}}

\newcommand{\Ltwo}{L^2_+(\T)}
\newcommand{\Htwo}{H^2_+(\mathbb{T})}

\newcommand{\intc}{\int_{z\in \mathscr C(0,1)}}
\newcommand{\dz}{\frac{dz}{2\pi i z}}

\newcommand{\lracc}[1]{\left\{#1\right\}} %\set
\newcommand{\gnt}{g_n^{\,t}}
\newcommand{\gpt}{g_p^{\,t}}

\newcommand{\fnuzero}{f_n^{\,u_0}}

\newcommand{\fpuzero}{f_p^{\,u_0}}

\newcommand{\pkbar}{\overline{p_k}}
\newcommand{\pjbar}{\overline{p_j}}

\newcommand{\DomLu}{\mathrm{Dom}(L_u)}
\newcommand{\Tubar}{T_{\bar{u}}}

\newcommand{\Lu}{L_u}

\newcommand{\Lutilde}{\tilde{L}_{u}}
\newcommand{\Luzerotilde}{\tilde{L}_{u_0}}
\newcommand{\Luttilde}{\tilde{L}_{u(t)}}

\newcommand{\Butilde}{\tilde{B}_{u}}

\newcommand{\vertiii}[1]{{\vert\kern-0.25ex\vert\kern-0.25ex\vert #1 
    \vert\kern-0.25ex\vert\kern-0.25ex\vert}}% norm : |||.|||

\newcommand{\va}[1]{\lvert#1\rvert}

\newcommand{\Ldeux}[1]{\left\| {#1} \right\|_{L^2(\T)}} 

\newcommand{\ps}[2]{  \left\langle #1\,|\, #2  \right\rangle }

\newcommand{\parallelsum}{\mathbin{\!/\mkern-5mu/\!}}

\newcommand{\fr}{\widehat}

\let\oldtocsection=\tocsection

\let\oldtocsubsection=\tocsubsection

\let\oldtocsubsubsection=\tocsubsubsection

\renewcommand{\tocsection}[2]{\hspace{0em}\oldtocsection{#1}{#2}}
\renewcommand{\tocsubsection}[2]{\hspace{2em}\oldtocsubsection{#1}{#2}}
\renewcommand{\tocsubsubsection}[2]{\hspace{4.7em}\oldtocsubsubsection{#1}{#2}}

\begin{document}

%%%%%%%%%%%%%%%%%%%%%%%%%%%%%%%%%%%%%%%%%%%%%
%%%%%%%%%%%%%%% Title details %%%%%%%%%%%%%%%
%%%%%%%%%%%%%%%%%%%%%%%%%%%%%%%%%%%%%%%%%%%%%

\title[Traveling waves \& finite gap potentials for (CS)--equation]{Traveling waves \& finite gap potentials for the Calogero--Sutherland Derivative nonlinear Schrödinger equation
}

% \title[Global well--posedness of CS-DNLS in $H^{\frac{k}{2}}_+(\T)$\,, $k\in \Nzero$]{On global well--posedness results for the Calogero--Sutherland Derivative non--linear Schrödinger equation}

\author{Rana Badreddine}
\address{Universit\'e Paris-Saclay, Laboratoire de mathématiques d'Orsay, UMR 8628 du CNRS, B\^atiment 307, 91405 Orsay Cedex, France}
\email{\href{mailto: rana.badreddine@universite-paris-saclay.fr}{rana.badreddine@universite-paris-saclay.fr} }

\date{June 30, 2023.}

%%%%%%%%%%%%%%%%%%%%%%%%%%%%%%%%%%%%%%%%%%%%%
%%%%%%%%%%%%%%%%% Abstract %%%%%%%%%%%%%%%%%%
%%%%%%%%%%%%%%%%%%%%%%%%%%%%%%%%%%%%%%%%%%%%%
\begin{abstract}
% This paper is dedicated to the study of the traveling waves and finite gap potentials  for 
We consider the Calogero--Sutherland derivative nonlinear Schrödinger equation
    \begin{equation}\tag{CS}
    i\partial_tu+\partial_x^2u\,\pm\,\frac{2}{i}\,\partial_x\Pi(|u|^2)u=0\,,\qquad x\in\mathbb{T}\,,
    \end{equation}
		where  $\Pi$ is the Szeg\H{o} projector 
  $$\Pi\Big(\sum_{n\in \mathbb{Z}}\widehat{u}(n)\mathrm{e}^{inx}\Big)=\sum_{n\geq 0 }\widehat{u}(n)\mathrm{e}^{inx}\,.$$ 
         First, we characterize the \textit{traveling wave} $u_0(x-ct)$ solutions to the defocusing equation (CS$^-$)\,,
         and  prove  for the focusing equation (CS$^+$), that all the traveling waves must be either
         the constant functions, or plane waves, or rational functions. 
        A noteworthy observation is that
        the (CS)--equation is one of the fewest nonlinear PDE enjoying nontrivial traveling waves with arbitrary small and large $L^2$--norms.
        % Next, we establish the existence, respectively the nonexistence, of nontrivial \textit{stationary waves}  to the focusing equation (with sign $+$ in front of the nonlinearity), respectively to the defocusing equation.
        Second, we study the \textit{finite gap potentials}, and show that they are also rational functions, containing the traveling waves, and they can be grouped into sets that remain invariant under the evolution of the system. 
\end{abstract}

\maketitle
\tableofcontents

\section{\textbf{Introduction}}

\vskip0.5cm
In recent decades, the theory of traveling wave solutions has been the subject of intense research in theoretical  and numerical analysis.
Indeed, 
 many nonlinear PDEs  exhibit these type of waves \cite{Ch04, CH13,An09}\,.
Their importance resides as they are explicit solutions for nonlinear PDEs, and they can sometimes provide information regarding the dynamics of the equation.
However,
the problem of proving the existence of these waves can be more or less challenging depending on the nonlinear part of the PDE.
\vskip0.25cm
In this paper, we consider a type of derivative nonlinear Schrödinger equation with a  nonlocal nonlinearity, called the \textit{Calogero--Sutherland derivative nonlinear Schrödinger equation}  
\be\label{CS-DNLS}\tag{CS}
    i\p_tu+\p_x^2 u \,\pm\,2D\Pi(|u|^2)u=0\,,  \qquad x\in\T:=\R/(2\pi \Z)\,,
\ee
where $D=-i\p_x\,,$ and $\Pi$ denotes the Szeg\H{o} projector
\be\label{szego projector}
	\Pi\left(\sum_{n\in \Z}\,\fr{u}(n)\eee^{inx}\right):=\sum_{n\geq0 }\,\fr{u}(n)\eee^{inx}\,,
\ee
which is an orthogonal projector from $L^2(\T)$ into the Hardy space
\be\label{L2+}
	L^2_+(\T):=\lracc{u\in L^2(\T)\ |\ \fr{u}(n)=0\,,\,\qlq n\in \Z_{\leq -1}}\,.
 % \,\equiv\, \Pi(L^2(\T))\,.
\ee
Note that the operator $\Pi$ can also be read  as $\Pi=\frac{\Id+i\mathcal{H}+\ps{\cdot}{1}}{2}$ on the circle $\T$, where $\ps{u}{1}=\int_0^{2\pi}u \,\frac{dx}{2\pi}$ and $\mathcal{H}$ is the Hilbert transform 
\be\label{Hilbert op}
    \mathcal{H} u(x):=\sum_{n \in \mathbb{Z}}-i \operatorname{sign}(n)\, \widehat{u}(n) \mrm{e}^{i n x}\,, \qquad 
    \operatorname{sign}(0)=0\,.
\ee
\vskip0.25cm
We are interested in studying the traveling waves $u_0(x-ct)$ of this equation in the focusing (with sign $+$ in front of the nonlinearity) and defocusing case (with sign$-$) in the periodic setting, namely when $x\in\T$\,. 
As noted in \cite{An09}, the  presence of the dispersion operator $\p_x\mathcal{H}$ appearing in the nonlinearity can make the  problem of existence of traveling waves more complicated.
% This is the case of the Calogero--Sutherland DNLS equation \eqref{CS-DNLS}. 
The approach addressed in this paper to characterize the traveling waves 
is based 
on studying them, at a first stage, spectrally i.e. by means of spectral property of the Lax operator related to this equation (see below),
before deriving, at a second stage, their \textbf{explicit formulas}.
\footnote{It should be noted that the idea of using the spectral theory to derive the traveling waves of \eqref{CS-DNLS},  draws inspiration from  \cite[Appendix B]{GK21}\,, where the authors provide an alternative proof to the characterization of the traveling waves for the Benjamin--Ono equation \cite{AT91, Be67} by characterizing them first spectrally.}

 \subsection{Main results} 
 \label{results}
 \hfill
 \\
\textbf{Settings and notation.}
In the sequel, our study takes place with potentials in the Hardy Sobolev spaces of the torus $$H^s_+(\T):=H^s(\T)\cap L^2_+(\T)\,, \qquad s\geq 0\,,$$ where $L^2_+(\T)$ is defined in \eqref{L2+} and $H^s$ refers to the Sobolev space. 
We equip $\Ltwo$ with the standard inner product of $L^2(\T)$\,, $$\ps{u}{v}=\int_0^{2\pi}u \bar v \,\frac{dx}{2\pi }\,.$$
We recall also, that  via the following isometric isomorphism
\begin{align*}
    z\in\D\,,\quad	u(z)=\sum_{k\geq 0}\fr{u}(k)z^k &\qquad\longmapsto\qquad u^*(x):=\sum_{k\geq 0}\fr{u}(k)\eee^{ikx}\,, \quad x\in\T\,,
    \\
    &\sum_{k \geq 0}\,\va{\fr{u}(k)}^2<\infty\,,
\end{align*} 
one can interpret any element of the Hardy space 
 as an analytic function on the open unit disc $\D$\,, whose trace on the boundary $\partial\D$ is in $L^2\;$~.\footnote{~For a simple introduction to the different definitions of Hardy space, we refer to \cite[Chapter 3.]{GMR16}}
We will frequently utilize this property in various proofs.
Furthermore, we denote by $\D$ the open unit disc on $\C$\,, $\D^*:=\lracc{z\in\C\;;\;0<|z|<1}\,$. Moreover, $\N$ denotes the positive integers $1,2,3,\ldots$ And for all $a\in\N\cup \lracc{0}\,,$  $\N_{\geq a}$ refers to the set of integer numbers $\lracc{n\in\Z\,;\, n\geq a}\,.$  
\vskip0.3cm
\begin{center}
    ***
\end{center}
\vskip0.3cm

First,  we deal with the defocusing Calogero--Sutherland DNLS equation
\be\label{CS-defocusing}\tag{CS$^-$}
    i\p_tu+\p_x^2u-2D\Pi(\va{u}^2)u=0\,.
\ee
We denote by $\mathcal{G}_1$ the set of the trivial traveling waves, made up from the constant functions and the plane wave solutions
\be\label{G1-intro}
    \G_1=\lracc{C\eee^{iN(x-Nt)}\,\mid\, C\in\C\,,\, N\in \Nzero}\,.
\ee

\begin{theorem}[Characterization of the traveling waves of \eqref{CS-defocusing}]\label{traveling waves for the defocusing CS}
    A potential $u$ is a traveling wave  of \eqref{CS-defocusing} if and only if $u\in \mathcal{G}_1$ or 
    \be\label{travWaves-defoc}
       u(t,x):= \eee^{i\theta}\left(\alpha+
        \frac{\beta}{1-p\eee^{iN(x-c t)}}\right)\,,
                    \qquad
            p\in \D^*\,,\, \theta\in \T\,, \,
    \ee
    where $N\in\N\,,$ $\,c:=-N\Big(1+\frac{2\alpha}{\beta}\Big),$ and $(\alpha,\beta)$ are two real constants satisfying
    \be\label{alpha, beta-defo}
            \alpha\beta +\frac{\beta^2}{1-\va{p}^2}=-N\,.
    \ee
\end{theorem}
\begin{Rq}\hfill
   The condition \eqref{alpha, beta-defo}  implies that  the real constants $\alpha$ and $\beta$ must be of opposite signs.
\end{Rq}
\vskip0.25cm
Second, we pass to the focusing Calogero--Sutherland DNLS equation
\be\label{CS-focusing}\tag{CS$^+$}
    i\p_tu+\p_x^2u+2D\Pi(\va{u}^2)u=0\,.
\ee
By changing the sign in front of the nonlinearity, the strategy adopted in the defocusing case to exhibit the traveling waves becomes significantly more complicated.  
However,
% at this stage, 
we can  ensure the existence of a larger set of traveling wave solutions for \eqref{CS-focusing} comparing to \eqref{CS-defocusing}\,, 
and that all the non--trivial traveling waves $u(t,x):=u_0(x-ct)$ of \eqref{CS-focusing} are also rational functions.
% Note that, the full characterization of all the traveling waves for \eqref{CS-focusing} is still an open problem.

\begin{theorem}
    % The set $\mathcal{T}$ of t
    The traveling waves $u_0(x-ct)$ of \eqref{CS-focusing} are either rational functions or trivial waves in $\G_1$\,.
    % contains only rational functions or the trivial waves in $\G_1$\,. 
    In addition, the potentials
     \[
                u(t,x):= \eee^{i\theta}\left(\alpha+
                \frac{\beta}{1-p\eee^{iN(x-c t)}}\right)\,, \qquad
            p\in \D^*\,,\, \theta\in \T\,, \, N\in\N\,,
    \]
            where  $c= -N\left(1+\frac{2\alpha}{\beta}\right)$\,, $\,(\alpha,\beta)\in\R\times\R$ such that 
        \be\label{alpha, beta-fo, intro}
              \alpha\beta +\frac{\beta^2}{1-\va{p}^2}=N\,,
        \ee
        and the potentials 
         \[
            u(t,z)=\eee^{i\theta}\eee^{im(x-mt)}\left(\alpha+
                \frac{\beta}{1-p\eee^{i(x-m t)}}\right)
            \qquad p\in \D^*\,,\, \theta\in \T\,, \, m\in\N\,,
        \]
        where $(\alpha,\beta)\in\R\times\R$ such that 
        \[
             \alpha\beta +\frac{\beta^2}{1-\va{p}^2}=1\,,\qquad \beta (m-1)=2\alpha\,,
        \]
        % are parts of this set $\mathcal{T}\,.$
        are parts of the set of the traveling waves of \eqref{CS-focusing}\,. 
\end{theorem}

\begin{Rq}\hfill
 It is worth noting that the condition on $(\alpha,\beta)$ appearing in \eqref{alpha, beta-fo, intro} for the focusing case, allows to obtain a \textit{larger set} of traveling waves in comparison to the condition \eqref{alpha, beta-defo} of the defocusing case.
        Indeed, \eqref{alpha, beta-fo, intro} enables, for instance $\alpha$ or $\alpha+\beta$ to vanish,  which leads respectively to the  following traveling waves
    $$
        u(t,x)=\eee^{i\theta}\frac{\sqrt{N(1-\va{p}^2)}}{1-p\eee^{iN(x+Nt)}}
        \quad \text{and} \quad 
       u(t,x)=\eee^{i\theta}\frac{\sqrt{N(1-\va{p}^2)}\,\eee^{iN(x-Nt)}}{1-p\eee^{iN(x-Nt)}}\,. 
    $$
    Contrary to the focusing case, no traveling waves $u(t,x):=u_0(x-ct)$ with a profile $
 \displaystyle   u_0(x):=\frac{\beta}{1-p\eee^{iNx}}\;$ or $\; \displaystyle u_0(x):=\frac{\alpha \eee^{iNx}}{1-p\eee^{iNx}}\:$ can be found for the \eqref{CS-defocusing}--equation because otherwise, thanks to  \eqref{alpha, beta-defo}\,, 
    $$
        \frac{\beta^2}{1-\va{p}^2}=-N 
    \qquad \text{or} \qquad
        \frac{ \va{p}^2}{1-\va{p}^2}\,\beta^2 =-N 
    $$
    which is clearly impossible for $p\in\D^*\,.$

\end{Rq}
\vskip0.25cm
\begin{Rq}(\textbf{The $L^2$--norm and the speed of the traveling waves} of \eqref{CS-DNLS}\,)\hfill
    \begin{enumerate}[itemsep=10pt]
        \item  As will be established in Subsection~\ref{L2 norm and speed defoc} for the defocusing equation and in Subsection~\ref{speed and L2 norm foc} for the focusing equation, 
       the $L^2$--norm 
       of the  non--trivial 
       traveling waves of \eqref{CS-DNLS}
       can be arbitrarily small or large in $\Ltwo\,.$ More rigorously, for any $r>0\,,$ there exists  a non--trivial traveling wave $ u(t,x):=u_0(x-ct)$ of \eqref{CS-DNLS}  where 
    \[
        \|u_0\|_{L^2}=r\,.
    \]
    \item \textit{For the defocusing \eqref{CS-defocusing}--equation}.
    The nontrivial traveling waves $u$ of the form \eqref{travWaves-defoc} propagate to the right with a speed $c>N\,,$ where $N$ is the degree appearing in the denominator of $u\,.$
    % In addition, the limit speed $c=N$ is only attained by the trivial traveling wave $u\in\G_1\,$. 
    In addition, when $\|u\|_{L^2}\to +\infty\,,$ we have $c\to+\infty\,,$ and when $\|u\|_{L^2}\to 0$ then $c\to N\,.$
       (See Remark~\ref{Rq speed defocusing} and Subsection~\ref{L2 norm and speed defoc} for the proofs).
       \vskip0.2cm
       \noindent
        \textit{For the focusing \eqref{CS-focusing}--equation}. Contrary to the defocusing equation, the \eqref{CS-focusing}'s nontrivial traveling waves do not necessarily  propagate at a relatively high speed (i.e. $c\to\infty$) when $\|u\|_{L^2}^2$ is large (i.e. $\|u\|_{L^2}\to\infty)$.
    In fact, the speed of the traveling waves in the focusing case is independent of the size of its $\Ltwo$--norm\,.
    % it is possible to find  traveling waves of  \eqref{CS-focusing}  propagating to the left or to the right, with  
     % arbitrarily small or large speed $c$\,,
     % in $\R$ and this speed can be chosen independently of the size of the $L^2$--norm of the wave.
    We refer to 
    % Proposition~\ref{prop speed travel waves foc} and
    Remark~\ref{Rq speed focusing L2 norm infinity} for an example.
    \end{enumerate}
\end{Rq}
\vskip0.25cm
In light of the  previous remarks, we infer that the  Calogero--Sutherland DNLS equation  enjoys a  significantly richer dynamic in the focusing case. 
In particular, one can observe that  the  \eqref{CS-focusing} admits non-trivial \textit{stationary waves} $u(t,x):=u_0(x)$, which is not the case of the defocusing equation. 
An example of nontrivial stationary waves for \eqref{CS-focusing} is
 $$
            u(t,x):=\eee^{i\theta} \sqrt{\frac{N(1-\va{p}^2)}{2(1+\va{p}^2)}}
            \left(1-\frac{2}{1-p\eee^{iNx}}\right)\,,\qquad p\in\D^*\,,\ \theta\in\T\,, \ N\in\N\,.
$$
% Conversely, the defocusing equation \eqref{CS-defocusing} does not exhibit any stationary wave solutions except the complex constant functions.
% \vskip0.15cm
% As a third step, our study focuses on the existence of \textit{stationary wave solutions} $u(t,x)\coloneqq u_0(x)$ to the Calogero--Sutherland DNLS \eqref{CS-DNLS}\,.

% \begin{corollary}[Stationary waves]\label{Stationary waves}
% The set $\mathcal{E}$ of stationary waves $u(t,x):=u_0(x)$ of the focusing \eqref{CS-focusing}--equation is restricted to constant functions and rational solutions. In particular,
% the potentials
       
%    are stationary solutions for \eqref{CS-focusing}\,.
% Conversely, the defocusing \eqref{CS-defocusing} equation does not exhibit stationary wave solutions except the complex constant functions.
% \end{corollary}

% \begin{Rq}
%     For the full characterization of the set $\mathcal{E}$ of the traveling wave solutions
% \end{Rq}

\vskip0.35cm
At a second stage, we study the \textit{finite gap potentials} of the Calogero--Sutherland DNLS equation \eqref{CS-DNLS}\,, i.e.  potentials satisfying that, from a certain rank,  all the gaps between the consecutive eigenvalues of the Lax operator are equal to $1$ (see Subsection~\ref{integranility of CS} for the Lax operator). 
% and Definition~\ref{finite gap potential} for the finite gap potentials
 It turns out that these potentials are \textit{multiphase solutions} containing the stationary and  traveling waves of \eqref{CS-DNLS}\,.
The following theorem aims to characterize the finite gap potentials on $\T\,$ in the state space.

\begin{theorem}[Characterization in the state space of the \eqref{CS-DNLS}'s finite gap potentials] \label{finite gap potential th characterization}
    The finite gap potentials of \eqref{CS-DNLS} are either the functions 
     $u(x)=C\eee^{iNx}$\,, $C\in\C^*\,,$ $N\in\Nzero\,,$ or the rational function 
    \[
     u(x)=\eee^{im_0\, x}\prod_{j=1}^r\left(\frac{\eix-\pjbar}{1-p_j\eix}\right)^{m_j-1}\left(a+\sum_{j=1}^r\frac{c_j}{1-p_j\eix}\right), \quad p_j\in\D^*\,,\ p_k\neq p_j\,,\ k\neq j\,,  
    \]
    where, for $N\in\N\,,$ $\,m_0\in\llbracket0,N-1\rrbracket\,,\,$ $m_1, \ldots,m_r\in\llbracket 1, N\rrbracket$ such that  
$
    m_0+\sum_{j=1}^rm_j=N\,,
$  and $(a, \,c_1\,,\,\ldots\,,\,c_r)\in\C\times\C^{r}$ satisfy for all $j=1\,,\,\ldots\,,r\,,$
    \begin{enumerate}[label=(\roman*),itemsep=2pt]
        \item In the defocusing case,
        \be\label{cond a, ck-defoc intro}
            \overline{a}\,c_{j}\,+\,\sum_{k=1}^{r}\,\frac{c_{j}\,\overline{c_k}}{1-p_j\overline{p_k}}
            =-m_j\,,
        \ee
        \item In the focusing case, 
        \be\label{cond a, ck-foc intro}
            \overline{a}\,c_{j}\,+\,\sum_{k=1}^{r}\,\frac{c_{j}\,\overline{c_k}}{1-p_j\overline{p_k}}
            =m_j\,,
        \ee
    \end{enumerate}
    with  $a\neq 0$ if $m_0\neq0\,.$
      Moreover, these finite gap potentials can be regrouped  into sets that remain invariant under the evolution of \eqref{CS-DNLS}\,.
\end{theorem}

\vskip0.3cm
In order to establish the results mentioned above, it is necessary to provide an overview regarding the integrability of the Calogero--Sutherland derivative nonlinear Schrödinger equation \eqref{CS-DNLS}.
% ~\\

\subsection{About the Calogero--Sutherland DNLS equation}
The Calogero--
Sutherland DNLS equation \eqref{CS-DNLS} has been actively studied by physicists and engineers. In particular, we cite the works of Tutiya \cite{Tu09}, Berntson--Fagerlund \cite{BF22}, Stone--Anduaga--Xing \cite{SAX08}, Polychronakos \cite{Po95a,Po95a} and  Matsuno \cite{Ma00, Ma01a, Ma01b, Ma02a, Ma02b, Ma03, Ma04a, Ma04b,Ma23}... 
% Formally, 
% the focusing \eqref{CS-focusing}--equation has been derived by Abanov and his collaborators as a thermodynamic limit of the classical Calogero--Sutherland
% systems\footnote{~A physical model describing the pairwise interaction of $N$ identical particles. In the case where these particles lie in the circle, they interact with a trigonometric potential $\frac{1}{\sin^2(x)}$. By taking the thermodynamic limit of this $N$-body problem, one finds formally the focusing \eqref{CS-DNLS}--equation.
% } \cite{ABW09}.
% On the other side, the defocusing equation of \eqref{CS-DNLS} (with the sign $-$) has been obtained as a limit version\footnote{~ The intermediate nonlinear Schrödinger equation introduced by Pelinovsky in 1995 \cite{Pe95, PG95}, by a multiscale expansion of the intermediate long-wave \cite{KKD78},
% and is given as \be\label{INSE}\tag{INS}
% 	i\p_tu = \p_xu^2 + (i - T )\p_x(\lvert u\rvert^2)\,u\,,
% 	\ee
% 	where  
% 	$T$ is the integral operator
% 	$$
% 	Tu(t,x)=\frac{1}{2\delta} \,\text{p.v.}\int_{-\infty}^{+\infty}\coth\Big(\frac{\pi(x-y)}{2\delta}\Big) u(t,y)\, dy\,.
% 	$$
%  The complex function $u$  in \eqref{INSE} represents the envelope of the fluid, and $\delta$ denotes its total depth.
%   Formally, by taking $\delta\to \infty$, one deduces the defocusing \eqref{CS-DNLS} equation.
%   } of  the \textit{intermediate nonlinear Schrödinger equation}, describing thereby, the interfacial wave packets
% in a shallow--deep limit of a stratified fluid.
\vskip0.25cm
Mathematically, recent progress has been made with regard to this equation. In this subsection, we provide a brief overview of some  established results concerning~\eqref{CS-DNLS}\,.

\subsubsection{Local and global well--posedness results}
 To the best of the author's knowledge, the first LWP result for \eqref{CS-DNLS} equation traced back to De Moura \cite{deM07} who established the LWP
\footnote{
    ~Actually, they prove the local well--posedness of a family of nonlocal nonlinear Schrödinger equation  \cite{PG96} that includes also the \eqref{CS-DNLS}--equation.
} of \eqref{CS-DNLS} for small initial data in $H^s(\R)$ with $s\geq 1\,$, and extend his LWP's result to a GWP by means of the gauge transformation.
% Using a gauge transformation, De Moura--Pilod \cite{deMP10} proved the local well-posedness\footnote{
%     ~Actually, they prove the local well--posedness of a family of nonlocal nonlinear Schrödinger equation  \cite{PG96} that includes also the \eqref{CS-DNLS}--equation.
% } of \eqref{CS-DNLS} in $H^s(\R)$ for $s>\frac12\,$, improving thus, the LWP's result already obtained by De Moura in $H^s(\R)\,,$ $s\geq 1$  with small initial data 
More recently, Barros--DeMoura--Santos present in \cite{BMS19} the LWP of \eqref{CS-DNLS} for small initial data in the Besov space $B_2^{\frac12,1}(\R)\,$. 
\vskip0.25cm
Besides, observe that the Calogero--Sutherland DNLS equations \eqref{CS-DNLS} is invariant under the scaling
$$
    u_{\lambda}(t,x)=\la^{\frac12}u(\la t,\la^2 x)\,,\qquad  \la>0\,.
$$
This suggests that \eqref{CS-DNLS} is $L^2$--critical.
In the Hardy Sobolev spaces settings, i.e. in $H^s_+:= H^s\cap L^2_+\,,$ where recall $L^2_+$ is the Hardy space defined in \eqref{L2+} in the periodic case, and as follows in the non-periodic case
$$
    L_{+}^2(\mathbb{R})=\left\{u \in L^2(\mathbb{R}) \ ;\ \operatorname{supp} \hat{u} \subset[0, \infty)\right\}\,,
$$ G\'erard--Lenzmann \cite{GL22} obtained the LWP in $H^s_+(\R)$ with $s>\frac12\,,$ by following the arguments of \cite{deMP10}\,. 
Furthermore, by the virtue of a Lax pair structure associated with the Calogero--Sutherland DNLS equation \eqref{CS-DNLS} (see below), they inferred the global well--posedness of the equation  in all $H^k_+(\R)\,$, $k\in\N_{\geq 1}$ 
% \footnote{~In fact,
    % the presented arguments in their paper \cite{GL22} can even lead to the GWP in all the Hardy Sobolev spaces $H^s_+(\R)$ with $s>\frac12\,$ by using an interpolation method \cite[Chapter I. 4.]{Ta81}.
% } 
for small initial data $\|u_0\|_{L^2(\R)}<\sqrt{2\pi}$ in the focusing case.
\vskip0.25cm
Moving to the {periodic setting, i.e. when $x\in\T\,,$ a recent work of the author shows  the GWP of \eqref{CS-DNLS} in all  $H^s_+(\T)$\,, $s\geq 0\,,$ for small critical initial data in the focusing case, namely when $\|u_0\|_{L^2(\T)}<1\,,$ and for arbitrary initial data in the defocusing case. 
In particular, the extension of the flow to the critical space $L^2_+(\T)$ has been achieved after deriving \textit{the explicit formula} for the solution of the Calogero--Sutherland DNLS equation \eqref{CS-DNLS} \cite[Proposition 2.5]{Ba23}. Moreover, under the same assumptions,  the relative compactness of the trajectories has been established in  $H^s_+(\T)$\,, for all $s\geq 0$  \cite{Ba23}.

\subsubsection{Integrability of the (CS)-equation} \label{integranility of CS}
One of the most remarkable features of the Calogero--Sutherland DNLS equation is its integrability as a PDE on $\R$ and on $\T$. In fact, it enjoys a\textit{ Lax pair structure} in the focusing and defocusing case \cite{GL22,Ba23}~: for any $u\in H^s_+(\T)\,,$ $s>\frac32\,,$ there exists two operators $(L_u,B_u)$ satisfying the Lax  equation 
$$
\frac{dL_u}{dt}=[B_u,L_u]\,, \qquad [B_u,L_u]:=B_uL_u-L_uB_u\,,
$$
where 
\begin{enumerate}[label=(\roman*)]
    \item In the focusing case,
        \be\label{Lax operators foc case}
            L_u=D-T_uT_{\bar{u}}\,,
            \qquad
            B_u=T_{u}T_{\partial_x\bar{u}}-T_{\partial_x{u}}T_{\bar{u}} +i(T_{u}T_{\bar{u}})^2\,.
        \ee
    \item In the defocusing case,
    \be\label{Lax operators defoc case}
        \tilde{L}_u=D + T_uT_{\bar u}\,, 
        \qquad
        \tilde{B}_u=-T_uT_{\partial_x\bar u}+T_{\partial_xu}T_{\bar u} +i(T_uT_{\bar u})^2\,.
    \ee
\end{enumerate}
The differential operator $D$ is $-i\p_x\,,$ and $T_u$ is the Toeplitz operator of symbol $u$ defined for any $u\in L^\infty$ by
\be\label{Toeplitz operator}
    T_{u}f=\Pi(uf)\,, \qquad \qlq f\in L^2_+\,,
\ee
where $\Pi$ is the Szeg\H{o} projector introduced in \eqref{szego projector}\,. Note that since we are working in the Hardy space, $L_u$ is a semi-bounded operator from below
and $\tilde{L}_u$ is a nonnegative operator. In addition, 
as noted in \cite[Proposition 2.3]{Ba23}\,, the Lax operators $L_u$ and $\tilde{L}_u$ are self--adjoint operators of domain $H^1_+(\T)\,,$ and are of compact resolvent. 
Therefore, their spectra are made up of a sequence of eigenvalues going to $+\infty\,,$
\begin{align}\label{spectre}
    \sigma(L_u):=&\lracc{\nu_0(u)\leq \ldots\leq \nu_n(u)\leq \ldots},\qquad \nu_0(u)\geq -\|u\|_{L^\infty}^2\,,
    \\
    \sigma(\tilde{L}_{u}):=&\lracc{\la_0(u)\leq \ldots\leq \la_n(u)\leq \ldots}\,,\qquad \la_0(u)\geq0\;.\notag
\end{align}
Recall that any Lax operator satisfies the isospectral property 
\be\label{isospectral property}
    L_{u_0}=U(t)^{-1}L_{u(t)}U(t)\,,
\ee
where  $u_0$ is the initial data, $u(t)$ is the evolution of the solution starting from $u_0\,,$ and $U(t)$ is a family of operators solving the Cauchy problem
\[
    \begin{cases}
    \frac{d}{dt}U(t) =B_{u(t)}\, U(t)\\
			U( 0) =\operatorname{Id}
		\end{cases}\,.
    \]
 The  identity \eqref{isospectral property} implies that the spectrum of $L_{u(t)}$ is invariant by the evolution, i.e.  $\nu_n(u(t))=\nu_n(u_0)$ and $\la_{n}(u(t))=\la_n(u_0)$ for all $n\,$.
Therefore, in the sequel,
 we omit the variable $u$  in $\nu_n(u)$ and $\la_n(u)$  when it does not make confusion.
\vskip0.1cm
 Further information regarding the spectrum of the Lax operators will be provided in Section \ref{spectral property of the Lax operator}.
\vskip0.3cm
\vskip0.25cm

 \subsubsection{Traveling waves on \texorpdfstring{$\R$}{R}}
Let us mention that the focusing Calogero--Sutherland DNLS equation \eqref{CS-focusing} also enjoys traveling waves and stationary waves in the \textit{nonperiodic case }(i.e. $x\in\R$). 
They are of the form 
 $$
    u(t,x)=\mrm{e}^{i\theta}\mrm{e}^{iv(x-vt) }\,\la^\frac12\,\mathcal{R}\big(\la (x-2vt)+y\big)\,,\quad \la>0,\ y\in\R,\ \theta \in \T\,, \, v\in\R\,,
  $$
where the profile
\be
    \mathcal{R}(x)= \mrm{e}^{i\theta} \frac{\sqrt{2\, \mrm{Im} p}}{x+p}\,, \qquad p\in \C_+\,,\ \theta\in\R\,,
 \ee
  is obtained as ground states (minimizers) for the energy functional  \cite[Section~4]{GL22}.
%  G\'erard--Lenzmann characterize the stationary waves in the Hardy Sobolev space $H^1_+(\R)$ as 
%  \be
%     \mathcal{R}(x)= \mrm{e}^{i\theta} \frac{\sqrt{2\, \mrm{Im} p}}{x+p}\,, \qquad p\in \C_+\,,\ \theta\in\R\,,
%  \ee
% where these waves have been obtained as ground states (minimizers) for the energy functional. Afterward, they derived all the traveling waves of \eqref{CS-focusing}, which are obtained from the stationary waves by applying translations and scaling :
%   $$
%     u(t,x)=\mrm{e}^{i\theta}\mrm{e}^{iv(x-vt) }\,\la^\frac12\mathcal{R}\big(\la (x-2vt)+y\big)\,,\quad \la>0,\ y\in\R,\ \theta \in \T\,, \, v\in\R\,.
%   $$
  Notice that all these waves are of $L^2$--norm equal to $\sqrt{2\pi}\,$. 
  Therefore, this situation differs from the torus $\T$, where in the latter case, there is no $L^2$-threshold that would prevent the existence of small or large  traveling waves in $L^2(\T)$\,.
  % for the Calogero--Sutherland DNLS equation, 
  % i.e. there exist traveling waves with small and large $L^2$--norm belonging to $(0,+\infty)\,.$  
  Essentially, the main reason that leads to a more diverse class of traveling waves in the periodic setting compared to the non--periodic seting, is due to 
  % the difference between the two case originate from 
  the spectral property carried by the Lax operator in both cases.
  Indeed,  on $\mathbb{R}$\,, the Lax operator has an absolute continuous spectrum and a finite number of eigenvalues \cite[Section~5]{GL22}. In contrast with $\T$, the Lax operator present only point spectrum formed by eigenvalues \cite[Section 2]{Ba23}.

\vskip0.25cm
To summarize, we refer to the following table (Table~\ref{Tab: Table trav waves})\,.
\FloatBarrier
\begin{table}[!ht]
\begin{adjustwidth}{-1.1cm}{}
        \centering
\begin{tabular}{p{0.33\textwidth}|p{0.3\textwidth}|p{0.3\textwidth}|}
\cline{2-3} 
  &\centering  Focusing \eqref{CS-focusing} on $\R$ & \hskip0.1cm Defocusing \eqref{CS-focusing} on $\R$ \\
\hline 
 \multicolumn{1}{|l|}{Stationary waves} & \centering \checkmark&  
 \\
\cline{1-2} 
 \multicolumn{1}{|l|}{Traveling waves} & \centering \checkmark &  \\
\cline{1-2} 
 \multicolumn{1}{|l|}{Wave speed} &\centering $c\in\R$  &  \\
\cline{1-2} 
 \multicolumn{1}{|l|}{$L^2$--norm of the traveling waves} & \centering $\|u\|_{L^2}=\sqrt{2\pi}$ &  \\
\hline 
  & \centering Focusing \eqref{CS-focusing} on $\T$ & \hskip0.1cm Defocusing \eqref{CS-focusing} on $\T$ \\
\hline 
 \multicolumn{1}{|l|}{Non--trivial stationary waves} & \centering \checkmark & \hskip2.1cm \xmark \\
\hline 
 \multicolumn{1}{|l|}{Traveling waves} & \centering \checkmark & \hskip2.1cm \checkmark \\
\hline 
 \multicolumn{1}{|l|}{Wave speed} & \centering$c\in\R$& $\hskip1.75cm c\geq N$  \\
\hline 
 \multicolumn{1}{|l|}{$L^2$--norm of the non--trivial traveling waves} & \centering $\|u\|_{L^2}\in(0,+\infty)$ & $\hskip0.8cm\|u\|_{L^2}\in(0,+\infty)$ \\
 \hline
\end{tabular}
\end{adjustwidth}
\begin{center}
\caption{
Where $N\in\N$ is the denominator's degree of a  traveling wave of the form~\eqref{travWaves-defoc}.
In addition, by a non--trivial traveling wave, we mean the traveling wave  that does not belong to $\G_1$ where $\G_1$ is the set of trivial traveling waves defined in \eqref{G1-intro}\,. Moreover, by non--trivial stationary waves we mean the solution $u(t,x)=u_0(x)$ that are not constant functions.}
\label{Tab: Table trav waves}
\end{center}
\end{table}

\subsection{Outline of the paper}
The paper is organized as follows.
In Section~\ref{spectral property of the Lax operator}\,, we present some spectral properties concerning  the eigenvalues and the eigenfunctions of the Lax operators $L_u$ and $\Lutilde\,.$ Moving on to Section~\ref{traveling waves for (CS-)}, we focus on the traveling waves of the defocusing Calogero-Sutherland DNLS equation~\eqref{CS-defocusing}. This section follows a two-step process. Subsection~\ref{spectral characterization of traveling waves of (CS-)} provides a spectral characterization of these waves, while Subsection~\ref{Explicit formulas trav. Waves for (CS-)} derives their explicit formulas. Moreover, Subsection~\ref{L2 norm and speed defoc} includes remarks concerning the speed and $L^2$-norm of these traveling waves for the defocusing \eqref{CS-defocusing}--equation.
\vskip0.2cm
In Section~\ref{Traveling waves for (CS+)}\,, we delve into the analysis of traveling waves   for the focusing Calogero--Sutherland DNLS equation~\eqref{CS-focusing}\,. 
% We begin by proving their existence in Subsection~\ref{Existence of trav waves for (CS+)}. 
Thus, we describe the set of traveling waves of \eqref{CS-focusing} in Subsubsection~\ref{characterization of traveling waves for (CS+)}, and we highlight the presence of a larger set of traveling waves in the focusing case comparing to the defocusing case.
% additional traveling waves in the focusing case by providing an illustrative example. 
% Furthermore, Subsection~\ref{characterization of traveling waves for (CS+)} characterizes these waves.
Similarly to the defocusing case, some remarks related to the speed and the $L^2$--norm of the traveling waves of \eqref{CS-focusing} are  discussed in Subsection~\ref{speed and L2 norm foc}\,, and in particular we establish the existence of stationary wave solutions for the focusing Calogero-Sutherland DNLS equation~\eqref{CS-focusing}\,. 
\vskip0.2cm
Note that in order to describe the traveling waves of \eqref{CS-focusing}, one need to understand the set of finite gap potentials. To this end, 
Section~\ref{finite gap potentials} is dedicated to the study of finite gap potentials for the Calogero--Sutherland DNLS equation~\eqref{CS-DNLS}.
\vskip0.2cm
Throughout this paper, we have assumed sufficient regularity on the solutions. However, in Section~\ref{Rk on the regularity}, we discuss how the same analysis can be extended to solutions with lower regularity. Lastly, in Section~\ref{Open Problems}, we present some open problems for further exploration.

\subsection*{Acknowledges}
The author would like to thank her Ph.D. advisor \textit{Patrick G\'erard }
for proposing this problem  and suggesting \cite[Appendix B]{GK21} as a useful reference to start the investigation.

%----------------------------------------------------
%----------------------------------------------------
\section{\textbf{Spectral properties for the Lax operators}}\label{spectral property of the Lax operator}

As mentioned in the introduction, our aim is to describe the traveling waves of the Calogero-Sutherland DNLS equation \eqref{CS-DNLS}. 
 In order to accomplish this goal,  our strategy relies on characterizing them first in the state space, by means of some spectral tools of the Lax operators $L_u$ and $\Lutilde$ introduced in \eqref{Lax operators foc case} and in \eqref{Lax operators defoc case}\,,  respectively. Therefore, we need to delve deeper into the spectral properties of the Lax operators.

\vskip0.2cm

% In this section, we give some spectral information with a potential $u$ in the state space.
In the sequel, we assume, for more convenience, that $u$ is any function of the state space with enough regularity, for example,  $u\in H^2_+(\T)\,.$  
But, it is worth mentioning that the analysis can be easily extended to potentials with less regularity as well (see Section~\ref{Rk on the regularity}). 
Besides, recall from \eqref{spectre}\,, that the Lax operators $\Lutilde$ and $L_u$ have point spectra, bounded from below
\begin{align*}
    \sigma(\tilde{L}_{u}):=&\lracc{\la_0\leq \ldots\leq \la_n\leq \ldots}\,,\qquad \la_0\geq0\;,
    \\
    \sigma(L_u):=&\lracc{\nu_0\leq \ldots\leq \nu_n\leq \ldots},\qquad \nu_0\geq -\|u\|_{L^\infty}^2\,.
\end{align*}
\vskip0.25cm
The following proposition aims to give more information, regarding the multiplicity of the eigenvalues $(\nu_n)$ and $(\la_n)\,$. But before, we need to recall two useful commutator identities. We denote by $S$ the \textit{shift operator} defined as
\begin{equation}\label{shift}
    S:\Ltwo\to\Ltwo\,, \qquad 
    Sh(x)=\eee^{ix} h(x)\,.
\end{equation}
Thus, for all $u\in H^2_+(\T)\,,$ we have from \cite[Lemma 2.3]{Ba23}\,,
\begin{align}
    \Lutilde S =S \Lutilde+S+\ps{\,\cdot}{S^*u}u\,,
    \label{commutateur [Lu,s]}
    \\
    L_uS =S L_u+S-\ps{\,\cdot}{S^*u}u\,,\notag
\end{align}
where $S^*$ denotes the adjoint operator of $S$
\begin{equation}\label{adjoint-shift}
    S^*:\Ltwo\to\Ltwo\,, \qquad 
    S^*h(x)=\Pi(\eee^{-ix} h(x))\,,
\end{equation}
with  $\Pi$ is the Szeg\H{o} projector defined in \eqref{szego projector}\,, and 
%the two operators 
$\Lu$ and $\Lutilde$ are defined in \eqref{Lax operators foc case} and \eqref{Lax operators defoc case}\,.
In addition, we also have from the same lemma \cite[Lemma 2.3]{Ba23}\,,
\begin{align}
    [S^*, B_u]=i\Big( S^*L_u^2 \,-\, (L_u+\Id)^2S^* \Big)\,,
    \label{commutateur [Bu,s]}
    \\
    [S^*, \Butilde]=i\Big( S^*\Lutilde^2 \,-\, (\Lutilde+\Id)^2S^* \Big)\,,\notag
\end{align}
where $[S^*, B_u]$ denotes the commutator $S^*B_u-B_uS^*\,,$ $B_u$ and $\tilde{B}_u$ are the two skew-adjoint operators of the Lax pairs, defined respectively in \eqref{Lax operators foc case} and \eqref{Lax operators defoc case}\,.

\begin{prop}[Multiplicity of $(\la_n)$ and $(\nu_n)$]\label{multiplicite val propres}\hfill

\noindent
\underline{Defocusing case.}
The eigenvalues $(\la_n)$ of $\Lutilde$ are all simple. More precisely,
\be\label{simplicite val prop Lu tilde}
    \la_{n+1}\geq \la_n+1\,, \qquad n\in \Nzero\,.
\ee
\underline{Focusing case.}
% In addition, t
The eigenvalues $(\nu_n)$ of $L_u$ are of multiplicity at most two
\be\label{gap-multiplicity2}
    \nu_{n+2} \geq \nu_n +1\, \qquad n\in \Nzero\,.
\ee 
Moreover,  when $n$ is large enough, the eigenvalues of $\Lu$ are simple. More precisely,
\be\label{liminf}
    \liminf_{n\to\infty}\ \nu_{n+1}-\nu_{n}\geq 1\,.
\ee 
Furthermore, for all $0\leq\alpha<1$ such that $\|u\|_{L^2}^2<1-\alpha\,,$ we have for all $n\in\Nzero\,,$
\be\label{simplicite u L2 <1}
    \nu_{n+1} > \nu_n +\alpha\, \,.
\ee
\end{prop}

\begin{Rq}\hfill
\label{val propre de multiplicite au plus 2}
    \begin{enumerate}
        \item It should be noted that for any potential $u$, the eigenvalues $(\nu_n)$ of $L_u$ cannot be all simple. For instance, take $u(x)=\eee^{ix}$\,, one can easily check that for $L_u=D-T_uT_{\overline{u}}\,,$
        $$L_u 1=L_u e^{ix} =0\,.$$
        \item  Inequality \eqref{liminf} implies that as $n>\!>0$, the lower bound of the distance between two consecutive eigenvalues $\nu_n$ gets closer to $1\,$.
    \end{enumerate}
\end{Rq}

\begin{proof} All the presented inequalities are a direct consequence of the max--min principle
\begin{gather*}
    \la_n= \max_{\underset{\dim F\leq n}{F\subseteq L^2_+}}\,
    \min\lracc{\ps{\tilde{L}_uh}{h}\,;\,h\in  F^{\perp}\cap H^\frac12_+(\T)\,,\ \|h\|_{L^2}=1 }\,.
\\
    \nu_n= \max_{\underset{\dim F\leq n}{F\subseteq L^2_+}}\,
    \min\lracc{\ps{L_uh}{h}\,;\,h\in  F^{\perp}\cap H^\frac12_+(\T)\,,\ \|h\|_{L^2}=1 }\,.
\end{gather*}
\textbf{Spectrum of $\Lutilde\,$.}
Let $F$ be any subspace of $\Ltwo$ of dimension $n\,,$ and consider $E:=\mathbb{C} 1 \oplus S(F)\,,$ where $S$ is the shift operator, then  
$$
    \la_{n+1}\geq \min\{\langle \tilde{L}_{u}h\mid h\rangle\ ;\|h\|_{L^2}=1, \, h\in E ^{\bot }\cap H^\frac12_+\}
$$
Observe that $E^{\perp}=S\left(F^{\perp}\right)\,$, thus by \eqref{commutateur [Lu,s]}\,,
\begin{align*}
    \la_{n+1}\geq &\,
    \min\left\{\langle \Lutilde g \mid g \rangle+1+\va{\ps{S g}{u}}^2\ ;\ \|g\|_{L^2}=1,\ g\in F^{\bot}\cap H^\frac12_+\right\}\,.
    % \\
    % \geq &\,
    %  \min\left\{\langle \Lutilde g \mid g \rangle+1\ ;\ \|g\|_{L^2}=1,\ g\in F^{\bot}\cap H^\frac12_+\right\}
\end{align*}
In addition, since $\lvert\ps{S g}{u}\rvert^2\geq0$\,, we infer for all $n\in\Nzero\,,$
$$
    \la_{n+1} \geq \la_n +1\,.
$$
\vskip0.25cm
\noindent
\textbf{Spectrum of $\Lu\,$--Inequality~\eqref{gap-multiplicity2}.}
let $F$ be any subspace of $L^2_+(\T)$ of dimension $n$, and take $G:=\mathbb{C} 1 \oplus S (F)+~\C u$\,.
Then, 
$$
    \nu_{n+2}(u)\geq  \min\{\langle L_{u}h\mid h\rangle\ ;\|h\|_{L^2}=1, \ h\in G ^{\bot }\cap H^\frac12_+\}\,.
$$
Since $G^{\perp} =S\left(F^{\perp}\cap (S^*u)^{\perp}\right)$, then
$$
    \nu_{n+2}\geq \min\left\{\langle L_u Sg \mid Sg \rangle\ ;\ \|g\|_{L^2}=1,\ g\in F^{\bot}\cap (S ^*u)^{\bot}\cap H^\frac12_+(\T)\right\}\,.
$$
Note that $g\perp S^*u$, then by \eqref{commutateur [Lu,s]}, 
\begin{align*}
    \nu_{n+2}\geq &\, 
    % \min\left\{\langle \Sh L_u h \mid \Sh h \rangle+\ps{\Sh h}{\Sh h};\ \|h\|_{L^2}=1,\ h\in F^{\bot}\cap H^1_+(\T)\cap (\Sh ^*u)^{\bot}\right\}
    % \\
    % = &\, 
    \min\left\{\langle  L_u g \mid  g \rangle+1;\ \|g\|_{L^2}=1,\ g\in F^{\bot}\cap (S ^*u)^{\bot}\cap H^\frac12_+(\T)\right\}\,,
\end{align*}
leading to 
\[
    \nu_{n+2} \geq \nu_n +1\,.
\]
\vskip0.25cm
\noindent
\textbf{Inequality~\eqref{liminf}\,.} For any $n\,,$  let
$
    F_n=\operatorname{span}\{f_0,f_1,\ldots,f_{n-1}\}
$
be the subspace of $\Ltwo$ of dimension $n$ made up of the first $n$ eigenfunctions of $L_u\,.$ For this choice of $F_n$\,, 
\be\label{ln=min}
    \min\left\{\langle L_u h \mid h \rangle\ ;\ \|h\|_{L^2}=1,\ h\in F^{\bot}_n\cap H^\frac12_+\right\}=\nu_n\,.
\ee
Let us consider the subspace  $E:=\mathbb{C} 1 \oplus S(F_n)$ of $L^2_+(\T)$  of dimension  $n+1$\,, then
\[
    \nu_{n+1}\geq\,  \min\{\langle L_{u}g\mid g\rangle\ ;\|g\|_{L^2}=1\,, \; g\in E ^{\bot }\cap H^\frac12_+\}\,.
\]
Note that
$E^{\perp} \cap H_{+}^{\frac12}=S\left(F^{\perp}_n \cap H_{+}^{\frac12}\right)$\,. Therefore, by \eqref{commutateur [Lu,s]}\,,
\begin{align}\label{ln+1}
    \nu_{n+1}\geq\, \min\left\{\langle L_u \phi \mid \phi \rangle+1-\va{\ps{S \varphi}{u}}^2\ ;\ \|\varphi\|_{L^2}=1,\ \varphi\in F^{\bot}_n\cap H^1_+\right\}.\notag
\end{align}
It results, for all $n\in \Nzero\,,$
\be\label{nu[n+1]-nu[n]}
    \nu_{n+1}\geq\,\nu_n\,+\,1\,-\,\sup _{\underset{\varphi\in F^{\bot}_n}{\Ldeux{\varphi} =1}}\left| \langle S \varphi\right| u\rangle |^2\ .
\ee
To conclude the proof, it remains to prove $\sup _{\underset{\varphi\in F^{\bot}_n}{\Ldeux{\varphi} =1}}\left| \langle S \varphi\right| u\rangle |^2\underset{n\to \infty}{\longrightarrow}0\,.$
\begin{lemma} Let $F_n$ be the 
be the subspace of $\Ltwo$ defined as above.  Then
$$
 \sup _{\underset{\varphi\in F^{\bot}_n}{\Ldeux{\varphi} =1}}
 \lvert \langle S\varphi \mid u\rangle \rvert \longrightarrow 0\ \ \text{as}\ n\to\infty\ .
$$
\end{lemma}
\begin{proof}
    Suppose for the sake of contradiction, that for all $n\in\Nzero\,,$
    $$
        \sup _{\underset{\varphi\in F^{\bot}_{n}}{\Ldeux{\varphi} =1}}\va{ \ps{S\varphi}{u}} \geq\eps\,, \qquad \eps>0\,.
    $$
    Namely, there exists $\varphi_n\in F^{\bot}_{n}$\,,  $\ \Ldeux{\varphi_n}=1$ such that 
    $
        \left| \langle S \varphi_n\right| u\rangle | \geq\frac\eps2\ .
    $
    Hence, since $\Ldeux{\varphi_n}=1\,,$  then up to a  sub-sequence $\varphi_{n}\rightharpoonup \varphi$ in $\Ltwo$ as $n\to \infty$\,, which yields to
    $$
        \va{\ps{S\varphi_{n}}{u}}\underset{n\to\infty}{\longrightarrow}\va{\ps{S\varphi}{u}}\,,
    $$
    and so $\ps{S\varphi}{u}\geq\frac\eps2$ . On the other hand, since $\varphi_{n}\perp F_{n}$ then 
    $$ 
        \ps{\varphi_{n}}{f_p}=0\,, \qquad \qlq \,0\leq p\leq n-1\,. 
    $$ 
    Taking $n\to \infty$\,, we infer
    $$  
        \ps{\varphi}{f_p}=0\ , \qquad \qlq p\in\Nzero\,. 
    $$
    Note that the eigenfunctions $(f_p)$ of the self--adjoint operator $\Lu$ form an orthonormal basis of $\Ltwo\,.$ 
    Therefore, we have $\varphi=0$\,, which is a contradiction with $\ps{S\varphi}{u}\geq~\frac\eps2\,$. 
\end{proof}
\vskip0.25cm
\noindent
\textbf{Inequality}~\eqref{simplicite u L2 <1}\,\textbf{.} 
It is a consequence of inequality~\eqref{nu[n+1]-nu[n]} after applying the Cauchy--Schwarz inequality and considering the fact that $\|u\|_{L^2}^2<1-\alpha\,.$
 
\end{proof}

% \textit{In the following, we make a slight abuse of notation and refer to $(f_n)$ as both the orthonormal basis of $\Ltwo$ consisting of the eigenfunctions of the self-adjoint operator $\Lu\,$, 
% % and as the orthonormal basis of $\Ltwo$ consisting of the eigenfunctions of 
% or the ones of $\Lutilde\,$.
% However, we will clarify the context in which we are working to ensure that the $(f_n)$ are understood appropriately.}
% % as either it refers to the orthonormal basis of $\Ltwo$ made up by the eigenfunctions of $L_u$ or $\Lutilde$.
\textit{In what follows, we make a slight abuse of notation by using $(f_n)$ to denote both an orthonormal basis of $\Ltwo$ consisting of the eigenfunctions of the self-adjoint operator $\Lu$, and an orthonormal basis of $\Ltwo$ consisting of the eigenfunctions of $\Lutilde$. Nonetheless, we will specify the context in which we are working to avoid confusion and ensure that $(f_n)$ is understood appropriately as either the eigenfunctions of $L_u$ or $\Lutilde$.}

\begin{lemma}\label{relation between <Sf|f> <u|f> and <1|f>}
Given $u\in H^2_+(\T)$\,, then for all $n,p\in\Nzero\,,$
\begin{itemize}
    \item \textbf{Defocusing case.} 
        \begin{gather*}
            \ps{1}{u}\langle u\mid f_n\rangle
            =
            \la_n\,\langle 1\mid f_n\rangle\,,
            \\
            (\la_n-\la_p-1)\,
            \langle S f_p\mid f_n\rangle 
            = 
            \langle S f_p\mid u\rangle\,
            \langle u\mid f_n\rangle\,.
        \end{gather*}
    \item \textbf{Focusing case.} 
        \begin{gather*}
        \ps{1}{u}\ps{u}{f_n}=-\nu_n\ps{1}{f_n}\,,
        \\
        (\nu_n-\nu_p-1)\ps{S f_p}{f_n}=-\ps{S f_p}{u}\ps{u}{f_n}\,.
\end{gather*}
\end{itemize}
\end{lemma}

\begin{proof}
We prove first the identities for the \textit{defocusing case}. 
% The same proof goes on for the focusing case.
By definition of $\tilde{L}_u=D+T_uT_{\bar u}$\,, we have
$$
    \tilde{L}_u1=\ps{1}{u}u\,.
$$
Then taking the inner product of both sides  with $f_n\,,$ and using the fact that $\tilde{L}_u$ is a self--adjoint operator, lead to the first identity.
    For the second one, thanks to the commutator relation between $\tilde{L}_u$ and $S$ 
    \[
         \tilde{L}_uSf_p =S \tilde{L}_uf_p+Sf_p+\ps{Sf_p}{u}u\,,
    \]
    of equation~\eqref{commutateur [Lu,s]}, we infer by taking  the inner product with $f_n$ the second identity.
    \vskip0.25cm
    Besides, by considering the \textit{focusing case} with $L_u=D-T_uT_{\bar u}$ it follows that $L_u1=-\ps{1}{u}u\,.$ This explains the sign $-$  appearing in the first statement. As for the second one, since by \eqref{commutateur [Lu,s]}
    \[
         L_uSf_p =S L_uf_p+Sf_p-\ps{Sf_p}{u}u\,.
    \]
    than taking once more the inner product with $f_n$ leads to the desired identity.
\end{proof}

\vskip0.25cm
In light of the previous lemma and based on the commutator identities \eqref{commutateur [Lu,s]}, one can investigate further information regarding  the spectral data (i.e. the eigenvalues and the eigenvectors) of $L_u$ and $\Lutilde\,,$ especially when the quantities $\ps{u}{f_n}$ vanishes.  The following lemmas/propositions aim to achieve this.
\vskip0.25cm
For the following, we denote by $E_{\nu_n}$  the eigenspace of  $\Lu$ corresponding to the eigenvalue $\nu_n$\,. In addition, the notation $f\parallelsum g$ means that the two vectors $f$ and $g$ are collinear in $\Ltwo\,$.
\begin{prop}\label{gap=0 and eigenspaces}
    For all $n\in\N\,,$ such that $\nu_{n}\neq 0\,,$ we have
    $$  
        \nu_{n}=\nu_{n-1}+1\implies
        [S f_{n-1}\in E_{\nu_{n}}] \text{ or }\, [f_{n}\in S E_{\nu_{n-1}}]\:,
    $$
     Moreover, for the defocusing case,
    $$  
        \la_{n}=\la_{n-1}+1\implies
        S f_{n-1}\parallelsum f_n  \:.
    $$
\end{prop}

\begin{Rq}\hfill
    \begin{enumerate}
        \item The condition $\nu_{n}\neq 0$ cannot be omitted. For an example, we refer to the Appendix~\ref{Appendix 1}\,.
        \item For the defocusing case, the condition of non--vanishing eigenvalues $\la_n\neq 0$ is already satisfied for all $n\in\N\,,$ since  $\Lutilde$ is a non--negative operator on the Hardy space, and for all $n\in\N$\,, we have by \eqref{simplicite val prop Lu tilde}\,, $\la_{n}\geq \la_{n-1}+1\,.$
    \end{enumerate}
\end{Rq}

\begin{proof}
The key is to use Lemma~\ref{relation between <Sf|f> <u|f> and <1|f>} and the commutator identities~\eqref{commutateur [Lu,s]}\,.
    In view of the second identity of Lemma~\ref{relation between <Sf|f> <u|f> and <1|f>}, we have  
\be
    \ps{S f_{n-1}}{u}\ps{u}{f_{n}}=0\,.
\ee 
%Hence, either $u\perp S f_{n-1}$ or $u\perp f_{n}$\,.
If $\ps{u}{S f_{n-1}}=0\,,$ then by \eqref{commutateur [Lu,s]},
\begin{align*}
    L_u Sf_{n-1}
     =&\,S L_u f_{n-1}+S f_{n-1}
     \\
    =&\,(\nu_{n-1}+1)\, Sf_{n-1}
    \\
    =&\,\nu_{n}\, S f_{n-1}\,,
\end{align*}
as $\nu_{n}=\nu_{n-1}+1$\,.
Namely, $Sf_{n-1}\in E_{\nu_{n}}$ .
Let us move to the second case where $\ps{u}{f_{n}}=0\,.$  By the first identity of Lemma~\ref{relation between <Sf|f> <u|f> and <1|f>}\,,
$$
    \nu_{n}\ps{f_{n}}{1}=0\,.
$$ 
Therefore, since $\nu_{n}\neq 0\,,$   there exists $g_{n}\in H^1_+(\T)$ such that $f_{n}=S g_{n}$\,.  Using again  the commutator identity \eqref{commutateur [Lu,s]}, we have
$$
    S L_ug_{n}=(\nu_{n}-1)S g_{n}\,.
$$
Applying $S^*$ to both sides of the latter identity, and  using the fact that $S^*S=\Id\,$,  and as $\nu_{n}=\nu_{n-1}+1\,$,  we find  
$$
    L_ug_{n}=\nu_{n-1}\,g_{n}\,.
$$
That is, $g_{n}\in E_{\nu_{n-1}}$\;, and so $f_{n}\in S E_{\nu_{n-1}}$ .
\vskip0.25cm
Besides, note that for the \textit{defocusing equation}, the vector spaces $E_{\la_n}$ are of dimension one, thanks to Proposition~\ref{multiplicite val propres}.  Consequently, the results $[S f_{n-1}\in E_{\la_{n}}]$ or $[f_{n}\in S E_{\la_{n-1}}]\,,$ leads to $S f_{n-1}\parallelsum f_n  \,.$

\end{proof}
\vskip0.25cm

In the sequel, we denote by $\mathcal{I}(u)$ the set of
\be\label{I}
    \I(u):=\lracc{n\in\N\;|\;\ps{Sf_{n-1}}{f_n}=0}\,.
\ee

\begin{lemma}\label{I=vide}
\textit{Defocusing case.} For any $u\in \Htwo\,$, the set $\mathcal{I}(u)$ is empty.

\vskip0.25cm
\noindent
\textit{Focusing case.}
Given $u\in H^2_+(\T)\,,$ 
    let $m\in\I(u)\,.$ Assume that  the eigenvalues $\nu_m$ and $\nu_{m-1}$ are simple, Then, either 
    \[
        \nu_{m-1}+1=\nu_{m+1}\,, \qquad \text{with} \qquad Sf_{m-1}\in E_{\nu_{m+1}}
    \]
    or
    \[
        \nu_{m-2}+1=\nu_{m}\,, \qquad \text{with} \qquad S^*f_{m}\in E_{\nu_{m-2}}
    \]
    or
    \[
        \nu_{m}=0\,, \qquad \text{with} \qquad f_{m}\parallelsum 1\,.
    \]
    % Furthermore, 
\end{lemma}

\begin{Rq}\hfill \label{Rq I=vide}
\begin{enumerate}[label=(\roman*),itemsep=6pt]
    \item  Observe that in the focusing case, if $\|u\|_{L^2}^2<\frac12$\,, then by inequality~\eqref{simplicite u L2 <1}\,, 
    \[
        \nu_{n}>\nu_{n-1}+\frac{1}{2}\,, \qquad \qlq n\in\N\,.
    \]
    Hence, for such $u\,,$ if $m\in \I(u)$ then the only possible choice is to have $\nu_m=0$ with $f_m\,\parallelsum\, 1\,.$ 
    In other words, 
    % by simplicity of the eigenvalues
    for $\|u\|_{L^2}^2<\frac12\,,$ we have, either $\I(u)=\varnothing\,,$ 
    or $\I(u)=\lracc{m}$ and in such case $\nu_m=0$ and $f_m\parallelsum1\,.$
    \item  For any $u\in \Htwo\,,$ the set $\I(u)$ in the focusing case is of finite cardinal, since by inequality \eqref{liminf} we have  $\nu_n>\nu_{n-1}+\frac{1}{2}$ and $\nu_n\neq 0$ for all $n$ large enough.
\end{enumerate}

\end{Rq}

\begin{proof}
    \textit{Focusing case}. Let $m\in \mathcal{I}(u)\,.$ By the 
    second identity of Lemma~\ref{relation between <Sf|f> <u|f> and <1|f>}\,, 
    \[
        \ps{Sf_{m-1}}{u} \ps{u}{f_m}=0\,.
    \]
    If $\ps{Sf_{m-1}}{u}=0\,,$ then applying the commutator identity \eqref{commutateur [Lu,s]}\,,
    \be\label{Sf[m-1] eigenfunction to Lu}
        \Lu S f_{m-1}=(\nu_{m-1}+1)Sf_{m-1}\,.
    \ee
   Namely, $\nu_{m-1}+1$ is an eigenvalue of $L_u$ and $S f_{m-1}$ is a corresponding eigenfunction.
   Since $\nu_m$ is simple, then $Sf_{m-1}$ cannot be collinear to $f_m$ as $\ps{Sf_{m-1}}{f_m}=0$ for $m\in\I(u)\,.$ Therefore, by \eqref{gap-multiplicity2}\,, 
   $$
        \nu_{m-1}+1=\nu_{m+1}\,.
    $$
   \vskip0.25cm
   \noindent
   If $\ps{u}{f_m}=0\,,$ then by applying the adjoint of the commutator identity  \eqref{commutateur [Lu,s]}\,, 
    \[
        S^*{L}_{u}  = {L}_{u} S^*+S^* +\ps{\,\cdot}{u}S^*u\,,
    \]
     we infer, 
    \[
        {L}_{u}\, S^*f_m=(\nu_m-1)S^*f_m\,.
     \]
    That is, if $S^*f_m\neq 0\,,$ then $S^*f_m$ is an eigenfunction of $L_u$ associated with the eigenvalue $\nu_m-1\,.$
    Recall, we have by assumption that $\nu_{m-1}$ is simple, and since $S^*f_m$ cannot be collinear to $f_{m-1}$ as $m\in\I(u)$\,,  then  
    $$
        \nu_m-1=\nu_{m-2}\,,
    $$ 
    thanks to \eqref{gap-multiplicity2}\,.
    It remains to study the case where $S^*f_m=0\,,$ i.e., $f_m\parallelsum 1\,.$ 
    For that case, we have thanks to  the first identity of Lemma~\ref{relation between <Sf|f> <u|f> and <1|f>}\,,  $\nu_m=0$ as 
    % Since 
    $\ps{u}{f_m}=0\,.$ 
    % then $\ps{u}{1}=0\,.$ Therefore, by the first identity of Lemma~\ref{relation between <Sf|f> <u|f> and <1|f>}\,, we infer $\nu_m=0\,.$
    \vskip0.25cm
    \noindent
    \textit{The defocusing case.} Suppose that there exists $m\in\I(u)$\,. Then, using the same analysis as in the focusing case, we infer that, either $\la_{m-1}+1=\la_{m+1}$ or $\la_{m-2}+1=\la_{m}$ or $\la_m=0\,.$ 
    However, recall that $\la_n\geq \la_{n-1}+1$ for all $n\in\N$ (inequality~\ref{simplicite val prop Lu tilde})\,, thus the first two cases cannot occur. 
    In addition, since $\Lutilde$ is a non negative operator, where all the eigenvalues satisfy the inequality~\eqref{simplicite val prop Lu tilde}\,, then $\la_m=0$ implies $m=0\notin\I(u)\,.$
    
\end{proof}

\begin{corollary}\label{<u|f[n]> <=> gamma[n]=0}
    % \begin{gather*}
    %    \nu_{n}=\nu_{n-1}+1\implies \ps{u}{f_n}=0\,.
    % \end{gather*}
    % In the defocusing case, we have f
    For all $n\geq 1\,,$
    % the following equivalence
    $$
        \la_{n}=\la_{n-1}+1\iff \ps{u}{f_n}=0\,.
    $$
    % In particular, 
    In addition, 
    \[
        \nu_{n}=\nu_{n-1}+1\,, \ \qlq n\geq N_1 \quad\iff\quad \ps{u}{f_n}=0\,, \ \qlq n\geq N_2\,.
    \]
\end{corollary}
\begin{Rq}
      We refer to the Appendix~\ref{Appendix 2} for an example that shows that $N_2$ is not necessarily equal to $N_1$\,.
\end{Rq}

\begin{proof}
For the \textit{defocusing case}.
%The proof is based in the same arguments used to establish the previous lemma.
    Suppose that $\la_{n}=\la_{n-1}+1\,.$  Then, from one side we have by Proposition~\ref{gap=0 and eigenspaces}\,, $Sf_{n-1}\parallelsum f_n\,,$  and from the other hand, we infer by the second identity of  Lemma~\ref{relation between <Sf|f> <u|f> and <1|f>}\,,
$$
    \ps{S f_{n-1}}{u}\ps{u}{f_{n}}=0\,.
$$
That is, $\ps{u}{f_{n}}=0\,.$
The converse is a direct consequence of the second identity of Lemma~\ref{relation between <Sf|f> <u|f> and <1|f>} and the previous lemma. 

\vskip0.25cm
For the \textit{focusing case,} the same analysis can be applied. However, it should be noted that, since not all the eigenvalues $(\nu_n)$ satisfy $\nu_{n}>\nu_{n-1}+\frac12\,,$  and $\nu_n\neq 0\,,$ for all $n\in\N\,$,
but only for large $n\,,$ thanks to Proposition~\ref{multiplicite val propres}\,,
then the equivalence holds for $n$
sufficiently large. 

\end{proof}

%-----------------------------------------------------
%-----------------------------------------------------
\section{\textbf{Traveling waves for the defocusing (CS\texorpdfstring{$^-$}{-})}}
\label{traveling waves for (CS-)}

\subsection{Spectral Characterization}
\label{spectral characterization of traveling waves of (CS-)}
One way to understand the behavior of a linear PDE's solution is to consider its Fourier transform. 
 Specifically, on the periodic domain $\T$, this consists of computing the inner product with $\ps{\,\cdot}{\eee^{inx}}$ for all $n\in\Z$.
The main idea behind this approach is to ``diagonalize'' the problem in the $(\eee^{inx})$--basis, which facilitates solving the equation. 
However, by considering the Calogero--Sutherland DNLS equation~\eqref{CS-defocusing}\,, we are dealing with a nonlinear \textit{integrable} PDE, which can also 
% The integrability fact about this equation implies that \eqref{CS-defocusing} can also 
be ``diagonalized'' in some coordinate system (think about the Birkhoff coordinates).
Thus, by imitating the idea of the linear case, we suggest taking the inner product of the \eqref{CS-defocusing}--equation with an appropriate orthonormal basis of $\Ltwo\,.$ 
% which allows to diagonalize the problem.
Before proceeding, observe that the defocusing  Calogero--Sutherland DNLS equation can be rewritten in terms of the Lax pair as \cite[Lemma~2.4, Lemma 5.2]{Ba23}
\be\label{rewriting CS-def}
    \p_t u=\tilde{B}_u u -i\tilde{L}^2_uu\,.
\ee 
This motivates the choice of the following orthonormal basis of $\Ltwo\,$.

\begin{defi}
		\label{orthonormal Basis gnt}
    Given $u\in \mathcal C_{\,t}\Htwo_x\,,$ let $(g_n^{\,t})$ be the evolving orthonormal basis of $L^2_+(\T)$ defined along the curve $t\mapsto u(t)$ as
		\begin{align}\label{Cauchy pb}
		\begin{cases}
			\p_{t} \,\gnt&=\,\tilde{B}_{u(t)}\,\gnt 
			\\
			{\gnt\,}_{|_{t=0}}&=\,\fnuzero
		\end{cases}\,,\qquad \qlq n\in\Nzero\,,
		\end{align}
		 where $(\fnuzero)$ is an orthonormal basis of $\Ltwo$ made up of the eigenfunctions of  $\tilde{L}_{u_0}$ at $t=0\,,$ and $\tilde{B}_{u(t)}$ is the skew--adjoint operator defined in \eqref{Lax operators defoc case}\,.
\end{defi}

\begin{Rq}\label{Rq gnt vect. propre+ evolution<u|fn>}
    Note that the $(\gnt)$ satisfies for all $n\in\Nzero$ \cite[Lemma 4.1]{Ku06}\,,
    \be\label{gnt vect. propre}
        \tilde{L}_{u(t)}\, \gnt =\la_n\, \gnt\,.
    \ee
    Therefore, as it was established in~\cite[Lemma 3.6]{Ba23}\,, by taking the inner product of \eqref{rewriting CS-def} with the $\gnt$ and using that $\Lutilde$ is a self--adjoint operator and $\Butilde$ is skew--symmetric, we find
    $$
        \p_t\ps{u(t)}{\gnt}=-i\la_n^2 \ps{u(t)}{\gnt}\,,
    $$
    or
    \be\label{evolution <u|gn>}
        \ps{u(t)}{\gnt}=\ps{u_0}{\fnuzero}\,\eee^{-i\la_n^2 \,t}\,.
    \ee
\end{Rq}

\begin{lemma}\label{evolution in the gnt}
    For any $ u\in \mathcal C_{\,t}\Htwo_x$ solution of \eqref{CS-defocusing} and for all $n\,,p\in\Nzero\,,$
    \begin{align}\label{evolution <1|fn>}
        \ps{1}{\gnt}=&\,\ps{1}{\fnuzero}\,\eee^{-i\la_n^2t}\ ,
        \\
        \ps{S\gpt}{\gnt}=&\,\ps{S\fpuzero }{\fnuzero}\,\eee^{i((\la_p+1)^2-\la_n^2)\,t}\,.
        % \ps{S\gpt}{u}=&\,\ps{Sf_p^{\, u_0}}{u_0}\eee^{i(\la_p+1)^2\,t}\,.
        \notag
    \end{align}
    % Furthermore, 
    % $$
    %     \ps{S\gpt}{\gnt}=\ps{S\fpuzero }{\fnuzero}\,\eee^{i((\la_p+1)^2-\la_n^2)\,t}\,.
    % $$
\end{lemma}

\begin{proof}
    By Definition~\ref{orthonormal Basis gnt}\,, and since $\Butilde$ is skew-symmetric operator,
    $$
    \p_t \ps{1}{\gnt}=\ps{1}{\Butilde\gnt}=-\ps{\Butilde1}{\gnt}\,,
    $$
    where by \eqref{Lax operators defoc case}\,,
    \begin{align*}
        \tilde{B}_u1
        =&\,-T_uT_{\partial_x\bar u}1+T_{\partial_xu}T_{\bar u}1 +i(T_uT_{\bar u})^21
        \\
        =&\,\ps{1}{u}\big(\p_xu+iT_uT_{\overline u}\,u\big)\,.
    \end{align*}
    Note that $\Lutilde1=-i\p_x1 + T_uT_{\bar u}1=\ps{1}{u}u\,.$  Therefore, $\tilde{B}_u1=i\Lutilde^21$ and 
    $$
        \p_t\ps{1}{\gnt}=-i\ps{L^2_u1}{\gnt}=-i\la_n^2\ps{1}{\gnt}\,.
    $$
    This achieves the proof of the first point.
    To prove the second one, we proceed with the same manner. By Definition~\ref{orthonormal Basis gnt}\,,
    $$
        \p_t\ps{\gnt}{S\gpt}=\ps{\tilde{B}_u\gnt}{S\gpt}+\ps{\gnt}{S\tilde{B}_u \gpt}=\ps{[S^*, \tilde{B}_u]\,\gnt}{\,\gpt}.
    $$ 
    Hence, applying the commutator identity~\eqref{commutateur [Bu,s]}\,, and since $\Lutilde$ is a self--adjoint operator, we infer
    \begin{align*}
        \p_t\ps{\gnt}{S \gpt}=&\,i\,\big\langle \big( S^*\Lutilde^2 \,-\, (\Lutilde+\Id)^2S^* \big) \gnt\;\big|\;\gpt\big\rangle
        \\
    =&\,i\,(\la_n^2+(\la_p+1)^2)\ps{\gnt}{S\gpt}\,.
    \end{align*}
    Therefore, 
    $$
        \ps{\gnt}{S \gpt}=\ps{\fnuzero}{S \fpuzero}\,\eee^{i(\la_n^2-(\la_p+1)^2)t}\,.
    $$
    % \vskip3cm
    % \begin{align*}
    %     \p_t\ps{S\gpt}{u}
    %     =&\,\ps{S\Butilde \gpt}{u}+\ps{S\gpt}{\Butilde u -i\Lutilde^2u}
    %     \\
    %     =&\,\ps{[S,\Butilde]\,\gpt}{u}+i\ps{S\gpt}{\Lutilde^2u}\,,
    % \end{align*}
    % where $[\cdot,\cdot]$ denotes the commutator. We recall from \cite[Lemma 2.2]{Ba23}\footnote{~In Lemma 2.2 of \cite{Ba23}\,, we have proved a commutator identity between $B_u$ and $S^*$. The same proof works for the defocusing case with $\Butilde\,.$}\,,
    % $$
    %     [S^*, \Butilde]=i\Big( S^*\Lutilde^2 \,-\, (\Lutilde+\Id)^2S^* \Big)\,,
    % $$
    % where $S^*$ is the adjoint operator of the shift operator $S$\,.
    % As a consequence,
    % $$
    %     \p_t\ps{S\gpt}{u}=i(\la_p+1)^2\ps{S \gpt}{u}\,,
    % $$
    % thanks to \eqref{gnt vect. propre}. In the same way we prove the last point.
\end{proof}

\begin{Rq}\label{formule explicite en terme matricielle}
    The consideration of the evolution of $\ps{u}{\gnt}$, $\ps{1}{\gnt}$, and $\ps{S\gpt}{\gnt}$ is motivated by the fact that any element $u$ of the Hardy space can be written as follows.
     \begin{lemma}[\cite{GK21,GMR16}] \label{formule inversion spectrale}
     For any $u\in\Ltwo\,,$
        \bes
		u(z)=\ps{(\Id-zS^*)^{-1}\,u}{1}\,, \qquad z\in\D\,,
    \ees
    where $S^*$ is the adjoint operator of $S$ in $\Ltwo\,.$
    \end{lemma}
    Therefore, by expressing the operator
    $S^*\,$,  and the two vectors $u$ and $1$  in their matrix representations  with respect to  the $(\gnt)$--basis, we obtain
    \be
      u(t,z)=\ps{(\Id-z M)^{-1}\,X}{Y}\,, \qquad z\in\D\,,
\ee
where   $X\,,$ $Y $ are infinite column vectors and $M$ is the infinite matrix representation~:
$$
    X:=\big(\ps{u}{\gnt}\big)\,,
    \quad  
    Y:=\big(\left\langle 1 \mid \gnt\right\rangle\big)\,,
    \quad M:=\Big(\left\langle \gnt\mid S \gnt\right\rangle\Big)\,.
$$ 
\begin{proof}[Proof of Lemma~\ref{formule inversion spectrale} ] (\cite{GK21})
     The idea is to observe that any element $u$ of the Hardy space $\Ltwo$ can be read as an analytic function on the open unit disc $\D$\,, whose trace on the boundary $\partial\D$ is in $L^2\;$~\footnote{For a simple introduction to the different definitions of Hardy space, we refer to \cite[Chapter 3.]{GMR16}}. Thus, for any $z\in\D\,,$
    \[
        u(z)=\sum_{k\in\Nzero}\widehat{u}(k) z^k=\sum_{k\in\Nzero} \ps{u}{S^k1}z^k=\sum_{k\in\Nzero}\ps{(S^*)^ku}{
        1}z^k\,.
    \]
    As a result, by Neumann series,
    \[
        u(z)=\left\langle(\Id-z S^*)^{-1} u \mid 1\right\rangle.
    \]
\end{proof}
% and $\mathcal{D}_t$ is the diagonal infinite matrix  given by $\mathcal{D}_t=\mrm{diag}(\eee^{(-2\la_n-1)it}\,)\,.$
%     % these quantities seem to encode the dynamics of the Calogero-Sutherland DNLS equation~\eqref{CS-defocusing}.
  %  \vskip2cm
  %   This can also be seen via the explicit formula
  %   Indeed, as derived in \cite[Proposition~2.5, Theorem 4]{Ba23}\,, 
  %   the  solution $u$ for the Cauchy problem of \eqref{CS-defocusing} with  initial value $u_0\,,$ is expressed \textbf{explicitly} via the explicit formula
  %   % \footnote{Observe that the solution $u$ is defined here
  %   \be\label{formule explicite}
		% u(t,z)=\ps{(\Id-z\eee^{-it(2\tilde{L}_{u_0}+\Id)}S^*)^{-1}\,u_0}{1}\,, \qquad z\in\D\,.
  %   \ee
\end{Rq}

\vskip0.25cm
 At this stage, we consider $u(t):=u_0(x-ct)$ to be a traveling wave to the Calogero--Sutherland DNLS equation \eqref{CS-defocusing}.
 % with $c\in\R$ and $\varphi(t)$ is a real--valued function. 
 We denote, for all $c,t\in\R\,,$ by $\tau_{ct}$ the isometric linear map 
\[
    \tau_{ct}: \Ltwo\to\Ltwo\,,
    \qquad \tau_{ct}\, u_0(x)=u_0(x-ct)\,.
\]
Our aim for this subsection is to prove the following Theorem.
\begin{theorem}\label{characterization u in the phase space <u|fn}
    Let $u(t):=\tau_{ct}\,u_0\,$ be a traveling wave to the  \eqref{CS-defocusing}--equation. Then there exists at most one $N\in\N$ such that 
    \bes
        % \left|
        % \begin{array}{l}
            \ps{u_0}{f_{N}^{\,u_0}}\neq0\,. 
        % \end{array}
        % \right.
    \ees  
    Moreover, the speed $c$ is given by
    \[
        c=1+\frac{2}{N}\sum_{k=0}^{N-1}\la_k\,.
    \]
\end{theorem}
\vskip0.25cm
To this end, we shall need two key elements. Firstly, Lemma~\ref{evolution in the gnt} and identity~\eqref{evolution <u|gn>}. Secondly, we will utilize the existence of a relationship (identity~\eqref{tauctfn-gnt}) connecting the eigenfunctions $(\gnt)$ of $\Luttilde$ introduced in Definition~\ref{orthonormal Basis gnt}\,, with the functions $(\tauct \fnuzero)$, where recall $(\fnuzero)$ represents the eigenfunctions of $\Luzerotilde$. 
% we try to  explore is there exists 
% the existence of a relation relating the eigenfunction $(\gnt)$ of $\Luttilde$ introduced in Definition~\ref{orthonormal Basis gnt} and the functions $(\tauct \fnuzero)$, where the $(\fnuzero)$ represents the eigenfunctions of $\Luzerotilde$. 
\vskip0.25cm
To establish this connection, we present the following proposition, which also describes the behavior of the eigenfunctions $(\fnuzero)$ of $\Luzerotilde$ under the action of the translation map on the spatial variable
$$
    \fnuzero\longmapsto \tau_{ct}\,\fnuzero\,, \qquad c,t\in\R\,.
$$

%%%%%%%%%%%%%%%%%%%%%%%%%%%%%%%%%%%%%%%%%%%%%%
% Of course, the eigenvalues $\big(\la_n(\tau_{ct}u_0)\big)$ of $\tilde{L}_{\tau_{ct}u_0}$ are invariant under $\tauct\,$ for all $c,t\in\R\,.$ Indeed, by definition of $\Lut\,,$
% \begin{align*}
%     \Lut
%     =&\,D-u(t)\,\Pi\big(\overline{u(t)}\, \cdot\big)
%     \\
%     =&\,D-\tau_{ct}u_0\,\Pi\big(\overline{\tau_{ct}u_0}\; \cdot\big)
%     \\
%     =&\,\tilde{L}_{\tau_{ct}u_0}\,,
% \end{align*}
% then $\la_n(u(t))=\la_n(\tauctuzero)$ for all $n$\,, implying thereby, in view of the isospectral property of the Lax operator $\Lu\,,$
% $$
%     \la_n(u_0)=\la_{n}(\tauctuzero)\,.
% $$
% for all $n\in\Nzero\,.$
%%%%%%%%%%%%%%%%%%%%%%%%%%%%%%%%%%%%%%%%%%%%%%%%%%%%

\begin{prop}\label{tauct fn}
%%%%%%%%%%%%%%%%%%%%%%%%%%%%%%%%%%%%%%%%%%%%%%%%%%%%%
    %  The eigenvalues $\big(\la_n(\tau_{ct}u_0)\big)$ of $\tilde{L}_{\tau_{ct}u_0}$ satisfy for all $n\in\Nzero\,,$
    %     $$
    %         \la_n(\tau_{ct}u_0)=\la_n(u_0)\,.
    %     $$
    % Moreover, for all $n\in \Nzero\,,$ there exists $\theta_n(t)$ such that for all $t\in\R\,,$ 
    % % the eigenfunctions $\big(f_n^{\,\tau_{ct}u_0}\big)$ of $\tilde{L}_{\tau_{ct}u_0}$ satisfy
    % $$
    % \tau_{ct}\fnuzero=\eee^{i\theta_n(t)}\gnt \,,
    % $$
%%%%%%%%%%%%%%%%%%%%%%%%%%%%%%%%%%%%%%%%%%%%%%%%%%%%%%%
    Let $u(t):=\tau_{ct}\,u_0\,$ be a solution to \eqref{CS-defocusing}\,.  
    % For all $n\in\Nzero\,,$ t
    There exists a sequence $(\theta_n(t))\subseteq \R\,,$ such that
    \be\label{tauctfn-gnt}
        \tau_{ct}\,\fnuzero=\eee^{i\theta_n(t)}\,\gnt \;, \qquad \qlq n\in\Nzero\,.
    \ee
    In other words, the $(\tauct\fnuzero)$ are also eigenfunctions of $\Luttilde\,.$
\end{prop}
% Recall from Remark~\ref{Rq gnt vect. propre+ evolution<u|fn>}\,, the $(\gnt)$ are also eigenfunctions for $\Luttilde\,.$ 

% \begin{Rq}
%     Recall from Remark~\ref{Rq gnt vect. propre+ evolution<u|fn>}\,, the $(\gnt)$ are eigenfunctions of the Lax operator $\Luttilde\,.$ Hence, in light of the previous proposition, if $(\fnuzero)$ are eigenfunctions of $\tilde{L}_{u_0}$ then $(\tauct\fnuzero)$ are eigenfunctions of $\Luttilde\,$, where $u(t):=\tauct u_0\,$. In addition, they are obtained from the orthonormal basis $(\gnt)$ --which satisfies the Cauchy problem~\eqref{Cauchy pb}-- by a phase multiplication $\ethetant\,.$ 
% \end{Rq}

\begin{proof}
    By definition of $\tilde{L}_{u}=D+u\,\Pi\big(\bar{u}\, \cdot\big)\,,$ and since $u(t)=\tauct u_0$
    \begin{align*}
        \tilde{L}_{u(t)}\;\tauct\, \fnuzero
        =&\, D\fnuzero(x-ct)\,+\,u_0(x-ct)\,\Pi\big(\bar{ u}_0(x-ct)\; \fnuzero(x-ct)\big)\,, 
        \\
        =&\,\tauct\left(\tilde{L}_{ u_0} \fnuzero\right)
        \\
        =&\, \la_n(u_0)\,\tauct \fnuzero\,, \hskip4.25cm \qlq n\in\Nzero\,.
    \end{align*}
    In other words, $\tauct \fnuzero$ is an eigenfunction of $\tilde{L}_{u(t)}$ associated with the eigenvalue $\la_n(u_0)\,.$ On the other hand, recall that all the eigenvalues
    $\la_n(u_0)$ of $\tilde{L}_{u(t)}$ are simple, as stated in Proposition~\ref{multiplicite val propres}\,.
    Additionally, according to Remark~\ref{Rq gnt vect. propre+ evolution<u|fn>}, the $(\gnt)$ are eigenfunctions of $\tilde{L}_{u(t)}$ associated to the eigenvalues $\la_n(u_0)\,$. Therefore, for all $n\in\Nzero\,$, the two vectors $\tauct\fnuzero$ and $\gnt$ are collinear. Since both vectors belongs to an orthonormal basis of $\Ltwo$, then each one has an $L^2$-norm equal to one. Thus, we infer for all $n\in\Nzero\,,$ there exists $\thetant\in\R$ such that for all $t\in\R\,,$
    $$
        \tau_{ct}\fnuzero=\eee^{i\theta_n(t)}\gnt\,.
    $$
\end{proof}

% To summarize, we have shown that for any orthonormal basis $(\fnuzero)$ of $\Ltwo$ consisting of the eigenfunctions of $\tilde{L}_{u_0}\,$, the $(\tauct\fnuzero)_{_n}$ are eigenfunctions of $\tilde{L}_{u(t)}$ and are obtained from the orthonormal basis $(\gnt)$ --which satisfies the Cauchy problem~\eqref{Cauchy pb}-- by a phase multiplication $\ethetant\,.$ 

% The following corollary, describes the behavior of $\theta_n(t)\,.$

\begin{corollary}\label{evolution theta[n]}
    For all $n\,,p\in\Nzero\,,$ and for all $t\in\R\,,$ we have
    \begin{enumerate}
        \item If $\langle 1\,|\fnuzero\rangle\neq 0$   then
            $$
                \theta_n(t)=-\la_n^2t\,.
            $$
        \item  If $\langle u_0\,|\fnuzero\rangle\neq 0$  then  
            $$
                \theta_n(t)=-\la_n^2t\,.
            $$
            % $$
            %     \theta_n(t)=-\la_n^2t-\varphi(t)\,.
            % $$
        % \item If $\langle Sf_p^{\, u_0}|u_0\rangle\neq 0$ then 
        %     $$
        %         \thetapt=ct-(\la_p+1)^2t-\varphi(t)\,.
        %     $$
        \item If $\langle S\fpuzero|\fnuzero\rangle\neq0$ then
            \[
                \theta_n(t)=((\la_p+1)^2-\la_n^2)t-ct+\thetapt\,,
            \]
    \end{enumerate}  
    where $\theta_n(t)$ is the angle obtain in \eqref{tauctfn-gnt}\,.
\end{corollary}

\begin{proof}
     By combining identity~\eqref{evolution <u|gn>} and the two identities of Lemma~\ref{evolution in the gnt}\,, with identity~\eqref{tauctfn-gnt} of  the previous proposition, we infer 
    %  \bes
    %     \left|
    %     \begin{array}{l}
    %         \ethetant\ps{1}{\tauct\fnuzero}=\ps{1}{\fnuzero}\,\eee^{-i\la_n^2t}\ ,
    %                 \\[0.2cm]
    %         \ethetant\eee^{i\varphi(t)}\ps{\tauct u_0}{\tauct\fnuzero}=\ps{u_0}{\fnuzero}\,\eee^{-i\la_n^2 \,t}
    %                 \\[0.2cm]
    %         \eee^{-i\thetapt}\eee^{-i\varphi(t)}\ps{S\tauct\fpuzero}{\tauct u_0}=\ps{Sf_p^{\, u_0}}{u_0}\eee^{i(\la_p+1)^2\,t}\,.
    %                 \\[0.2cm]
    %         \eee^{-i\thetapt}\eee^{i\thetant}\ps{S\tauct\fpuzero}{\tauct\fnuzero}=\ps{S\fpuzero }{\fnuzero}\,\eee^{i((\la_p+1)^2-\la_n^2)\,t}\,.
    %     \end{array}
    %     \right.
    % \ees
    \bes
        \left|
        \begin{array}{l}
            \ethetant\ps{1}{\tauct\fnuzero}=\ps{1}{\fnuzero}\,\eee^{-i\la_n^2t}\ ,
                    \\[0.2cm]
            \ethetant\ps{\tauct u_0}{\tauct\fnuzero}=\ps{u_0}{\fnuzero}\,\eee^{-i\la_n^2 \,t}
                    \\[0.2cm]
            % \eee^{-i\thetapt}\eee^{-i\varphi(t)}\ps{S\tauct\fpuzero}{\tauct u_0}=\ps{Sf_p^{\, u_0}}{u_0}\eee^{i(\la_p+1)^2\,t}\,.
            %         \\[0.2cm]
            \eee^{-i\thetapt}\eee^{i\thetant}\ps{S\tauct\fpuzero}{\tauct\fnuzero}=\ps{S\fpuzero }{\fnuzero}\,\eee^{i((\la_p+1)^2-\la_n^2)\,t}\,.
        \end{array}
        \right.
    \ees
    Note that, $S\,\tauct \;(\cdot)=\eee^{ict}\tauct(S\;\cdot)\,$, and since we are dealing with periodic functions\,, we deduce, 
     \be
    \label{eqt theta n}
        \left|
        \begin{array}{l}
            \ethetant\ps{1}{\fnuzero}=\ps{1}{\fnuzero}\,\eee^{-i\la_n^2t}\ ,
                    \\[0.2cm]
            \ethetant\eee^{i\varphi(t)}\ps{ u_0}{\fnuzero}=\ps{u_0}{\fnuzero}\,\eee^{-i\la_n^2 \,t}
                    \\[0.2cm]
            % \eee^{-i\thetapt}\eee^{-i\varphi(t)}\eee^{ict}
            % \ps{S\fpuzero}{ u_0}=\ps{Sf_p^{\, u_0}}{u_0}\eee^{i(\la_p+1)^2\,t}\,.
            %         \\[0.2cm]
            \eee^{-i\thetapt}\eee^{i\thetant}\eee^{ict}
            \ps{S\fpuzero}{\fnuzero}=\ps{S\fpuzero }{\fnuzero}\,\eee^{i((\la_p+1)^2-\la_n^2)\,t}\,.
        \end{array}
        \right.
    \ee
    leading to the result.
    
\end{proof}

At this point, we are ready to prove the spectral characterization of the traveling waves for \eqref{CS-defocusing}, namely Theorem~\ref{characterization u in the phase space <u|fn}\,.
% \begin{theorem*}[\ref{characterization u in the phase space <u|fn}]
%     Let $u(t):=\tau_{ct}\,u_0\,$ be a traveling waves to the  \eqref{CS-defocusing}--equation. Then there exists at most one $N\in\N$ such that 
%     \bes
%         % \left|
%         % \begin{array}{l}
%             \ps{u_0}{f_{N}^{\,u_0}}\neq0\,. 
%         % \end{array}
%         % \right.
%     \ees  
%     Moreover, the speed $c$ is given by
%     \[
%         c=1+\frac{2}{N}\sum_{k=0}^{N-1}\la_k\,.
%     \]
% \end{theorem*}

\begin{proof}[Proof of Theorem~\ref{characterization u in the phase space <u|fn}]
    The proof relies on the spectral property of $\Lutilde$ discussed in Section~\ref{spectral property of the Lax operator} and on Corollary~\ref{evolution theta[n]}\,. Indeed, observe first
    by  Lemma~\ref{I=vide}\,, we have  $\ps{Sf_{n-1}^{\, u_0}}{\fnuzero}\neq0$ for all $n\in\N\,.$ Hence, applying the third identity of Corollary \ref{evolution theta[n]} with $p=n-1$\,, leads to the recurrence relation 
    $$
        \theta_n(t)=\big((\la_{n-1}+1)^2-\la_n^2\big)t-ct+\thetanMoinsUnt\,, \qquad  n\geq 1\,.
    $$
    Taking the sum of all these expressions from $n=1$ to $n\in\N$\,, we infer 
    \be\label{expression theta[n]}
        \theta_n(t)=\la_0^2\,t+2t\sum_{k=0}^{n-1}\,\la_k+nt-\la_n^2\,t-nct+\theta_0(t)\,.
    \ee
    Our aim is to prove that for all $n\geq 1$\,, $\ps{u_0}{f_{n}^{\,u_0}}=0$ unless at most for one $n$\,. For the sake of contradiction, suppose
     that there exist two integers $1\leq n_1<n_2$ such that $\ps{u_0}{f_{n_1}^{\,u_0}}\neq 0$ and $\ps{u_0}{f_{n_2}^{\,u_0}}\neq 0$\,. Then by Corollary~\ref{evolution theta[n]}, we infer 
     \begin{gather}\label{theta[n1], theta[n2]}
         \theta_{n_1}(t)=-\la_{n_1}^2t
         % -\varphi(t)
         \\
         \theta_{n_2}(t)=-\la_{n_2}^2t
         % -\varphi(t)
         \notag
     \end{gather}
     Plugging  \eqref{theta[n1], theta[n2]} in \eqref{expression theta[n]} we obtain
     % \begin{gather*}
     %     -\varphi(t)=-n_1ct+n_1t+\theta_0(t)+2t\sum_{k=0}^{n_1-1}\la_k+\la_0^2t\,.
     %     \\
     %     -\varphi(t)=-n_2ct+n_2t+\theta_0(t)+2t\sum_{k=0}^{n_2-1}\la_k+\la_0^2t\,.
     % \end{gather*}
     \begin{gather}\label{n[1]ct=}
         n_1ct=n_1t+2t\sum_{k=0}^{n_1-1}\la_k+\theta_0(t)+\la_0^2t\,.
         \\
         n_2\,ct=n_2t+2t\sum_{k=0}^{n_2-1}\la_k+\theta_0(t)+\la_0^2t\,.\notag
     \end{gather}
     Besides, notice that 
     \be\label{theta[0]}
        \theta_0(t)=-\lambda_0^2t\,.
    \ee
    Indeed, if $\ps{u_0}{f_0^{\,u_0}}\neq 0$ then by the second point of Corollary~\ref{evolution theta[n]}, we have the claimed identity.  Otherwise, if $\ps{u_0}{f_0^{\,u_0}}=0$ then $\ps{1}{f_0^{\,u_0}}\neq0$\,, 
    % by Lemma~\ref{relation between <Sf|f> <u|f> and <1|f>}\,,
     % \bes
     %    \la_0\ps{1}{f_0^{\, u_0}}=0\,.
     % \ees
     since if it is not the case, 
     % Thus, if $\ps{1}{f_0^{\, u_0}}=0\,,$ 
     i.e. if there exists $h\in\Ltwo$ such that $f_0^{\, u_0}=Sh\,,$ then we have by the commutator relation~\eqref{commutateur [Lu,s]}
     $$
        \la_0 Sh=\tilde{L}_{u_0} Sh =S \tilde{L}_{u_0} h+Sh+\ps{Sh}{u_0}u_0\,,
     $$
     implying, as $\ps{Sh}{u_0}=\ps{f_0^{\,u_0}}{u_0}=0$,
     $$
        \tilde{L}_{u_0} h=(\la_0-1)h\,.
     $$
     That means, $h$ is an eigenvector of $\tilde{L}_{u_0} $ associated with an eigenvalue strictly less than $\la_0\,,$ which is impossible. 
     Therefore $\ps{1}{f_0^{\, u_0}}\neq0\,,$ and so by the first identity of Corollary~\ref{evolution theta[n]}, we infer $\theta_0(t)=-\la_0^2\,t\,.$ \ 
     Substituting \eqref{theta[0]} in \eqref{n[1]ct=}\,, we obtain
     \begin{gather*}
     \left|
        \begin{array}{l}\displaystyle
         c=1+\frac{2}{n_1}\,\sum_{k=0}^{n_1-1}\la_k\,.
         \\[0.5cm]
         \displaystyle c=1+\frac{2}{n_2}\,\sum_{k=0}^{n_2-1}\la_k\,.
     \end{array}
     \right.
     \end{gather*}
     That is,
     $$
        n_2\,\sum_{k=0}^{n_1-1}\la_k=n_1\,\sum_{k=0}^{n_2-1}\la_k\,,
     $$
     or
     \[
        (n_2-n_1) \,\sum_{k=0}^{n_1-1}\la_k\,=\,n_1\,\sum_{k=n_1}^{n_2-1}\la_k
     \]
     But recall by \eqref{simplicite val prop Lu tilde}\,, $\la_{n+1}> \la_n\,,$ for all $n\,.$ Combining this fact with the last equality, we  conclude
     $$
        n_1(n_2-n_1)\,\la_{n_1-1}\,>\, n_1(n_2-n_1)\,\la_{n_1}\,,
     $$
     leading to a contradiction. As a consequence, for any traveling wave solution $u(t,x):=u_0(x-ct)$ of \eqref{CS-defocusing}\,, there exists at most one $N\in\N$ such that 
     $$
        \ps{u_0}{f_{N}^{\,u_0}}\neq0 \,,
    $$
    where $(\fnuzero)$ is any orthonormal basis of $\Ltwo$ consisting of the eigenfunctions of $\Luzerotilde$. Moreover, $u$ 
    travels with the speed
    \be\label{speed c defocusing}
         c=1+\frac{2}{N}\,\sum_{k=0}^{N-1}\la_k\,.
    \ee
\end{proof}

\begin{Rq} \label{Rq speed defocusing}
In view of the previous Theorem and Corollary~\ref{<u|f[n]> <=> gamma[n]=0}\,,  it follows  that any traveling wave solution $u$ of \eqref{CS-defocusing} propagates with a speed
\be \label{c=n+la[0]}
    c=N+2\la_0\,.
\ee
Indeed, since  $\ps{u_0}{\fnuzero}=0$ for all $1\leq n<N\,,$ then by Corollary~\ref{<u|f[n]> <=> gamma[n]=0}\,, 
\bes
    \la_n=\la_{n-1}+1\,, \qquad \qlq 1\leq n < N\,.
\ees leading to the fact that  \eqref{speed c defocusing} is equivalent to \eqref{c=n+la[0]}\,.
% $$
%     c=1+\frac{2}{N}\,\sum_{k=0}^{N-1}\la_k
% $$
% with $\ps{u_0}{\fnuzero}=0$ for all $1\leq n<N\,.$ Hence, by Corollary~\ref{<u|f[n]> <=> gamma[n]=0}\,, 
% \be\label{la[n]=la[n-1]+1}
%     \la_n=\la_{n-1}+1\,, \qquad \qlq 1\leq n < N\,.
% \ee
% Replacing \eqref{la[n]=la[n-1]+1} in \eqref{speed c defocusing} we obtain \eqref{c=n+la[0]}\,.
Besides, since $\Lutilde$ is a non--negative operator, then $\la_0\geq 0$\,, which implies that the speed of the traveling wave solution satisfies $c\geq N\,.$ 
However, as will be observed in Subsection~\ref{L2 norm and speed defoc}, the speed $c=N$ can only  be reached by traveling waves of the form $u(t,x)=\eee^{iN (x-N t)}\,.$
\end{Rq}

%----------------------------------------------
%----------------------------------------------
\subsection{Explicit formulas of the traveling waves}
\label{Explicit formulas trav. Waves for (CS-)}

Recall by Remark~\ref{formule explicite en terme matricielle}\,, any elements of  the Hardy space, in particular $u_0\,,$ can be written as
\be\label{formule inversion spectrale u0}
      u_0(z)=\ps{(\Id-zM)^{-1}\,X}{Y}\,, 
\ee
where   $X\,,$ $Y $ are infinite column vectors, $M$ is an infinite matrix :
\be\label{X,Y,M}
    X:=\big(\ps{u_0}{\fnuzero}\big)\,,
    \quad  
    Y:=\big(\left\langle 1 \mid \fnuzero\right\rangle\big)\,,
    \quad M=\Big(\left\langle f_p^{\,u_0}\mid S \fnuzero\right\rangle\Big)\,.
\ee 
% and $\mathcal{D}_t$ is the diagonal infinite matrix $\mathcal{D}_t=\mrm{diag}(\eee^{(-2\la_n-1)it}\,)\,.$
% In particular, at $t=0\,,$
% \be\label{formule inversion spectrale u0}
%     u_0(x)=\ps{(\Id-\eee^{ix} M)^{-1}\,X}{Y}
% \ee
% with $X$\,, $Y$\,, $M$ are described in \eqref{X,Y,M}\,.
\vskip0.25cm
In the following, we denote by $\mathcal{G}_1$ the set of the \textit{semi--trivial traveling waves}, made up from the constant and the plane wave solutions
\be\label{G1}
    \G_1=\lracc{C\eee^{iN(x-Nt)}\,\mid\, C\in\C\,,\, N\in \Nzero}\,.
\ee
\begin{theorem*}[\ref{traveling waves for the defocusing CS}]
The traveling waves $u(t,x)=u_0(x-ct)$  of \eqref{CS-defocusing} are the potentials $u(t,x)\in \mathcal{G}_1$ and
    \[
       u(t,x):= \eee^{i\theta}\left(\alpha+
        \frac{\beta}{1-p\eee^{iN(x-c t)}}\right)\,,
                    \qquad
            p\in \D^*\,,\, \theta\in \T\,, \,
    \]
    where $N\in\N\,,$ $\,c:=-N\Big(1+\frac{2\alpha}{\beta}\Big),$ and $(\alpha,\beta)$ are two real constants satisfying
    \be\label{alpha, beta- cond. defo}
            \alpha\beta +\frac{\beta^2}{1-\va{p}^2}=-N\,.
    \ee
\end{theorem*}

% \begin{Rq}\label{alpha, beta-defo non vanishing}
%     Observe that $\alpha$ and $\alpha+\beta$ cannot vanish, as long $(\alpha\,,\beta)$ satisfies the condition~\eqref{alpha, beta- cond. defo}\,, because otherwise we obtain 
%     \[
%     \frac{\beta^2}{1-\va{p}^2}=-N \quad\text{or}\quad \frac{\va{p}^2\, }{1-\va{p}^2}\,\beta^2=-N\,,
%     \]
%     where $p\in\D^*\,,$ which is clearly impossible. That implies the non--existence of traveling waves for \eqref{CS-defocusing} with profiles $\frac{\beta}{1-p\eee^{iNx}}$ or $\frac{a\eee^{iNx}}{1-p\eee^{iNx}}$\,.
% \end{Rq}

\begin{proof}
    The proof is based on the inversion spectral formula
    \[
        u_0(z)=\ps{(\Id-z M)^{-1}\,X}{Y}\,,
     \] of~\eqref{formule inversion spectrale u0}\,, and on the spectral characterization of $u_0$ described in Theorem~\ref{characterization u in the phase space <u|fn}.
      % for a traveling wave solution $u(t,x):=u_0(x-ct)\,$.
     In the sequence, to make the notation less cluttered, we denote $f_n:=f_n^{\,u_0}\,.$
    \vskip0.25cm
    Let $u(t,x):=u_0(x-ct)\,$. As a first step, we prove that the infinite matrices $X\,,Y$ and $M$ reduce to finite matrices in the context of a traveling wave solution. 
    Indeed, by Theorem~\ref{characterization u in the phase space <u|fn}\,, there exists at most one $N\in\N\,,$ such that  $\ps{u_0}{f_N}\neq 0$\,.
    We focus on the case where such an $N$ exists, that is:
    \be\label{N exists}
        \begin{cases}
            \displaystyle\ps{u_0}{f_{N}}\neq0
            \\[0.12cm]
            \displaystyle\ps{u_0}{f_n}=0 \,,\; \qlq n\in\N\bk\lracc{N}
        \end{cases}\,.
    \ee
    The case where $\ps{u}{f_n}=0$ for all $n\in\N$ can be handled similarly, leading also to the reduction of the study to finite matrices. 
    From now on, we suppose \eqref{N exists} holds.
    Therefore, it follows by Lemma~\ref{relation between <Sf|f> <u|f> and <1|f>}\,, that $\la_n \ps{1}{f_n}=0\,,$ implying that
    $$
        \ps{1}{f_n}=0\,, \qquad \qlq n\in\N\bk\lracc{N}\,,
    $$
    as the eigenvalues $\la_n$ are all positive for any $n\in\N$ since $\Lutilde$ is a non--negative operator. 
    % and as for all $n\in\Nzero$\,, $\ \la_{n+1}\geq \la_n+1\,$ thanks to \eqref{simplicite val prop Lu tilde} .
    Therefore, the two infinite column vectors $X$ and $Y$ of \eqref{X,Y,M} reduces to 
    \be\label{X}
    X=\left(\begin{array}{c}
    \ps{u_0}{f_0}\\
    0\\
    \vdots\\
    0\\
    \ps{u_0}{f_N}\\
    0\\
    \vdots
    \end{array}\right)
\,,
\qquad
    Y=\left(\begin{array}{c}
    \ps{1}{f_0}\\
    0\\
    \vdots\\
    0\\
    \ps{1}{f_N}\\
    0\\
    \vdots
    \end{array}\right)\,.
\ee
\vskip0.2cm
On the other hand, since $\ps{u_0}{f_n}=0\,$, for all $n\in\N\bk\lracc{N}$\,, then by Corollary~\ref{<u|f[n]> <=> gamma[n]=0}\,, we have $\la_n=\la_{n-1}+1$ for all $n\in\N\bk\lracc{N}\,$. Whence, $Sf_{n-1}\,\parallelsum \,f_n$ for all $n\in\N\bk\lracc{N}\,,$ thanks to Proposition~\ref{gap=0 and eigenspaces}\,. More specifically, 
\be\label{fn parallele}
\begin{cases}
    f_n\,\parallelsum \,S^nf_0\,\hskip0.6cm, \qquad 1\leq n\leq N-1
    \\
    f_n\,\parallelsum\, S^{n-N}f_N\,, \qquad  n\geq N
\end{cases}\,.
\ee
% where we recall that the $(f_n)$ form an orthonormal basis of $\Ltwo\,.$
As a consequence, the set $ \big\{(S^nf_0)_{n=0\,,\ldots\,, N-1}\,, (S^nf_N)_{n\geq 0}\big\}$ is an orthonormal basis of $\Ltwo$  and the  matrix $M=\big(\ps{f_p}{Sf_n}\big)$ reduces to 
$$
    M= 
    \left(\begin{array}{ccccc|cccc}
    0 & 1&0& \ldots &0 &0&\cdots&0
    \\
    \vdots& \ddots & \ddots&\ddots&\vdots&\vdots&&\vdots
    \\ 
    0&&\ddots&1&0&&\\
    \langle f_0\mid S^{N} f_0\rangle & 0 &\ldots&0 &\ps{f_N}{S^{N} f_{0}} &0&\\
    0&\ldots&&\ldots&0&1&0
    \\\hline
    0& \ldots&&\ldots&0&0&1&0\\
    \vdots&&&&\vdots&\vdots&\ddots&\ddots
    \end{array}\right).
$$
% and
% $$
%     D=\left(\begin{array}{ccccccccc}
%          \eee^{-i(2\la_0+1)t} &  \\
%           &  \eee^{-i(2(\la_0+1)+1)t}\\
%            &                          &\ddots  \\
%            &                          &      & \eee^{-i(2(\la_0+N-1)+1)t}\\
%            &                          &      & &\eee^{-i(2\la_N+1)t}\\
%            &                          &      & &  &\eee^{-i(2(\la_N+1)+1)t}\\
%            &                          &      & &  &&\ddots
%     \end{array}\right)
% $$
\vskip0.2cm
Hence, the infinite matrices $X\,,\, Y$ and $M$ in formula \eqref{formule inversion spectrale u0} can be restrained to  finite matrices involving only the first $N+1$ coordinates of $X\,$, $Y$, and $M$ \cite{GK21}\,.
 Indeed, denoting $\boldsymbol{\xi}:=\left(\Id-zM\right)^{-1}X$\,,  we have 
 $$
    \left(\Id-zM\right)\boldsymbol{\xi}=X\,.
$$
That is,
$$ 
    {\tiny
    \left( \begin{array}{ccccc|cccc}
    1 & -z&0& \ldots &0 &0&\cdots&0
    \\
    \vdots& \ddots & \ddots&\ddots&\vdots&\vdots&&\vdots
    \\ 
    0&&\ddots&-z&0&&\\
    -\langle f_0\mid S^{N} f_0\rangle z & 0 &\ldots&1 &-\ps{f_N}{S^{N} f_{0}}z &0&\\
    0&\ldots&&\ldots&1&-z&0
    \\\hline
    0& \ldots&&\ldots&0&1&-z&0\\
    \vdots&&&&\vdots&\vdots&\ddots&\ddots
    \end{array}
    \right)
    \cdot
    \left(
    \begin{array}{c}
    \xi_0\\
    \vdots\\
    \xi_{N-1}\\
    \xi_N\\
    \xi_{N+1}\\
    \vdots
    \end{array}\right)
    =
    \left(\begin{array}{c}
    \ps{u_0}{f_0}\\
    0\\
    \vdots\\
    0\\
    \ps{u_0}{f_N}\\
    0\\
    \vdots
    \end{array}\right)}
$$
Thus, for all $n\geq N+1\,,$ the $n^\text{th}$ coordinate of $\boldsymbol{\xi}$ is 
% equal to 
$\xi_n=z\,\xi_{n+1}\,$,  i.e. 
$$
    \xi_{N+1}=z^{n-N-1}\xi_n\,,\qquad\qlq n\geq N+1\,.
$$
And since $\sum_{n\geq0}|\xi_n|^2<\infty\,$, then 
$$
    \xi_n=0\,,\qquad \qlq n\geq N+1\,.
$$
As a result,
\be\label{Xn-zero}
    \ps{\left(\Id-z(M_{mn})_{m,n\geq N+1}\right)^{-1}(X_n)_{n\geq N+1}}{(Y_m)_{m\geq N+1}}=0\,,
\ee
and therefore
$$
    u_0(z)=\ps{\left(\Id-zM_{\leq N}\right)^{-1}X_{\leq N}}{Y_{\leq N}}_{\C^{N+1}\times\C^{N+1}}\,,
$$
where $M_{\leq N}:=(M_{mn})_{\,0\leq m,n\leq N}\,,$ $X_{\leq N}:=(X_n)_{\,0\leq n\leq N}$ and $Y_{\leq N}:=(Y_N)_{\,0\leq n\leq N}\,.$
Consequently, $u_0$ is a rational function 
$$
   u_0(z)=\frac{P(z)}{\det\left(\Id-zM_{\leq N}\right)}\;,\; 
$$
where $P(z)=Y_{\leq N}^*\cdot\,\mrm{Com}(\Id-z M_{\leq N})^T\,\cdot X_{\leq N}$.
Computing the numerator $P$ and the denominator of $u_0$ via these finite matrices, we obtain that $u_0$ is of the form
\be\label{u0}
    u_0(z)=\frac{az^N+b}{1-pz^N}\,,\qquad a,b\in \C\,,
\ee
where $p=\ps{f_0}{S^Nf_0}\,,$  $\va{p}<1\,.$ 
\vskip0.25cm
\noindent
% Two cases should be distinguished: $p=0$ and $p\neq 0\,.$ 
\underline{If $p=0\,$}. Namely, if $\ps{f_0}{S^Nf_0}=0$\,, then
$$
    S^Nf_0=\sum_{n\geq n}\ps{S^Nf_0}{f_n}f_n=\ps{S^Nf_0}{f_N}f_N\,,
$$
since by  \eqref{fn parallele}\,, the set $ \big\{(S^nf_0)_{n=0\,,\ldots\,, N-1}\;,\, (S^nf_N)_{N\geq 0}\big\}$ is an orthonormal basis of $\Ltwo$\,. 
%made up of the eigenfunctions of the self--adjoint operator  $\Lutilde$\,. 
Thus, the two vectors $f_N$ and $S^Nf_0$ are collinear,
%$f_N\parallelsum S^Nf_0\,,$ 
leading to : for all $n\in\Nzero\,,$
$$
    f_n\parallelsum S^nf_0\,, 
$$
thanks to \eqref{fn parallele}\,. Consequently, $\lracc{S^nf_0\,, n\in\Nzero}$ is an orthonormal basis of $\Ltwo$\,, which means,
the vector $f_0$ is necessarily collinear to $1\,.$ Besides, recall from \eqref{X}\,,
\begin{align*}
    u_0=&\,\ps{u_0}{f_0}f_0+\ps{u_0}{f_N}f_N
    \\
    =&\,\ps{u_0}{f_0}f_0+\ps{u_0}{S^Nf_0}S^N f_0\;,
    \\
    =&\, \ps{u_0}{1}+\ps{u_0}{\eiNx}\eiNx
\end{align*}
and, as $p=0$ i.e. $\ps{f_0}{S^Nf_0}=0\,,$ we have by the second identity of Lemma~\ref{relation between <Sf|f> <u|f> and <1|f>}\,, either
$$
\ps{u_0}{f_0}=0\qquad \text{or} \qquad \ps{u_0}{S^Nf_0}=0\,.
$$
i.e.
$$
\ps{u_0}{1}=0\qquad \text{or} \qquad \ps{u_0}{\eiNx}=0\,.
$$
Therefore, either $u_0(x)$ is a complex constant, or $u_0(x)=C \eiNx\,,$ with $C\in\C$\,, $N\in\N\,.$ 
Taking, $u(t,x)=u_0(x-ct)=C\eee^{iN(x-ct)}$, and substituting it in the defocusing Calogero--Sutherland DNLS equation \eqref{CS-defocusing}\,,
we infer, since the nonlinearity $D\Pi(\va{\eee^{iN(x-ct)}}^2)\eee^{iN(x-ct)}$ vanishes,
$$
    Nc \eee^{iN(x-ct)}-N^2 \eee^{iN(x-ct)}=0\,,
$$
and thus $c=N\,.$ As a result, if $p=0$ then the traveling waves $u(t,x):=u_0(x-ct)$ are 
$$
u(t,x)=C\eee^{iN(x-Nt)}\,, \qquad C\in\C\,,\, N\in\Nzero \,.
$$
\vskip0.25cm
Let us move, to the case where \underline{$p\neq 0\,.$}
% First, it is worth noting that any function on the hardy space $\Ltwo$ can be seen as an analytic function on the open unit disc $\D$\,, whose trace on the boundary $\partial\D$ is in $L^2\,$\footnote{~For a simple introduction to the different definitions of Hardy space, we refer to \cite[Chapter 3.]{GMR16}}
% $$
% 	u(z)=\sum_{k\geq 0}\fr{u}(k)z^k \overset{\sim}{\longmapsto} u^*(x):=\sum_{k\geq 0}\fr{u}(k)\eee^{ikx}\,, \qquad\sum_{k \geq 0}\va{\fr{u}(k)}^2<\infty\,.
% 	$$
The potential $u_0$ of \eqref{u0} can be rewritten as
$$
u_0=\alpha +\frac{\beta}{1-pz^n} \,, \qquad \alpha,\beta\in\C\,, \quad p,z\in\D\,,
$$
In order to find the relation between $\alpha,\beta$ and obtain the speed $c\,,$ 
% the approach consists of substituting 
we substitute $u(t,z):=u_0(\eee^{-ic t}z)$  into the defocusing Calogero--Sutherland DNLS equation \eqref{CS-defocusing}. This equation can be rewritten as
\be\label{CS-defocusing-z}
    i\p_tu-(z\p_z)^2u-2z\p_z\Pi(\va{u}^2)u=0\,,
\ee
after observing that  $D=-i\p_x$ can be expressed as $D\equiv z\p_z$.
Thus, starting from 
\[
    u:=u(t,z)=\alpha+\frac{\beta}{1-p\eee^{-iNct}z^N}\,,
\]
and computing $i\p_t u$ and $(z\p_z)^2u\,,$ we find

\[
    i\p_tu =-c\beta N \left(\frac{1}{1-p\eee^{-iNct}z^N}-\frac{1}{(1-p\eee^{-iNct}z^N)^2}\right)\,,
\]
\[    
    (z\p_z)^2u = \beta N^2\left(\frac{1}{1-p\eee^{-iNct}z^N}-\frac{3}{(1-p\eee^{-iNct}z^N)^2}+\frac{2}{(1-p\eee^{-iNct}z^N)^3}\right)\,.
\]
For the nonlinear part, 
\begin{align*}
    \va{u}^2=
    \va{\alpha}^2+\alpha\bar{\beta}
    +\frac{\alpha\bar{\beta}\,\bar{p}\eee^{iNct}}{z^N -\bar{p}\eee^{iNct}}
    +\frac{\alpha\bar{\beta}}{1-p\eee^{-iNct}z^N}
    +\frac{\va{\beta}^2 \,z^N}{(1-p\eee^{-iNct}z^N)(z^N- \bar{p}\eee^{iNct})}\,.
\end{align*}
Recall that $\Pi$ is an orthonormal projector into the Hardy space (in particular to a subspace of the holomorphic functions on $\D$). Thus, applying $\Pi$\,, it follows
\[
    \Pi(\va{u}^2)=\va{\alpha}^2+\alpha\bar{\beta}
     +\frac{\alpha\bar{\beta}}{1-p\eee^{-iNct}z^N}
     +\frac{\va{\beta}^2}{1-\va{p}^2}\frac{1}{1-p\eee^{-iNct}z^N}\,.
\]
    And hence,
\[
    z\p_z\Pi(\va{u}^2)\cdot u=A\left(\frac{-\alpha}{1-p\eee^{-iNct}z^N}+\frac{-\beta+\alpha}{(1-p\eee^{-iNct}z^N)^2}+\frac{\beta}{(1-p\eee^{-iNct}z^N)^3}\right)\,,
\]
where 
$$
    A=N\left(\bar{\alpha}\beta+\frac{\va{\beta}^2}{1-\va{p}^2}\right)\,.
$$
Substituting the expressions of $i\p_t u\,,\,$ $(z\p_z)^2u\,$ and $\,z\p_z\Pi(\va{u}^2)u$ into \eqref{CS-defocusing-z}\,, and comparing the terms $\displaystyle\frac{1}{(1-p\eee^{-iNct}z^N)^n}$ for $n=1,2,3$\,, we deduce
\begin{itemize}
    \item With $n=3\,,$ $A=-N^2\,.$ That is,
    \bes
        \bar{\alpha}\beta+\frac{\va{\beta}^2}{1-\va{p}^2}=-N
    \ees
    \item With $n=2$ and $n=1\,,$
    \[
    c=-N\left(1+\frac{2\alpha}{\beta}\right)\,.
    \]
\end{itemize}
As a result, for $p\neq 0\,,$
\[
     u(t,z):= \alpha+
        \frac{\beta}{1-p\eee^{-iNc t}z^N}\,,
                    \qquad
            p\in \D^*\,,\, \theta\in \T\,, \,
\]
where $N\in\N\,,$ $\,c:=-N\Big(1+\frac{2\alpha}{\beta}\Big),$ and $(\alpha,\beta)\in\C\times\C$  satisfy
    \be\label{cond alpha, beta complex def}
            \alpha\beta +\frac{\va{\beta}^2}{1-\va{p}^2}=-N\,.
    \ee
 Finally, observe by \eqref{cond alpha, beta complex def}\,, the two complex constants $(\alpha,\beta)$ satisfy  $\bar{\alpha}\beta\in \R\,.$ Thus,  by making a slight abuse of notation on $\alpha$ and $\beta$, we 
    % can suppose that $\alpha$ and $\beta$ are real-valued, up to multiplying $u$ by a phase. 
    have obtained that the traveling waves of \eqref{CS-defocusing} with $p\neq 0$ are given by  \[
       u(t,z):= \eee^{i\theta}\left(\alpha+
        \frac{\beta}{1-p\eee^{-iNc t}z^N}\right)\,,
                    \qquad
            p\in \D^*\,,\, \theta\in \T\,, \,
    \]
    where $N\in\N\,,$ $\,c:=-N\Big(1+\frac{2\alpha}{\beta}\Big),$ and $(\alpha,\beta)\in\R\times\R$ satisfy
    \[
            \alpha\beta +\frac{\beta^2}{1-\va{p}^2}=-N\,.
    \]

\end{proof}

%---------------------------------------------
%---------------------------------------------

\subsection{The \texorpdfstring{$L^2$}{L2}--norm and the speed} \label{L2 norm and speed defoc}

In this subsection, we analyze  how the traveling waves of \eqref{CS-defocusing} behaves, by providing information regarding their $L^2$--norm and their speed $c$. 
Recall that the set of  traveling wave solutions of the defocusing Calogero--Sutherland DNLS equation are made up by the trivial solutions
\[
 \G_1=\lracc{C\eee^{iN(x-Nt)}\,\mid\, C\in\C\,,\, N\in \Nzero}\,,
\]
and by the set of functions
    \be\label{non semi trivial wave for CS-}
       u(t,x):= \eee^{i\theta}\left(\alpha+
        \frac{\beta}{1-p\eee^{iN(x-c t)}}\right)\,,
                    \qquad
            p\in \D^*\,,\, \theta\in \T\,, \,
    \ee
    where $N\in\N\,,$ $\,c:=-N\Big(1+\frac{2\alpha}{\beta}\Big),$ and $(\alpha,\beta)$ are two real constants
    % \footnote{~By identity~\eqref{alpha, beta- cond. defo}\,, the two real constant $\alpha$ and $\beta$ are of opposite sign.} 
    satisfying \eqref{alpha, beta- cond. defo}\,.
    % \be\label{alpha, beta- cond. defo}
    %         \alpha\beta +\frac{\beta^2}{1-\va{p}^2}=-N\,.
    % \ee
    \vskip0.25cm
For $u\in \G_1\,,$ it is easy to see that the $L^2$--norm of the semi--trivial solution can be arbitrarily small or large in $[0,+\infty)\,,$ and its speed  $c$ is given as $c=
N\in\Nzero\,.$ 
   The following proposition aims to provide those for the nontrivial traveling waves of \eqref{CS-defocusing}\,. 
   
\begin{prop}[$L^2$ norm of a non--trivial traveling wave and the speed]\label{norme L2 +c defo}
\hfill
\begin{enumerate}[label=(\roman*),itemsep=2pt]
    \item For any $r>0\,,$ there exists  a non--trivial traveling waves $ u(t,x):=u_0(x-ct)$ for \eqref{CS-defocusing} with 
    \[
        \|u_0\|_{L^2}=r\,.
    \]
    In other words, the traveling waves of \eqref{CS-defocusing} can be arbitrarily small or large in $\Ltwo\,.$
    \item Let $u$ be a traveling wave for \eqref{CS-defocusing} of the form \eqref{non semi trivial wave for CS-}\,, then $u$ propagates to the right with a speed $c>N\,.$ 
    In addition, when $\|u\|_{L^2}\to \infty$ 
    % is arbitrary large, 
    then $c\to\infty$
    % the traveling wave propagates with a high relatively speed, 
    and when $\|u\|_{L^2}\to 0$ then $c\to N\,.$ 
\end{enumerate}
\end{prop}
\begin{Rq}[\textit{Non-existence of stationary solution for \eqref{CS-defocusing}}]
% \hfill
% \begin{enumerate}
%     \item 
    % In view of the previous proposition, we conclude that the limit speed $c=N$ of a traveling waves of \eqref{CS-defocusing} of degree $N\,,$ is only  attained by the trivial solution $u\in\G_1\,.$  
    % \item 
    Since for any traveling wave $u_0(x-ct)$ of the defocusing Calogero--Sutherland DNLS equation \eqref{CS-defocusing} we have $c\geq N\,,$ where $N$ is the numerator's degree of $u_0\,,$ then there is no stationary solution (i.e. $u(t,x)=u_0(x)$) for the \eqref{CS-defocusing}--equation\,.
    Another way to see this, is by observing that if $c=0\,$, which occurs when $\alpha=-\frac{\beta}{2}$ according to Theorem \ref{traveling waves for the defocusing CS}, then we have by~\eqref{alpha, beta- cond. defo} 
    $$
        \frac{1+\va{p}^2}{1-\va{p}^2}\,\beta^2=-N\,,
    $$ which is impossible as $p\in\D^*$.
% \end{enumerate}
\end{Rq}

\begin{proof}
    (i) \textit{The $L^2$--norm of the non--trivial traveling wave can be arbitrarily small or large.} Let $u$ be a traveling wave of the form \eqref{non semi trivial wave for CS-} 
    \[
        u(t,x):= \eee^{i\theta}\left(\alpha+
        \frac{\beta}{1-p\eee^{iN(x-c t)}}\right)\,, \qquad p\in\D^*\,,\, N\in\N\,,
    \]
    where $(\alpha,\beta)\in\R^2$ satisfies 
    the identity~\eqref{alpha, beta- cond. defo}\,. Recall that any function $u$ in the Hardy space can be seen as an analytic function on the open unit disc
    $\D\,,$ whose trace on the boundary $\p \D$ is in $L^2\,.$ Hence,
    \[
        \|u\|_{L^2}^2
        =\,\intc \va{u(z)}^2\, \dz\,,
    \]
    where
    \begin{align*}
        \va{u(z)}^2=&\,\left(\alpha+
        \frac{\beta}{1-p\eee^{-iNc t}z^N}\right)
        \left(\alpha+
        \frac{\beta\, z^N}{z^N-\bar{p}\eee^{iNc t}}\right) 
        \\
        =&\; \alpha^2+\alpha\beta+\frac{\alpha\beta \eee^{iNc t} \bar{p}}{z^N-\bar{p}\eee^{iNc t}}+\frac{\alpha\beta}{1-p\eee^{-iNc t}z^N}
        +\frac{\beta^2\, z^N}{(1-p\eee^{-iNc t}z^N)(z^N-\bar{p}\eee^{iNc t})}
    \end{align*}
    Writing 
    \[
        \frac{\beta^2\, z^N}{(1-p\eee^{-iNc t}z^N)(z^N-\bar{p}\eee^{iNc t})}=
        \frac{\beta^2}{1-\va{p}^2}\left(\frac{1}{1-p\eee^{-iNc t}z^N}+\frac{\bar{p}\eee^{iNc t}}{z^N-\bar{p}\eee^{iNc t}}\right)\,,
    \]
we infer
\[
     \va{u(z)}^2= \alpha^2+\alpha\beta
     +\left(\alpha\beta +\frac{\beta^2}{1-\va{p}^2}\right)\frac{1}{1-p\eee^{-iNc t}z^N}
     +\left(\alpha\beta +\frac{\beta^2}{1-\va{p}^2}\right)\frac{ \eee^{iNc t} \bar{p}}{z^N-\bar{p}\eee^{iNc t}}\,.
\]
Therefore, 
\begin{align}\label{u L2}
    \|u\|_{L^2}^2=\alpha^2+\alpha\beta
     +\alpha\beta +\frac{\beta^2}{1-\va{p}^2}  \;,
\end{align}
since for $N\in\N\,,$
$$
    \left\langle 1\;\Big|\;\frac{1}{z^N-\bar{p}\eee^{iNc t}}\right\rangle=\intc \frac{z^N}{1-p\eee^{-iNc t}z^N}\dz=0\,.
$$
Consequently, 
by \eqref{alpha, beta- cond. defo}\,, \footnote{As we shall see in Corollary~\ref{L2 norm of a finite gap} of Section~\ref{finite gap potentials}, this corresponds to $\|u\|_{L^2}^2=\la_N-N$ where $\la_N>N+\la_0>N\,.$}
\be\label{u L2 =a[2]+ab-N}
    \|u\|_{L^2}^2=\alpha^2+\alpha\beta-N\,.
\ee
In addition,
since by \eqref{alpha, beta- cond. defo}
\[
    \alpha=-\frac{N}{\beta}-\frac{\beta}{1-\va{p}^2}\,,
\]
then,
\begin{align*}
    \|u\|_{L^2}^2
    =&\,\left(-\frac{N}{\beta}-\frac{\beta}{1-\va{p}^2}\right)^2+\left(-\frac{N}{\beta}-\frac{\beta}{1-\va{p}^2}\right)\beta-N
    \\
    =&\,\frac{\va{p}^2}{1-\va{p}^2}\left(\frac{\beta^2}{1-\va{p}^2}+2N\right)+\frac{N^2}{\beta^2}\,.\notag
\end{align*}
Observe that, $\|u\|_{L^2}^2$ is a continuous function of $\va{p}^2$ and $\beta^2$\,. Moreover,  by taking $\beta\to0$ then  $\|u\|_{L^2}^2\to \infty\,.$ And if we take  $\va{p}^2\to0$ then 
$$
    \|u\|_{L^2}^2\underset{\va{p}^2\to0}{\sim}\; \frac{N^2}{\beta^2}
$$
which can be arbitrary small when $\beta>\!>1$.

\vskip0.3cm
\noindent
(ii) 
\textbf{Speed : $\mathbf{c>\textit{N}}\,$.}  By Theorem~\ref{traveling waves for the defocusing CS}\,, the speed of the traveling waves of the form \eqref{non semi trivial wave for CS-} is given by
        $c=-N(1+\frac{2\alpha}{\beta})\,.$ Besides, recall from \eqref{alpha, beta- cond. defo}\,, 
        $$
            \frac{\alpha}{\beta}=-\frac{N}{\beta^2}-\frac{1}{1-\va{p}^2}\,.
        $$ 
         Substituting the latter identity in the expression of $c$\,, it follows
         \be\label{speed defocusing trav waves}
            c\,=\,N\left(\frac{1+\va{p}^2}{1-\va{p}^2}+\frac{2N}{\beta^2}\right)\,>\,N\,.
         \ee
         \vskip0.25cm
         \noindent
    It remains to prove that 
    \begin{itemize}
        \item when $\|u\|_{L^2}\to +\infty$\,, we have $c\to+\infty\,,$ 
        \item and when $\|u\|_{L^2}\to 0$ then $c\to N\,.$  
    \end{itemize}
    Indeed, observe that $\|u\|_{L^2}^2\to \infty$ when  $\beta^2\to 0$ or $\va{p}^2\to1\,,$
    and in both cases $$c\to \infty\,.$$ On the other side, $\|u\|_{L^2}^2$ is arbitrary small when $\va{p}^2\to 0$ and $\beta$ is big enough.
    % Indeed, since  the speed of the traveling waves of the form \eqref{non semi trivial wave for CS-} is given by
    %     $c=-N(1+\frac{2\alpha}{\beta})\,,$ where $\alpha, \beta$ satisfy \eqref{alpha, beta- cond. defo}\,, then by repeating the previous computation on the analysis of the $L^2$--norm, we have for $\alpha=\delta\,\sqrt{\frac{N}{1-\va{p}^2}}$ where $\delta>2$ is a parameter, 
    %     \begin{align}\label{speed c delta}
    %         C
    %         =&\,-N\left(1+\frac{2\alpha}{\beta_1}\right)
    %         \\
    %         =&\,-N\left(1-\frac{4}{\left(1+\sqrt{1-\frac{4}{\delta^2}}\,\right)(1-\va{p}^2)}\right)\notag\,.
    %     \end{align}
    % On the other hand, recall by the previous computation 
    % \begin{itemize}
    %     \item $\|u\|_{L^2}\to \infty$ when $\va{p}\to 1$\,,
    %     \item and  by taking the parameter $\delta>\!>2$\,, we have $\|u\|_{L^2}\to 0$ if $\va{p}\to 0\,.$
    % \end{itemize}
    Hence, by passing to the limit $\va{p}^2\to 0$ in \eqref{speed defocusing trav waves}\,, we infer 
    \[
        c \underset{\va{p}^2\to0}{\sim} N\left(1+\frac{2N}{\beta^2}\right)\,,
    \]
    which can arbitrary close to $N$ as $\|u\|^2_{L^2}$ is arbitrary close to $0$\,.
    % $c\to \infty\,.$ 
    % In addition, for $\delta>\!>2$\,,
    % \begin{align*}
    %     C
    %     =&-N\left(1-\frac{4}{\left(2-\frac{2}{\delta^2}+o(\frac{1}{\delta^2})\right)(1-\va{p}^2)}\right)
    %     \\
    %     =&-N\left(1-\frac{2}{1-\va{p}^2}+\frac{2}{1-\va{p}^2} \frac{1}{\delta^2}+o(\frac{1}{\delta^2})\right)
    %     \\
    %     \longrightarrow & +N\left(1-\frac{2}{\delta^2}+o(\frac{1}{\delta^2})\right)
    % \end{align*}
    % when $\va{p}\to 0\,.$
    
\end{proof}

%---------------------------------------------
%---------------------------------------------

\section{\textbf{Traveling waves for the focusing~(CS\texorpdfstring{$^+$}{+})}}
\label{Traveling waves for (CS+)}
% \vskip0.25cm
\subsection{Toward the characterization of the traveling waves for (CS\texorpdfstring{$^+$}{+})}
\label{characterization of traveling waves for (CS+)}
% In this section, we deal with the traveling waves $u_0(x-ct)$ of \eqref{CS-focusing}\,.
Recall that 
to characterize the traveling waves of the defocusing equation \eqref{CS-defocusing}, a spectral analysis was initially conducted, followed by the derivation of explicit formulas. Here, we aim  to footstep the same strategy. 
% andthe characterization of the traveling waves  for the defocusing \eqref{CS-defocusing} has been achieved by first  understanding them from a spectral viewpoint before establishing their explicit formulas.
% We propose to footstep the same strategy here.  
But before proceeding, 
 we shall require some analogous lemmas to the defocusing case.

\begin{lemma}[The analog of Lemma~\ref{evolution in the gnt}]
\label{evolution in the gnt-foc}
    Let $u\in \mathcal C_t\Htwo_x$ solution of \eqref{CS-focusing}. Then, for all $n\,,p\in\Nzero\,,$
    \begin{align}\label{evolution <1|fn>-foc}
        \ps{1}{\gnt}=&\,\ps{1}{\fnuzero}\,\eee^{-i\nu_n^2t}\,,\notag
        \\
        \ps{u}{\gnt}=&\,\ps{u_0}{\fnuzero}\eee^{-i\nu_n^2t} 
        \\
        \ps{S\gpt}{u}=&\,\ps{Sf_p^{\, u_0}}{u_0}\eee^{i(\nu_p+1)^2\,t}\,,\notag 
        \\
        \ps{S\gpt}{\gnt}=&\,\ps{S\fpuzero }{\fnuzero}\,\eee^{i((\nu_p+1)^2-\nu_n^2)\,t}\,,
        \notag
    \end{align}
    where the 
    $(\gnt)$ denotes the orthonormal basis of $\Ltwo$ solution to the Cauchy problem
\[
    \begin{cases}
			\p_{t} \,\gnt&={B}_{u(t)}\,\gnt 
			\\
			{\gnt\,}_{|_{t=0}}&=\fnuzero
		\end{cases}\,,\qquad \qlq n\in\Nzero\,,
\]
and $(\fnuzero)$ is an orthonormal basis of $\Ltwo$ made up of the eigenfunctions of  $L_{u_0}$  and ${B}_{u(t)}$ is the skew--adjoint operator defined in \eqref{Lax operators defoc case}\,.
\end{lemma}

\begin{proof}
    Since  the focusing Calogero--Sutherland DNLS equation can also be rewritten in terms of its Lax operators \cite[Lemma~2.4]{Ba23}
    $$
		\p_t u=B_u u -iL^2_uu\,,
  $$
  then one can repeat exactly  the same proof of Lemma~\ref{evolution in the gnt} and obtain the same results.
\end{proof}

\begin{lemma}[The analog of Proposition~\ref{tauct fn}]\label{tauct fn foc}
Let $u:=\tau_{ct}u_0$ be a traveling wave of \eqref{CS-focusing}\, such that the eigenvalue $\nu_n(u_0)$ is simple. Then,  there exists $\theta_n(t)\in\R$ such that 
    \be\label{tauctfn-gnt-foc}
        \tau_{ct}\fnuzero=\eee^{i\theta_n(t)}\gnt \;,
    \ee
    where the $(\gnt)$ denotes the orthonormal basis defined in the previous lemma.
\end{lemma}

\begin{lemma}[The analog of Corollary~\ref{evolution theta[n]} in the focusing case]\label{evolution theta[n] foc}
    Let $u_0$ be a function such that the eigenvalues $(\nu_n(u_0))$ are simple. Then, for all $n\,,p\in\Nzero\,,$ $t\in\R\,,$ we have
    \begin{enumerate}
        \item If $\langle 1\,|\fnuzero\rangle\neq 0$   then
            $$
                \theta_n(t)=-\nu_n^2t\,.
            $$
        \item  If $\langle u_0\,|\fnuzero\rangle\neq 0$  then  
            $$
                \theta_n(t)=-\nu_n^2t\,.
            $$
            % $$
            %     \theta_n(t)=-\la_n^2t-\varphi(t)\,.
            % $$
        % \item If $\langle Sf_p^{\, u_0}|u_0\rangle\neq 0$ then 
        %     $$
        %         \thetapt=ct-(\la_p+1)^2t-\varphi(t)\,.
        %     $$
        \item If $\langle S\fpuzero|\fnuzero\rangle\neq0$ then
            \[
                \theta_n(t)=((\nu_p+1)^2-\nu_n^2)t-ct+\thetapt\,,
            \]
    \end{enumerate}  
    where $\theta_n(t)$ is the angle obtained in \eqref{tauctfn-gnt-foc}\,.
\end{lemma}

\vskip0.25cm
At this stage, we are equipped with the necessary tools 
to footstep the proof of the defocusing equation.
% characterization of the \eqref{CS-defocusing}'s traveling waves.
However, 
% let us 
% But, allow us to
% underline 
% , and to provide some information regarding them.
% Again, our approach will be based on the spectral properties of the Lax operator $L_u$, who presents 
it is important to emphasize two fundamental 
% two main 
differences between the Lax operators $L_u$ and $\Lutilde\,,$ which ultimately offer a considerably expanded set of traveling
waves for  \eqref{CS-focusing} in comparison to \eqref{CS-defocusing}~:
% , compared to those for the defocusing equation~:
% with respect to $\Lutilde:$
% Recall that there are two main differences between $L_u$ and $\Lutilde\,$: 
\begin{itemize}
    \item The gap between the eigenvalues differ between the focusing and the defocusing case (Proposition~\ref{multiplicite val propres}).
    \item The fact that the eigenvalues $\la_n$ of $\tilde{L}_u$ are not zero for any $n\in\N$\,.
    % may present negative eigenvalues, unlike for $\Lutilde$\,, where  all its eigenvalues are non-negative.
\end{itemize}
Indeed, in  the defocusing case, since all the eigenvalues  satisfy $\la_n> \la_{n-1}+\frac12$ (Inequality~\eqref{simplicite val prop Lu tilde}) and $\la_n\neq 0$ for all $n\in\N$\,, then we obtained in Lemma~\ref{I=vide}\,,
\[
    \I(u)=\varnothing\,, \qquad \qlq u\in\Htwo\,,
\]
 where $\I(u)$ was defined in \eqref{I}\,. As a consequence, we inferred that if $u(t,x)=u_0(x-ct)$ is a traveling wave of \eqref{CS-defocusing}\,, then there exists at most one $N\in\N$ such that $\ps{u_0}{\fnuzero}=0$ for all $n\in\N\bk\{N\}\,.$
Now, for the focusing equation, recall we have previously observed in the second point of Remark~\ref{Rq I=vide}\,, that $\I(u)$ is of finite cardinal for all $u\in\Htwo\,$. In particular, for $u_0\in\Htwo\,,$ we denote by $m_1,\ldots,m_n$ its elements
\[
    \I(u_0)=\{m_1\,,\ldots\,, m_n\}\,.
\]

\begin{theorem}[Toward the Characterization of the traveling waves of \eqref{CS-focusing}]\label{trav waves for foc} 
    % The set $\mathcal{T}$ of t
    The traveling waves $u_0(x-ct)$ of \eqref{CS-focusing} are either rational functions or the plane waves $u(t,x)=C\eee^{iN(x-Nt)}$\,.
    % contains only rational functions or the trivial waves in $\G_1$\,. 
    In addition, the potentials
     \be\label{ex 1 trav waves foc}
                u(t,x):= \eee^{i\theta}\left(\alpha+
                \frac{\beta}{1-p\eee^{iN(x-c t)}}\right)\,, \qquad
            p\in \D^*\,,\, \theta\in \T\,, \, N\in\N\,,
    \ee
            where  $c= -N\left(1+\frac{2\alpha}{\beta}\right)$\,, $\,(\alpha,\beta)\in\R\times\R$ such that 
        \be\label{alpha, beta-fo}
               \alpha\beta +\frac{\beta^2}{1-\va{p}^2}=N\,,
        \ee
        and the potentials 
         \be\label{ex 2 trav waves foc}
            u(t,z)=\eee^{i\theta}\eee^{im(x-ct)}\left(\alpha+
                \frac{\beta}{1-p\eee^{i(x-c t)}}\right)
            \qquad p\in \D^*\,,\, \theta\in \T\,, \, m\in\N\,,
        \ee
        where $c=m$\,, $(\alpha,\beta)\in\R\times\R$ such that 
        \[
             \alpha\beta +\frac{\beta^2}{1-\va{p}^2}=1\,,\qquad \beta (m-1)=2\alpha\,,
        \]
        % are parts of this set $\mathcal{T}\,.$
        are parts of the set of the traveling waves of \eqref{CS-focusing}\,. 
        
\end{theorem}
\begin{proof}\hfill
\\
    \textbf{Step 1. }(Spectral characterization of the traveling waves of \eqref{CS-focusing}). Let $u(t,x):=u_0(x-ct)$ be a traveling wave for \eqref{CS-focusing}\,. 
% and let $(f_n^{u_0})$ be an  orthonormal basis of $\Ltwo$ made up by the eigenfunctions of $\Luzero\,.$
Our goal is to prove that there exists $N\in\N$ such that $\ps{u_0}{f_n^{u_0}}=0$ for all $n\geq N\,.$  Once more, in order to simplify the notation, we denote in the following $f_n$ instead of $f_n^{u_0}$ . Recall by the second point of Remark~\ref{Rq I=vide}\,, $\I(u_0)$ is of finite cardinal, that is there exists $m_1<\ldots< m_j\in\N$ such that 
\[
\begin{cases}
    \ps{Sf_{n-1}}{f_n}=0\,, \qquad\qlq n\in\{m_1\,,\ldots m_{j}\}
    \\
    \ps{Sf_{n-1}}{f_n}\neq 0\,, \qquad\qlq n\in\N\bk\{m_1\,,\ldots m_{j}\}
\end{cases}\,.
\]
Suppose that there exists an integer $\ell>\!>1\,,$ $\ell>m_j$ such that $\ps{u_0}{f_{\ell}}\neq0\,.$ Otherwise, we already have what we claim to prove. 
Then, 
\begin{itemize}
\item For all $n\geq \ell+1\,,$ the quantities $\ps{Sf_{n-1}}{f_n}\neq 0$\,.
    \item Since $\ell>\!>1\,,$ then by inequality \eqref{liminf}\,,  the eigenvalues $(\nu_n)_{n\geq \ell+1}$ are simple. This implies that Lemma \ref{tauct fn foc} holds for $n\geq \ell+1\,.$
\end{itemize}
Therefore,
using the third point of Lemma~\ref{evolution theta[n] foc}\,, we obtain, for all $n\geq \ell+1\,,$
\begin{align*}
    \theta_n(t)
    =&\,-\nu_n^2t+(\nu_{n-1}+1)^2t-ct+\theta_{n-1}(t)
    \\
    =&\,-\nu_n^2\,t+(\nu_{n-1}+1)^2t-\nu_{n-1}^2\,t+(\nu_{n-2}+1)^2t-2ct+\theta_{n-2}(t)\notag
    \\
    =&\quad\ldots\notag
    \\
    =&\, 
    -\nu_n^2\,t +(n-\ell)t+\nu_{\ell}^2\,t-(n-\ell)ct+2t\sum_{k=\ell}^{n-1}\nu_{k}+\theta_{\ell}(t)\,,\notag
\end{align*}
where $\theta_n(t)$ is the angle obtained in Lemma \ref{tauct fn foc}\,, and $\theta_\ell(t)=-\nu_{\ell}^2\,t$ thanks to  the second point of Lemma~\ref{evolution theta[n] foc}\,.
Hence, for all $n\geq \ell+1\,,$
\be\label{theta[n] foc G3}
    \theta_n(t)
    = 
    -\nu_n^2\,t +(n-\ell)t-(n-\ell)ct+2t\sum_{k=\ell}^{n-1}\nu_{k}\,,
\ee
% Besides, recall by inequality \eqref{liminf}\,, there exists $N_1\in\N$ such that the eigenvalues $(\nu_k)$ of $\Luzero$ are all simple for $k\geq N_1\,.$
As a consequence, there exists at most one  integer $N\geq\ell$ such that 
\[
% \begin{cases}
    \ps{u_0}{f_{N}}\neq 0\,.
%     \\
%     \ps{u_0}{f_n}= 0\,, \qquad \qlq n\in\N_{\geq \ell}\bk \{N\}\,.
% \end{cases}
\]
Indeed, suppose for the sake of contradiction that there exist $n_2>n_1>\ell\,,$  such that  $\ps{u_0}{f_{n_1}}\neq 0$ and $\ps{u_0}{f_{n_2}}\neq 0\,.$ Then, combining  the second point of Lemma~\ref{evolution theta[n] foc}\,, and equation \eqref{theta[n] foc G3}\,, we obtain
\begin{gather*}
     \left|
        \begin{array}{l}\displaystyle
         c=1+\frac{2}{n_1-\ell}\,\sum_{k=\ell}^{n_1-1}\nu_k\,.
         \\[0.5cm]
         \displaystyle c=1+\frac{2}{n_2-\ell}\,\sum_{k=\ell}^{n_2-1}\nu_k\,.
     \end{array}
     \right.
     \end{gather*}
     Hence,  
     $$
        (n_2-\ell)\,\sum_{k=\ell}^{n_1-1}\nu_k=(n_1-\ell)\,\sum_{k=\ell}^{n_2-1}\nu_k\,,
     $$
     or
     \[
        (n_2-n_1) \,\sum_{k=\ell}^{n_1-1}\nu_k\,=\,(n_1-\ell)\,\sum_{k=n_1}^{n_2-1}\nu_k\,.
     \]
     As a result, 
     $$
        (n_1-\ell)(n_2-n_1)\,\nu_{n_1-1}\,\geq\, (n_1-\ell)(n_2-n_1)\,\nu_{n_1}\,,
     $$
     leading to a contradiction, since for $k\geq \ell\,,$ we have  $\nu_{k+1}> \nu_{k}\,.$ Therefore, there exists $N\in\N$ such that $\ps{u_0}{f_n}=0$ for all $n\geq N\,.$
\vskip0.25cm
\noindent
\textbf{Step 2.} (They are rational functions or potentials in $\G_1$)
% for the elements of $\G_3$).
Since $\ps{u_0}{f_n}=0$ for all $n\geq N$\,,  it follows by Corollary~\ref{<u|f[n]> <=> gamma[n]=0} that $\nu_n=\nu_{n-1}+1$ for all $n\geq N_2\,.$ 
Note that the  potentials satisfying $\nu_n=\nu_{n-1}+1$ for all $n\geq N_2\,,$ are referred to be ``finite gap potentials" for \eqref{CS-focusing}, and are studied deeply in Section~\ref{finite gap potentials}.
In particular, Theorem~\ref{finite gap potential th characterization}  provides a full characterization of these potentials in the state space. 
They are
% called \textit{the finite gap potentials} for \eqref{CS-focusing} and are studied deeply  in Section~\ref{finite gap potentials}\,. In addition, as we shall prove in Theorem~\ref{finite gap potential th characterization}, these potentials are  
either 
 $u(x)=C\eee^{iNx}$\,, $C\in\R^*\,,$ $N\in\Nzero\,,$ or  rational functions
    \be\label{finite gap preuve trav waves}
     u(x)=\eee^{im_0\, x}\prod_{j=1}^r\left(\frac{\eix-\pjbar}{1-p_j\eix}\right)^{m_j-1}\left(\alpha+\sum_{j=1}^r\frac{\beta_j}{1-p_j\eix}\right), \quad p_j\in\D^*\,,\ p_k\neq p_j\,,\ k\neq j\,,  
    \ee
    where, for $N\in\N\,,$ $\,m_0\in\llbracket0,N-1\rrbracket\,,\,$ $m_1, \ldots,m_r\in\llbracket 1, N\rrbracket$\,, such that  
$
    m_0+\sum_{j=1}^rm_j=N\,,
$  and $(\alpha, \,\beta_1\,,\,\ldots\,,\,\beta_r)\in\C\times\C^{r}$ satisfy for all $j=1\,,\,\ldots\,,r\,,$
\bes
    \overline{\alpha}\,\beta_{j}\,+\,\sum_{k=1}^{r}\,\frac{\alpha_{j}\,\overline{\alpha_k}}{1-p_j\overline{p_k}}
            =m_j\,.
\ees
\vskip0.25cm
It remains to verify that \eqref{ex 1 trav waves foc} and \eqref{ex 2 trav waves foc} are traveling waves for \eqref{CS-focusing}\,. To do so, one can  simply substitute them into the equation \eqref{CS-focusing}--equation and check that they satisfy the equation.

\end{proof}

\begin{Rq}
    As was observed in the previous proof all the traveling waves $u_0(x-ct)$ of \eqref{CS-focusing} are either 
 $u(t,x)=C\eee^{iN(x-Nt)}$ or the rational functions $u(t,x):=u_0(x-ct)$ where $u_0$ is defined in \eqref{finite gap preuve trav waves} and the constants $\alpha, \beta_j$ and $c$ can be described by substituting $u$ in the \eqref{CS-focusing}--equation.
\end{Rq}

%----------------------------------------------
%----------------------------------------------

\subsection{The \texorpdfstring{$L^2$}{L2}--norm and the speed} \label{speed and L2 norm foc}

In this subsection, we analyze the $L^2$--norm and the speed of the traveling waves of \eqref{CS-focusing} and establish the existence of stationary solutions for the focusing Calogero-Sutherland DNLS equation \eqref{CS-focusing}\,.

\begin{prop}\label{prop speed travel waves foc}\hfill
\begin{enumerate}[label=(\roman*),itemsep=2pt]
    \item For any $r>0\,,$ there exists  a non--trivial traveling wave $ u(t,x):=u_0(x-ct)$ for \eqref{CS-focusing} with 
    \[
        \|u_0\|_{L^2}=r\,.
    \]
    % In other words, the traveling waves of \eqref{CS-focusing} can be arbitrarily small or large in $\Ltwo\,.$
    \item Let $u$ be a traveling wave for \eqref{CS-focusing} of the form \eqref{ex 1 trav waves foc}\,, then $u$ can propagate to the right or to the left with any speed $c\in\R\,.$ 
\end{enumerate}
\end{prop}

\begin{Rq} \label{Rq speed focusing L2 norm infinity}
    Contrary to the defocusing case, we do not necessarily have that the traveling wave propagate with a speed $c\to\infty$ when $\|u\|_{L^2}^2\to\infty$\,. 
    For instance, take 
    % Actually, for any $N\in\N$ there exists 
    $$
        u(t,x):= \frac{N}{\beta}-\frac{\beta}{1-\va{p}^2}+\frac{\beta}{1-p\eee^{iN(x-ct)}}\,,
            \qquad 
        \beta^2:=\frac{2N}{\frac{1+\va{p}^2}{1-\va{p}^2}-\frac{c}{N}}\,.
    $$ 
    % such that  this wave can propagate with any arbitrarily small or large speed $c$ in $\R$. 
    The proof of this statement will be achieved in the end of the following proof.
\end{Rq}

\begin{proof}[Proof of proposition~\ref{prop speed travel waves foc}]
\textbf{The $L^2$--norm .}
 Let $u$ be a traveling wave for \eqref{CS-focusing} of the form \eqref{ex 1 trav waves foc}\,,
\bes
% \label{trav wave CS+ particular case proof}
        u(t,x):= \eee^{i\theta}\left(\alpha+
        \frac{\beta}{1-p\eee^{iN(x-c t)}}\right)\,, \qquad p\in\D^*\,.
    \ees
    Our goal is to prove that the $L^2$ norm of these traveling waves can be arbitrary small or large.
    The computation of its $L^2$-norm has been performed in the proof of Proposition~\ref{norme L2 +c defo}. Therefore, by identity \eqref{u L2}, 
    \be \label{L2 norm u G2}
            \|u\|_{L^2}^2=\alpha^2+\alpha\beta
     +\alpha\beta +\frac{\beta^2}{1-\va{p}^2}  \;,
    \ee
    where the two reals $(\alpha,\beta)$ satisfies condition \eqref{alpha, beta-fo}
    \[
        \alpha\beta+\frac{\beta^2}{1-\va{p}^2} =N\,.
    \]
    That is, 
    \be\label{u L2 foc, |p|, beta}
        \|u\|_{L^2}^2
        =\frac{\va{p}^2}{1-\va{p}^2}\left(\frac{\beta^2}{1-\va{p}^2}-2N\right)+\frac{N^2}{\beta^2}\,.
    \ee
    Like for the defocusing case, $\|u\|_{L^2}^2$ is a continuous function of $\beta^2$ and $\va{p}^2\,.$ And by taking $\beta\to 0$ one has $\|u\|_{L^2}^2\to \infty\,.$ In addition, if $\va{p}\to 0$ then 
    $$
        \|u\|_{L^2}^2\underset{\va{p}^2\to0}{\sim}\; \frac{N^2}{\beta^2}\,.
    $$
    Hence, it is sufficient to take $\beta$ big enough so that $\|u\|_{L^2}$ can be arbitrary small.
%     At this stage, take for instance, $\beta=\delta(1-\va{p}^2)$ where $\delta>0$ is a parameter and $N=1\,$. Then by the latter condition, we infer
%     \[
%         \alpha=\frac{1}{\delta(1-\va{p}^2)}-\delta\,.
%     \]
%     Therefore, by \eqref{L2 norm u G2}
%     \begin{align*}
%         \|u\|_{L^2}^2
%         =&\,\left(\frac{1}{\delta(1-\va{p}^2)}-\delta\right)^2+1-\delta^2(1-\va{p}^2)+1\,.
%     \end{align*}
%     Observe that $\|u\|_{L^2}^2$ is a continuous function in $\va{p}^2\,,$ and has as infimum $\frac{1}{\delta^2}$ which is attained when $\va{p}^2\to0\,,$ and its supremum is $+\infty$ when $\va{p}^2\to 1\,.$
% By taking $\delta>\!>0\,,$ we infer that $\|u\|_{L^2}^2$ can be arbitrarily small.

% or 
%     \[
%         \|u\|_{L^2}^2=\alpha^2+\alpha\beta+N\,,
%     \]
% since for $u\in\G_2\,,$ the two reals $(\alpha,\beta)$ satisfies condition \eqref{alpha, beta-fo}\,.
\vskip0.25cm
\noindent
    \textbf{Speed : $\mathbf{c\in\R}\,$.} By Theorem~\ref{trav waves for foc}\,, there exists traveling waves for \eqref{CS-focusing} that propagates with a speed $c=-N\left(1+\frac{2\alpha}{\beta}\right)$ where $N\in\N$ and the two reals $(\alpha,\beta)$ satisfy 
    \[
        \alpha\beta +\frac{\beta^2}{1-\va{p}^2}=N\,,\qquad 0<\va{p}<1\,.
    \]
    That is 
    \begin{align}\label{speed trav waves foc}
        c=&\,-N\left(1+\frac{2N}{\beta^2}-\frac{2}{1-\va{p}^2}\right)\notag
        \\
        =&\,-N\left(-\frac{1+\va{p}^2}{1-\va{p}^2}+\frac{2N}{\beta^2}\right)\,.
    \end{align}
By taking, for example $\beta=\va{p}$\,, we infer 
\[
    c=N\,\frac{\va{p}^4+(2N+1)\va{p}^2-2N}{\va{p}^2(1-\va{p}^2)}
\]
Assume that $N=1\,,$ and by taking $x=\va{p}^2\in(0,1)$\,, we infer that  the continuous function
$$c(x):=\frac{x^2+3x-2}{x(1-x)}\,,$$
satisfies
$\displaystyle\inf_{x\in(0,1)}c(x)=-\infty$ and $\displaystyle
\sup_{x\in(0,1)}c(x)=+\infty\,.$

\vskip0.25cm
\noindent
\textbf{Proof of Remark~\ref{Rq speed focusing L2 norm infinity}\,.}
For a traveling wave $u$ of the form~\eqref{ex 1 trav waves foc}, 
\[
     u(t,x):= \alpha+
             \frac{\beta}{1-p\eee^{iN(x-c t)}}\,, \qquad \alpha\beta+\frac{\beta^2}{1-\va{p}^2}=N\,,
\]
where $N\in\N$\,, one has by \eqref{speed trav waves foc}\,, that $u$ propagates with a speed
\[
    c=-N\left(-\frac{1+\va{p}^2}{1-\va{p}^2}+\frac{2N}{\beta^2}\right)\,.
\]
Thus, for any $N\in\N$\,, let 
$$
    \beta:=\sqrt{\frac{2N}{\frac{1+\va{p}^2}{1-\va{p}^2}-\frac{\lambda}{N}}}\;, \qquad p\in\D\,,
$$ 
where $\lambda$ is a parameter in $\R\,,$  and with $\va{p}^2$ big enough so that $\beta$ is well defined.
% in the case where $\lambda>0$. 
Hence, one computes
\begin{align*}
    c=&\,-N\left(-\frac{1+\va{p}^2}{1-\va{p}^2}
    +\frac{2N}{\frac{2N}{\frac{1+\va{p}^2}{1-\va{p}^2}-\frac{\lambda}{N}}}\right)=\lambda\in\R\,.
\end{align*}
 That is, $u$ can propagate with any speed in $\R\,,$ regardless of the valued attained by the $L^2$--norm of $u$ 
% On the other hand, one can observe via \eqref{u L2 foc, |p|, beta}\,, that the $L^2$-norm can take various value when we change the values of $\va{p}$ and $\beta$\,.

% \begin{align*}
%     \|u\|_{L^2}^2\to N\,,
% \end{align*}
%  since 
% \[
%     \frac{\beta^2}{1-\va{p}^2}\sim N\,, \qquad \va{p}^2\to 1\,.
% \]
\end{proof}

% \vskip0.25cm
 % To end this section, we prove that the set  existence of the \textit{stationary wave solutions} $u(t,x)\coloneqq u_0(x)$ for the Calogero--Sutherland DNLS equation \eqref{CS-DNLS}.

\begin{corollary}\label{Stationary waves}
The  potentials
  $$
            u(t,x):=\eee^{i\theta} \sqrt{\frac{N(1-\va{p}^2)}{2(1+\va{p}^2)}}
            \left(1-\frac{2}{1-p\eee^{iNx}}\right) \,, \qquad p\in\D^*\,,\, N\in\N\,, \, \theta\in\T\,,
        $$
     are stationary solutions for \eqref{CS-focusing}\,.
  Conversely, the defocusing \eqref{CS-defocusing} equation does not exhibit stationary wave solutions except the complex constant functions.

\end{corollary}

\begin{proof}
Through a straightforward calculation, one can  easily check that the obtained waves satisfy the \eqref{CS-focusing}--equation. On the other side, for the defocusing equation, we already established via Remark~\ref{Rq speed defocusing} or the second point of Proposition~\ref{norme L2 +c defo}\,, the non--existence of stationary waves $u(t,x)=u_0(x)$ for \eqref{CS-defocusing}\,.
  
\end{proof}

%----------------------------------------------------------------------
%----------------------------------------------------------------------

\section{\textbf{The finite gap potentials}} \label{finite gap potentials}
This section aims to examine the finite gap potentials associated with the Calogero--Sutherland DNLS equation \eqref{CS-DNLS} in both the focusing and defocusing cases. Remarkably, these potentials manifest as rational functions containing the traveling and solitary waves of \eqref{CS-DNLS}.

\vskip0.25cm
In the following, we adopt a slight abuse of notation, where for all $n\in\N$, we denote
 \be\label{gap gamma n}
    \gamma_n(u):=\nu_{n}-\nu_{n-1}-1\,,
\ee
the \textit{gap} between the consecutive eigenvalues in the focusing context, and 
\[
    \gamma_n(u):=\la_{n}-\la_{n-1}-1\,,
\] as the gap in the defocusing context. At this point, several observations can be made. First, recall that in the defocusing case, the $(\la_n)$ satisfies inequality~\eqref{simplicite val prop Lu tilde}, and thus, for all $n\in\N$\,, $\gamma_n(u)$ is non--negative in the defocusing case\,. Second, notice that since the eigenvalues $(\nu_n)$ and $(\la_n)$ of the Lax operators $L_u$ and $\Lutilde$ are invariant by the evolution, then for all $n\in \N\,,$
\[
    \gamma_n(u(t))=\gamma_n(u_0)\,, \qquad \qlq t\,.
\]
\begin{defi}[Finite gap potential]\label{finite gap potential}
A function $u\in L^2_+(\T)$ is said to be a {finite gap potential} of \eqref{CS-DNLS} if there exists $m\in \N$ such that 
    \be
        \g_n(u)=0\,,\qquad \qlq n\geq m\,,
    \ee
    where $\gamma_n$ is defined in \eqref{gap gamma n}\,.
\end{defi}

\vskip0.25cm
Recall that any function in the Hardy space $\Ltwo$ can be seen as a holomorphic function on the unit disc $\D$ whose trace on the boundary $\p \D$ is in $L^2\,.$ Hence, in what follows, we denote by $\mathscr{B}_N$ the set of finite Blaschke products of degree~$N$ :
% ~\footnote{~The $N$ corresponds to the degree of the numerator of $\uppsi$ }:
\[
     \uppsi(x)=\eee^{i\theta}\,\prod_{k=1}^N\frac{\eee^{ix}-\pkbar}{1-p_k\eee^{ix}}\;, \qquad  \theta\in \R\,, \; p_k\in \D\,,
\]
which can be identified as the set of functions 
$$
    \uppsi(z)=\eee^{i\theta}\,\frac{z^N \bar{Q}(\frac{1}{z})}{Q(z)}\,, \qquad z\in \overline{\D}:=\lracc{\va{z}\leq 1}\,,\; \theta\in \R\,.
$$
where $$
    Q(z):=\prod_{j=1}^N (1-p_jz)\,, \qquad p_j\in \D\,.
$$
In other words, $z^N\bar{Q}(\frac{1}{z})$ is a Schur polynomial~\footnote{~A polynomial $q(z)=\sum_{k=1}^N a_kz^k$ is  called a Schur polynomial if all its roots are in the open unit disc $\D\,.$ \label{schur}} of degree $N$\,.

 \begin{Rq}
 By convention, we suppose that a finite Blaschke product of degree $0$ is a constant in $\C$.
 \end{Rq}

\begin{prop}
\label{produit-de-Blashke}
     Let $u$ be a finite gap potential of \eqref{CS-focusing}\,. There exist $(\nu,\uppsi)\in\R\times \mathscr{B}_n\,,$ $n\in \Nzero$ such that
 \be\label{Lu S=SLu sans second terme}
     L_u\,S^k\uppsi=(\nu+k)\, S^k\uppsi\,,
     \qquad 
     \qlq k\in\Nzero\,.
 \ee 
 In addition, the same goes for the defocusing Calogero--Sutherland DNLS equation~\eqref{CS-defocusing}.
\end{prop}

\begin{proof}
   Let $u$ be a finite gap potential, that is $\nu_{n}=\nu_{n-1}+1$ for all $n\geq m\,.$  We denote by $n_0$ the eventual indices where  $\nu_{n_0}$ may vanish. Then, by  
Proposition~\ref{gap=0 and eigenspaces}\,, 
\[
    Sf_{n-1}\parallelsum f_{n}\,, \qquad \qlq n\geq \,N\coloneqq\max\{m, n_0\}+2\,,
\]
as the eigenvalues $(\nu_n)$ are simple for $n\geq m+1\,.$
Therefore,  letting  $\uppsi:=f_{N-1}\,,$ we have
% by applying the commutator identity~\eqref{commutateur [Lu,s]}\,, we obtain
\be\label{LuSkB}
    L_u\,S^k\uppsi=(\nu_{N-1}+k) S^k\uppsi\qquad \qlq\,k\in\Nzero\,.
\ee 
It remains to prove that $\uppsi$ is a finite Blaschke product. Observe that, by
taking the inner product of both sides of the previous identity with $\uppsi\,,$ 
\be\label{S^kfNfN}
    \ps{S^k\uppsi}{\uppsi}=0\, ,\qquad \qlq k\in \N\,  .
\ee
That is 
$$
    \ps{\,\va{\uppsi}^2}{\eee^{ikx}}=0, \quad\qlq k\in\N\,.
$$
or 
$$
    \ps{\,\va{\uppsi}^2}{\eee^{ikx}}=0, \quad\qlq k\in\Z\bk\{0\}\,,
$$
as  $\va{\uppsi}^2$ is real value.
Consequently, $|\uppsi|^2$ is a real constant, which can be supposed equal to $1$ since we have assumed that the eigenfunctions of $L_u$ constitute an orthonormal basis of $L^2_+(\T)$\,. Thus, $\va{\uppsi}=1 $ on $\T$.
% since $(S^k\uppsi)_{k\geq 0}$\, are eigenfunctions of the self--adjoint operator $L_u$ corresponding to the different eigenvalues $(\la+k)_{k\geq 0}$. 
% This means, $\uppsi$ is an \textit{inner function}~\footnote{A bounded analytic function $\uppsi$ on $\D$ is said to be \textit{inner} if $\va{\uppsi(\eee^{ix})}=1$ for almost every $x\,.$ Note that a Blaschke product is a  rational inner function.}.
 In order to conclude, we need the following lemma.

\begin{lemma}
\label{produit_de_Blaschke}
    Let $\upphi$ be an
    % non-constant
    analytic function on the open unit ball that extends continuously
    %as a holomorphic function to the closed unit disc,  and it satisfies $|\upphi(z)|=1$ on $\T:=\lracc{z=1}$\,. 
    to an inner function\footnote{A bounded analytic function $\uppsi$ on $\D$ is said to be \textit{inner} if $\va{\uppsi(\eee^{ix})}=1$ for almost every $x\,.$ Note that a Blaschke product is a  rational inner function.} on the closed unit disc. Then $\upphi\in\mathscr{B}_n\,$.
    % is  a Blaschke product, $$\upphi(z)=\eee^{i\theta}\prod_{k=1}^{n} \frac{z-\pkbar}{1-p_kz}\ ,\quad \theta\in\R\,,\ n\in\N\,$$ and $\overline{p_k}$ are the zeros of $\upphi$ counted with multiplicity and satisfying $ |p_k|<1$\,.
\end{lemma}

\begin{proof}
Given a holomorphic function $\upphi$ on the open unit ball that extends continuously to the unit circle while satisfying $\va{\uppsi}=1>0$ on $\T$ , we know that its zeros are finite, isolated and all localized  inside the open unit disk $\D$.
We denote them by $\overline{p_1},\ldots, \overline{p_n}$ . Hence,
%there exists  a holomorphic function  such that the function 
$\upphi$ can be factorized as 
$$
\upphi(z)=\upsilon(z)\cdot\prod_{k=1}^{n} \frac{z-\pkbar}{1-p_kz}\, ,
$$
where $\upsilon$ is a holomorphic function without zeros on $\D$\,.
Therefore, $1/\upsilon$  is a holomorphic function on $\D$\,, which continuously extends to the unit circle while satisfying $\va{1/\upsilon}=1$ on $\T$\,. 
%Indeed,
%  $$
%     \displaystyle\left|\frac{z-\pkbar}{1-p_kz}\right|=1\,,\quad \qlq\,k=1\,,\,\ldots\,,n$$   and $|\upphi|=1$ on $\T$ then $\va{\upsilon}=1$ on the unit circle.
Thus, by the maximum principle, we infer that $\left|1/\upsilon\right|\leq1$ on $\D$\,.
Using the same argument on $\upsilon $ instead of $1/\upsilon$ , we deduce that $\va{\upsilon}\leq 1$ on the unit disc. As a consequence, $\va{\upsilon}=1$ on the close unit disc $\lracc{|z|\leq 1}$\, and so 
$$
\upphi(z)=\eee^{i\theta}\cdot\prod_{k=1}^{n} \frac{z-\pkbar}{1-p_kz}\, , \qquad \theta\in\R\,.
$$
\end{proof}
% By an abuse of notation, we define from $\uppsi\in H^{1/2}_+(\T)$, the new function $\uppsi$ on the open unit disc $\D$ using the Poisson kernel 
% $$
%     \uppsi(z):=\int_{\zeta\in\mathbb{S}^1}\frac{1-\va{z}^2}{\va{z-\zeta}^2}\,\uppsi(\zeta) \,\dz\,,\quad z\in \D\,.
% $$
% Hence, the new function $\uppsi$ is a harmonic function on $\D$. Therefore, $\mathrm{Re}(\uppsi)$ and $\mrm{Im}(\uppsi)$ are holomorphic functions
Coming back to the proof of Proposition~\ref{produit-de-Blashke}, we denote $\underline{\uppsi}$ the function obtained by the isometric isomorphism map 
$$
    \underline{\uppsi}(z)=\sum_{k\geq 0}\fr{\uppsi}(k)z^k\,, \;\; z\in\D \quad\longmapsto\quad \uppsi(x):=\sum_{k\geq 0}\fr{\uppsi}(k)\eee^{ikx}\,,\ \ x\in \T\,,
$$ 
In particular, since $\uppsi\in L^2_+(\T)$ then $\underline{\uppsi}\in \mathbb H_2(\mathbb D)\,,$ where 
\[
    \mathbb{H}_2(\mathbb D):=\lracc{u\in \mathrm{Hol}(\D)\,;\,\sup_{0\leq r<1}\int_0^{2\pi} \,\va{u(r\eee^{i\theta})}^2\;\dfrac{d\theta}{2\pi}<\infty}\,,
\]
Hence, by \cite[Theorem 4.5.3]{Ch},
$$
    \underline{\uppsi}(r\eee^{ix})=\frac{1}{2\pi}\int_0^{2\pi}P_r(x-\theta)\uppsi(\eee^{i\theta})\,dt\,, \quad 0\leq r<1\,,
$$
where $P_r$ denotes the Poisson Kernel
$$
    P_r(x-\theta)=\frac{1-r^2}{1-2r\cos(x-\theta)+r^2}\;.
$$
Note that the function
$\uppsi\in \DomLu:= H^{1}_+(\T)$ is continuous on $\T$ as
\be\label{in L1}
    D\uppsi=L_u\uppsi+u\Pi(\bar{u}\uppsi)\in L^1_+(\T)\,.
\ee
% thanks to Lemma~\ref{Tubarf}\,.
Therefore, the Poisson Theorem \cite[Theor\`eme 30]{Gi04}  implies that the holomorphic function $\underline{\uppsi}(r\eee^{ix})$ extends continuously to $\uppsi(\eee^{ix})$ as $r\to1\,.$ In addition, recall that $\va{\uppsi}=1$ on $\T$\,. Thus, applying the previous lemma, we infer that $\underline{\uppsi}$ is a Blaschke product and so is $\uppsi\,$.

\end{proof}

\vskip0.25cm
 At this stage, we aim to characterize the finite gap potentials of \eqref{CS-DNLS}\,. To this end, we regroup them according to the following procedure : for any finite gap potential $u$ of \eqref{CS-DNLS}\,, we denote by  $\mathcal{N}(u)$ the non--negative integer
\be\label{N(u)}
    \mathcal{N}(u):=\min\lracc{\,n\in\Nzero\;\mid\,\e\, \uppsi\in \mathscr{B}_n\;,\;L_u\,S^k\uppsi=(\nu+k) S^k\uppsi\,,\, \qlq k\geq 0 }\,,
\ee
and we define, for $N\in \Nzero$\,, the set
$$
     \boxed{\ \ \mathcal{U}_N:=\lracc{u \text{ finite gap potential }, \; \mathcal{N}(u)={N}}\,.\ }  
$$
This means that for any $u\in \mathcal{U}_N\,,$ there exists a finite Blaschke product $\uppsi_u$ of minimal degree $N\,,$  satisfying 
\be\label{LuSk uppsi[u]}
    L_u\,S^k\uppsi_u=(\nu_u+k) S^k\uppsi_u\,, \qquad \qlq k\in\Nzero\,,
\ee
where $\nu_u$ is the corresponding eigenvalue of $\uppsi_u$\,. That is $\lracc{S^k\uppsi_u\, \mid\,k\in\Nzero}$ are parts of the orthonormal basis of $\Ltwo\,.$ Besides, observe that,
since $\deg\uppsi_u=N\,,$ then there exists $N$ eigenfunctions $f_0,\ldots f_{N-1}$ of $L_u$ that generate the model space $(\uppsi_u\, L^2_+)^{\perp}\,$ which is of dimension $N$ \cite[Corollary 5.18]{GMR16}\,. We denote $\nu_0\,,\,\ldots\,,\,\nu_{N-1}$ the associated eigenvalues. Note that the latter $N$ eigenvalues  are not necessarily smaller than $\nu_u\,.$ 
% and can be anywhere on $\R$. 
We  summarize this discussion by the following diagram. For any $u\in\U_N\,,$
% $
%     \dim \big((\uppsi_u L^2_+)^{\perp}\big)=N\,.
% $ 

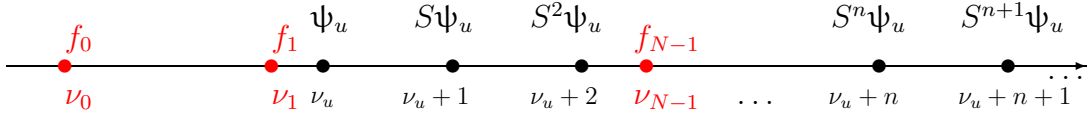
\setlength{\unitlength}{1.7cm}
\begin{picture}(20, 1.3) 
    
    \put(-0.25, 0.5){\vector(1, 0){14.2cm}}
    \put(2.1,0.21){{\scalebox{0.8}{$\nu_u$}}}
    \put(2.1,0.8){{$\uppsi_u$}}
    \put(2.2,0.5){{\circle*{0.1}}}

    \put(2.81,0.2){{\scalebox{0.8}{$\nu_u+1$}}}
    \put(2.9,0.8){{$S\uppsi_u$}}
    \put(3.2,0.5){{\circle*{0.1}}}
    
    \put(4.2,0.5){{\circle*{0.1}}}
    \put(3.8,0.2){{\scalebox{0.8}{$\nu_u+2$}}}
    \put(3.8,0.8){{$S^2\uppsi_u$}}
    
    % \put(5,0.5){\textcolor{blue}{\circle*{0.1}}}
    
    \put(5.4,0.2){{$\ldots$}}
    
    \put(6.11,0.2){{\scalebox{0.8}{$\nu_u+n$}}}
    \put(6.13,0.8){{$S^n\uppsi_u$}}
    \put(6.5,0.5){{\circle*{0.1}}}
    
    \put(7.11,0.2){{\scalebox{0.8}{$\nu_u+n+1$}}}
    \put(7.13,0.8){{$S^{n+1}\uppsi_u$}}
    \put(7.5,0.5){{\circle*{0.1}}}

     \put(7.8,0.4){{$\ldots$}}
    
    \put(0.2,0.2){\textcolor{red}{$\nu_0$}}
    \put(0.2,0.65){\textcolor{red}{$f_0$}}
    \put(0.2,0.5){\textcolor{red}{\circle*{0.1}}}
    
    \put(1.8,0.2){\textcolor{red}{$\nu_1$}}
    \put(1.8,0.65){\textcolor{red}{$f_1$}}
    \put(1.8,0.5){\textcolor{red}{\circle*{0.1}}}
    
    \put(4.6,0.2){\textcolor{red}{$\nu_{N-1}$}}
    \put(4.6,0.65){\textcolor{red}{$f_{N-1}$}}
    \put(4.7,0.5){\textcolor{red}{\circle*{0.1}}}
\end{picture}
\vskip0.5cm
\noindent
% Note that
% by moving to the defocusing case, we are sure that the remaining 
Of course,
the same goes for the defocusing equation  with $\Lutilde$ instead of $\Lu\,,$ up to the fact that the remaining $N$ eigenvalues $\nu_0\,,\ldots\,\nu_{N-1}$ are necessarily smaller than 
$\nu_u\,,$ since the eigenvalues of $\Lutilde$ satisfy the property~\eqref{simplicite val prop Lu tilde}\,.
Besides, note that by taking the minimum in \eqref{N(u)} we guarantee that:
\begin{enumerate}
    \item If $u\in\lracc{v \text{ finite gap potential }, \; L_v\,S^k\uppsi=(\nu+k) \,S^k\uppsi\,,\; \uppsi\in \mathscr{B}_N}\,,$ then 
    $u\notin\lracc{v \text{ finite gap potential }, \; L_v\,S^k\uppsi=(\nu+k) \,S^k\uppsi\,,\; \uppsi\in \mathscr{B}_{N-1}}\,.$ 
    % (See the proof of the next Theorem). 
    \item The set $\U_N$ is invariant under the evolution of \eqref{CS-DNLS}\,. (See Proposition \ref{U[N] invariant by the evolution}). 
\end{enumerate}

\vskip0.25cm
The following theorem aims to characterize the finite gap potentials of the Calogero--Sutherland DNLS \eqref{CS-DNLS} in the state space.
\begin{theorem}\label{finite gap potential th characterization Un}
    Let $N\in\N\,.$ A potential $u$ is in $\U_{N}$ 
    % be a finite gap potential for \eqref{CS-DNLS} in $\,\U_{N}$. Then 
    if and only if $u(x)=C\eee^{iNx}$\,, $C\in\C^*\,,$ or $u$ is a rational function 
    \be
    \label{u rational function th characterization of finite gap}
        u(x)=\eee^{im_0\, x}\prod_{j=1}^r\left(\frac{\eix-\pjbar}{1-p_j\eix}\right)^{m_j-1}\left(a+\sum_{j=1}^r\frac{c_j}{1-p_j\eix}\right)\,, \quad p_j\in\D^*\,,\ p_k\neq p_j\,,\ k\neq j\,, 
    \ee
    where $\,m_0\in\llbracket0,N-1\rrbracket\,,\,$ $m_1, \ldots,m_r\in\llbracket 1, N\rrbracket$\,, such that  
$
    m_0+\sum_{j=1}^rm_j=N\,,
$  and $(a, \,c_1\,,\,\ldots\,,\,c_{r})\in\C\times\C^{r}$ satisfy for all $j=1\,,\,\ldots\,,N-m\,,$
    \begin{enumerate}[label=(\roman*),itemsep=2pt]
        \item In the focusing case, 
        \be\label{cond a, ck-foc}
            \overline{a}\,c_{j}\,+\,\sum_{k=1}^{r}\,\frac{c_{j}\,\overline{c_k}}{1-p_j\overline{p_k}}=m_j\,,
        \ee
        \item In the defocusing case,
        \be\label{cond a, ck-defoc}
            \overline{a}\,c_{j}\,+\,\sum_{k=1}^{r}\,\frac{c_{j}\,\overline{c_k}}{1-p_j\overline{p_k}}=-m_j\,,
        \ee
    \end{enumerate}
    with $a\neq 0$ if $m_0\in\llbracket1,N-1\rrbracket\,.$
    % Furthermore,
    Besides, if $N=0\,,$ then $u$ is a complex constant function.
% \[
%     \mathcal{U}_0=\lracc{u=\eee^{i\theta}\sqrt{-\la_u}\;;\; \theta\in\R\,, \la_u\leq0}.
% \]
\end{theorem}

\begin{Rq}\label{rq valeur propre finite gap}
    As we shall see in Step 4 of the proof of Theorem~\ref{finite gap potential th characterization Un}, if $u\in\U_N\,,$ then  the eigenvalue of $L_u$ associated 
    \begin{itemize}[itemsep=5pt]
        \item with the Blaschke product $\uppsi_u=\eee^{i\theta} \eee^{iNx}$ if $u=C\eee^{iNx}\,,$ is given by 
            \begin{enumerate}[label=(\roman*),itemsep=2pt]
                \item $\nu_u=N-C^2$ in the focusing case.
                \item  $\la_u=N+C^2
            $ in the defocusing case.
            \end{enumerate}
        \item with the Blaschke product
        $$
            \uppsi_u=\eee^{i\theta}\,\eee^{im_0\cdot x}\prod_{j=1}^{r}\left(\frac{\eee^{ix}-\pjbar}{1-p_j\eee^{ix}}\right)^{m_j}\,,\quad \theta\in\R\,, p_j\neq p_k\,, j\neq k\,,
        $$ if $u$ is the rational function \eqref{u rational function th characterization of finite gap}\,, is given by 
            \begin{enumerate}[label=(\roman*),itemsep=2pt]
                \item \label{nu[u]}$\displaystyle\nu_u=m_0-\va{a}^2-
            \sum_{j=1}^{r} \,a\,\overline{c_j}$ in the focusing case.
                \item \label{lambda[u]}$\displaystyle\la_u=m_0+\va{a}^2+
            \sum_{j=1}^{r} \,a\,\overline{c_j}\,
            $ in the defocusing case.
            \end{enumerate}
    \end{itemize}
\end{Rq}

\vskip0.25cm

\vskip0.25cm
In order to establish this theorem, we recall a specific case of formula~\eqref{formule inversion spectrale u0}\,.
\begin{Rq} \label{formule d'inversion f_n} Let $(f_n)$ be an orthonormal basis of $\Ltwo\,.$
    For any $n\geq 0$\,,   
$$  
    f_n(z)=\left\langle(\Id-z M)^{-1} \mathds{1}_{n} \mid Y\right\rangle_{\ell^{2}}\,, \qquad z\in\D\,,
$$
where $\mathds{1}_n$ and $Y$ are the column vectors 
$$
\mathds{1}_{n}:=\left(\delta_{p n}\right)_{  p \geq 0}\ ,\qquad  Y:=\left(\left\langle 1 \mid f_{m}\right\rangle\right)_{m \geq 0}
$$ 
and $M$ is the  matrix representation of the operator $S^{*} $ in the $\left(f_m\right)$--basis 
$$
    M=\left(M_{mp}\right)_{mp \geq 0}, \qquad M_{mp } =\left\langle f_{p} \mid S f_{m}\right\rangle\ .
$$
\end{Rq}

\textit{In what follows, we denote by $\C_{\leq N}[X]$ the set of polynomials $P$ in complex coefficients with degree at most $N$ and by $\C_N[X]$ those of degree $N\,.$}

\begin{proof}[Proof of Theorem \ref{finite gap potential th characterization Un}]
We present the proof for the focusing case. Note that the same arguments can be performed to deduce the result in the defocusing case.
    The key ingredient is the inversion spectral formula \eqref{formule inversion spectrale u0}
    \be\label{u-formule-inversion}
         u(z)=\ps{(\Id-z M)^{-1}\,X}{Y}\,, 
    \ee
    where $X\,,\, Y$ and $M$  are defined in \eqref{X,Y,M}\,.   
    The proof will be split in 5 steps. 

    \vskip0.3cm
    
Let $u\in \mathcal{U}_N\,,$ then by Proposition~\ref{produit-de-Blashke} there exists a finite Blaschke product 
\be\label{Blaschke product Un}
    \uppsi_u(z)=\eee^{i\theta}\;\frac{\eee^{iNx}\bar{Q}(1/z)}{Q(z)}\;; \quad Q(z):=\prod_{j=1}^N (1-p_jz)\,, \qquad p_j\in \D\,,
\ee
such that \eqref{Lu S=SLu sans second terme} is satisfied
\[
     L_u\,S^k\uppsi_u=(\nu_u+k)\, S^k\uppsi_u\,,
     \qquad 
     \qlq k\in\Nzero\,.
\]
\vskip0.2cm
\noindent
\textbf{Step 1.}
As a first step, we prove that any $u\in\U_N$ must be a rational function
% ~\footnote{The two polynomials $P$ and $Q$ as written in Step~1 does not necessarily satisfy $P\wedge Q=1$}
\[
    u(z)=\frac{P(z)}{Q(z)}\,, \qquad P\in\C_{\leq N}[z]\,, 
    % \quad Q(z)=\prod_{k=1}^N (1-p_kz)\,, \ p_k\in\D\,,
\]
where $Q(z)$ is the same denominator of the Blaschke product $\uppsi_u(z)$ associated with $u\in\U_N\,.$
Indeed, first observe that  combining  \eqref{Lu S=SLu sans second terme} with the commutator identity \eqref{commutateur [Lu,s]}
leads to 
\be\label{Xn}
    \ps{u}{S^k\uppsi_u}=0\,, \qquad\qlq k\geq 1\, .
\ee
Hence, we infer thanks to Lemma~\ref{relation between <Sf|f> <u|f> and <1|f>}\,, that the infinite matrices $M\,,\,X\,,$ and $Y$ of \eqref{u-formule-inversion} written in the basis $(f_k)_{k=0}^{N-1}\cup (S^k\uppsi_u)_{k\geq 0}$ are of the form
$$
    M= 
    \left(\begin{array}{cccc|cccc}
    \ps{f_0}{S f_0} & \ldots& \ps{f_{n-1}}{S f_0} & \ps{\uppsi_u}{S f_0}&&&
    \\
    \vdots&  & \vdots&\vdots&&&
    \\ 
    \langle f_0\mid S f_{N-1}\rangle & \ldots &\ps{f_{N-1}}{S f_{N-1}} &\ps{\uppsi_u}{S f_{N-1}}& &&\\
    0&\ldots&\ldots&0&1
    \\\hline
    0& \ldots&\ldots&0&0&1\\
    \vdots&&&&\vdots&\ddots&\ddots
    \end{array}\right)\ ,\; n\in \N
$$
$$
    X=\left(\begin{array}{c}
    \ps{u}{f_0}\\
    \vdots\\
    \ps{u}{f_{N-1}}\\
    \ps{u}{\uppsi_u}\\
    0\\
    \vdots
    \end{array}\right)
\,,
\quad
    Y=\left(\begin{array}{c}
    \ps{1}{f_0}\\
    \vdots\\
    \ps{1}{f_{N-1}}\\
    \ps{1}{\uppsi_u}\\
    0\\
    \vdots
    \end{array}\right)\,.
$$    
Therefore,  following the same procedure presented in the proof of Theorem~\ref{traveling waves for the defocusing CS}\,, one can observe that the infinite matrices $M\,, \,X$ and $Y$ can be reduced to finite matrices that involve only the first $N+1$ coordinates of each of these matrices. That is,
\[
    u(z)=\ps{\left(\Id-zM_{\leq N}\right)^{-1}X_{\leq N}}{Y_{\leq N}}_{\C^{N+1}\times\C^{N+1}}\,,
\]
where $M_{\leq N}=(M_{mn})_{0\leq m,n\leq N}\,,$ $\,X_{\leq N}=(X_n)_{\,0\leq n\leq N}$ and $\,Y_{\leq N}=(Y_n)_{\,0\leq n\leq N}\,.$ 
As a consequence, $u$ is a rational function 
$$
   u(z)=\frac{P(z)}{\det\left(\Id-zM_{\leq N}\right)}\;,\quad P\in \mathbb{C}_{\leq N}[z]\,.
$$
Note that  $\det\left(\Id-zM_{\leq N}\right)$ coincides with the denominator of the eigenfunction 
$$
    \uppsi_u=\eee^{i\theta}\,\frac{z^{N}\bar{Q}(1/z)}{Q(z)}\,,
$$
since by Remark~\ref{formule d'inversion f_n}\,, $\,\uppsi_u$ is also expressed via the inversion spectral formula 
$$
    \uppsi_u(z)
    =\left\langle(\Id-z M_{\leq N})^{-1} \mathds{1}_{N} \mid Y_{\leq N}\right\rangle_{\C^{N+1}\times \C^{N+1}}
    =\frac{\ps{\mathrm{Com}(\Id-zM_{\leq N})^T\mathds{1}_{N}}{Y_{\leq N}}}{\det\left(\Id-zM_{\leq N}\right)}\,,
$$
and hence $\det\left(\Id-zM_{\leq N}\right)=Q(z)\,.$ 
Thus,  
\be\label{u=P/Q, Q [pkin D]}
    u(z)=\frac{P(z)}{Q(z)}\,, \qquad P\in\C_{\leq N}[z]\,, \quad Q(z)=\prod_{k=1}^N (1-p_kz)\,, \ p_k\in\D\,.
\ee

\vskip0.25cm\noindent
\textbf{Step~2.} In this step, we prove that 
if $u\in\U_N$ then 
% the rational function $u$
% obtained in the last step satisfies
\be\label{|u|2=zpzlog uppsi[u]-lambda[u]}
    \va{u}^2=z\p_z\log \uppsi_u-\nu_u\,, \qquad \text{ on } \p\D\,.
\ee
Indeed, recall that $L_u\uppsi_u=\nu_u \uppsi_u$\,. Then by definition of $L_u=z\p_z-T_u \Tubar\,,$
\be\label{Lu uppsi[u]=lambda uppsi[u]}
    z\p_z\uppsi_u-u\Pi(\overline{u}\,\uppsi_u)=\nu_u \uppsi_u\,.
\ee
On $\p\D\,,$ 
$$
    \overline{u}\,\uppsi_u(z)
    =\eee^{i\theta}\,\frac{z^N\,\bar{P}(1/z)}{z^N\,\bar{Q}(1/z)}\cdot \frac{z^{N}\bar{Q}(1/z)}{Q(z)}
    =\eee^{i\theta}\,\frac{z^N\,\overline{P}(1/z)}{Q(z)}\,,
$$
 extends as a holomorphic function on $\D\,.$ Hence, 
$\Pi(\overline{u}\,\uppsi_u)=\overline{u}\,\uppsi_u\,, $ 
and so identity \eqref{Lu uppsi[u]=lambda uppsi[u]} can be read as
$$
    \frac{z\p_z\uppsi_u}{\uppsi_u}=|u|^2+\nu_u\,,
$$
implying that identity~\eqref{|u|2=zpzlog uppsi[u]-lambda[u]} holds.

\vskip0.25cm\noindent
\textbf{Step~3.} In this step, we prove that  the rational function $u$ obtained in Step 1 can be rewritten either as $u(z)=Cz^N\,,$ $C\in\C^*\,,$ or
\[
    u(z)=z^{m_0}\prod_{j=1}^r\left(\frac{z-\pjbar}{1-p_jz}\right)^{m_j-1}\frac{q(z)}{\prod_{j=1}^r(1-p_jz)}\,, \qquad p_j\in\D^*\,,\ p_k\neq p_j\,,\ k\neq j\,,
\]
where  $m_0\in\llbracket0,N-1\rrbracket\,,\,$ $m_1, \ldots,m_r\in\llbracket 1, N\rrbracket$\,,  
$$
    m_0+\sum_{j=1}^rm_j=N\,,
$$ 
and such that  $\deg(q)=r$ if $m_0\neq 0\,.$ 
Indeed, we write
\eqref{Blaschke product Un} 
as $\uppsi_u=\eee^{i\theta}z^{N}$ (if all the $p_k$ in \eqref{Blaschke product Un} vanish), or
% \footnote{~Observe that, contrary to formula~\eqref{u=P/Q, Q [pkin D]} the $p_k\neq0$.}
\be\label{Blaschke m0}
    \uppsi_u=\eee^{i\theta}z^{m_0}\prod_{j=1}^r\left(\frac{z-\pjbar}{1-p_jz}\right)^{m_j}\,, \qquad  p_j\in\D^*\,,\ p_k\neq p_j\,,\ k\neq j\,, \ \theta\in\R
\ee
where  $m_0\in\llbracket0,N-1\rrbracket\,,\,$ $m_1, \ldots,m_r\in\llbracket 1, N\rrbracket$\,, such that  $m_0+\sum_{j=1}^rm_j=N$\,. 
As a first point, we prove when $m_0\geq 1\,,$ then the numerator $P$ of $u$ can be factorized as $P(z)=z^{m_0}P_{N-m_0}(z)$ with $P_{N-m_0}\in\C_{N-m_0}[z]\,.$ 
% with $\deg (P)=N$ if $m_0\neq 0$\,. 
% we prove that the rational function 
% \[
%     u=\frac{P(z)}{\prod_{j=1}^r(1-p_jz)^{m_j}}
% \]
Let $m_0\geq 1$ then 
 $\ps{u}{\uppsi_u}\neq 0\,$, because otherwise there exists a Balschke product $\upphi_u=S^*\uppsi_u$ of degree $N-1\,,$
 such that  by the commutator identity~\eqref{commutateur [Lu,s]}\,,
\[
    L_u S^k\upphi_u =(\nu_k-1+k)S^k\upphi_u\,, \qquad k\geq 0\,,
\]
% leading to $u\in \U_{N-1}\,,$ 
meaning that $u\in \U_{N-1}\,,$ which is a contradiction with the fact that $u\in\U_N\,$. Hence, $\ps{u}{\uppsi_u}\neq 0\,. $ This leads  to
\begin{enumerate}[label=(\roman*)]
    \item The numerator $P$ of $u(z)$  must be degree $N\,.$ \label{deg P=N}
    \item $\ps{1}{u}=0$\label{moyenne nulle Un}\,.
\end{enumerate}
Indeed, for \ref{deg P=N}\,, it is sufficient to note that 
\[
    0\neq\ps{\uppsi_u}{u}=\intc \eee^{i\theta}\,\frac{z^{N}\bar{Q}(1/z)}{Q(z)}\, \frac{\bar{P}(1/z)}{\bar{Q}(1/z)}\, \dz
    =
    \eee^{i\theta}\,z^N\bar{P}(1/z)_{\,|z=0}
\]
For \ref{moyenne nulle Un}, observe by Lemma~\ref{relation between <Sf|f> <u|f> and <1|f>}\,, 
\[
    \ps{1}{u}\underbrace{\ps{u}{\uppsi_u}}_{\neq 0}=-\nu_u\ps{1}{\uppsi_u}\,,
\]
where the right--hand side vanishes since $\uppsi_u=S\upphi$\,, $\upphi\in\Ltwo$ for $m_0\geq1\,.$
Therefore, if $m_0=1\,,$ then by  \ref{deg P=N} and \ref{moyenne nulle Un},
% there exists $v\in\Ltwo$ such that  $u=Sv\,.$ 
% If $m=1\,,$ this means that 
\[
    u(z)=\frac{z\,P_{N-1}(z)}{\prod_{j=1}^{r}(1-p_jz)^{m_j}}\,, \qquad P_{N-1}\in\C_{N-1}[z]\,,\quad  p_k\in\D^*\,, \ p_j\neq p_k\,.
\]
Now, if $m_0=2\,,$ we have by \ref{moyenne nulle Un} $\ps{u}{1}=0\,$, that is $u=Sv$ with $v\in\Ltwo\,.$ Thus, by the definition of $L_u=z\p_z-u\Pi(\bar{u} \, \cdot)$\,,
\[
    L_u\,z= z- \ps{1}{v}u\,.
\]
Taking the inner product of the latter identity with $\uppsi_u\,,$
\[
    (\nu_u-1)\ps{z}{\uppsi_u}=-\ps{1}{v}\underbrace{\ps{u}{\uppsi_u}}_{\neq\, 0}\,.
\]
Note that for $m_0= 2$ we have $\ps{z}{\uppsi_u}=0\,.$
This implies that $\ps{1}{v}=0$ 
% as $\ps{u}{\uppsi_u} $  cannot vanish for $u\in\U_N\,.$ That is, if $m=2\,,$
leading to $u=S^2w$ with $w\in\Ltwo$\,. 
Therefore, if $m_0=2\,,$ then $u$ can be decomposed as
 \[
    u(z)=\frac{z^2\,P_{N-2}\,(z)}{\prod_{j=1}^{r}(1-p_jz)^{m_j}}\,, \qquad P_{N-2}\in\C_{N-2}[z]\,,\quad  p_k\in\D^*\,, \ p_j\neq p_k\,.
\]
% that is 
% \[
%     u(z)=z^2\left(a+\sum_{k=1}^{\ell}\sum_{j=1}^{n_k} \frac{c_{j,k}}{(1-p_kz)^j}\right)\,, \qquad p_k\in\D^*\,,\ a\neq0\,, \ \ \sum_{k=1}^\ell n_k=N-2\,,
% \]
Now, if $m_0=3$ then by repeating the same above procedure and  taking the inner product of 
\[
    L_uz^2=2z^2-\ps{1}{w} u
\]
with $\uppsi_u\,,$ one obtains
\[
    (\nu_u-2)\underbrace{\ps{z^2}{\uppsi_u}}_{=0}=-\ps{1}{w}\underbrace{\ps{u}{\uppsi_u}}_{\neq \,0}\,.
\]
That is, $\ps{1}{w}=0$ i.e, $u=S^3\underline{w}$ with $\underline{w}\in\Ltwo\,$...
Therefore, for all $m\in\llbracket0,N-1\rrbracket\,,$
 \[
    u(z)=\frac{z^{m_0}\,P_{N-m_0}(z)}{\prod_{j=1}^{r}(1-p_jz)^{m_j}}\,, \qquad  p_j\in\D^*\,,\ p_j\neq p_k\,,
\]
where $P_{N-m_0}\in\C_{N-m_0}[z]$ if $m_0\geq 1$ thanks to \ref{deg P=N}\,.
And if $m_0=N$ i.e. all the $p_j$ in \eqref{Blaschke m0} vanish,  then $u(z)=Cz^N\,,$ $C
\in\C^*$.
Finally, it remains to prove that $(z-\pjbar)^{m_j-1}$ divides  the numerator of $u\,.$ Indeed, by identity \eqref{|u|2=zpzlog uppsi[u]-lambda[u]} of Step~2\,, $ u(z)\bar{u}\left(\frac{1}{z}\right)=z\p_z\log\uppsi_u-\nu_u$ where one computes by \eqref{Blaschke m0}\,,
\[
    z\p_z\log \uppsi_u=m_0+\sum_{j=1}^rm_j\left(\frac{1}{1-p_jz}+\frac{\pjbar}{z-\pjbar}\right)\,.
\]
That is, for all $m_0\in\llbracket0, N-1\rrbracket\,,$
\[
    \frac{P_{N-m_0}(z)}{\prod_{j=1}^r(1-p_jz)^{m_j}}\frac{z^{N-m_0}\,\overline{P_{N-m_0}}(\frac{1}{z})}{\prod_{j=1}^r(z-\pjbar)^{m_j}}
    =
    m_0-\nu_u+\sum_{j=1}^rm_j\left(\frac{1}{1-p_jz}+\frac{\pjbar}{z-\pjbar}\right)\,,
\]
where  $ p_j\in\D^*\,$, $p_k\neq p_j\,$ for $k\neq j\,$.
Observe that in the right--hand side, $\frac{1}{p_j}$ is a pole of multiplicity  one. Then, the same should hold for the left-hand side as well. Therefore, 
if $m_j\geq 2\,,$  $j= 1,\ldots r\,,$ this implies that  $\frac{1}{p_j}$ in the left--hand side must be a root of multiplicity $(m_j-1)$ of  $\overline{P_{N-m_0}}(\frac{1}{z})$\,. That is,
\[
    P_{N-m_0}(\pjbar)=0\ , \ \ldots \ ,\ P^{\,(m_j-2)}_{N-m_0}(\pjbar)=0\,,
\]
where $P^{(m)}_{N-m_0}$ is the $m^{\text{th}}$ derivative of $P_{N-m_0}$\,. As a result, $(z-\pjbar)^{m_j-1}$ divides $P_{N-m_0}(z)$\,, and so 
\[
    u(z)=z^{m_0}\prod_{j=1}^r\left(\frac{z-\pjbar}{1-p_jz}\right)^{m_j-1}\frac{q(z)}{\prod_{j=1}^r(1-p_jz)}\,, \qquad p_j\in\D^*\,,\ p_k\neq p_j\,,\ k\neq j\,,
\]
with $m_0\in\llbracket0,N-1\rrbracket\,,\,$ $m_1, \ldots,m_r\in\llbracket 1, N\rrbracket$\,,  
$
    m_0+\sum_{j=1}^rm_j=N\,,
$ 
and such that  $\deg(q)=r$ if $m_0\neq 0$  thanks to \ref{deg P=N}\,.

\vskip0.25cm
\noindent
\textbf{Step~4.} In this step, we write the rational function $u$ obtained in Step~3 on its  partial fractional decomposition 
\[
    u(z)=z^{m_0}\prod_{j=1}^r\left(\frac{z-\pjbar}{1-p_jz}\right)^{m_j-1}\left(a+\sum_{j=1}^r\frac{c_j}{1-p_jz}\right)\,, \quad p_j\in\D^*\,,\ p_k\neq p_j\,,\ k\neq j\,,
\] 
where  $a\neq 0$ if $m_0\neq 0$\,, and 
we infer  by  \eqref{|u|2=zpzlog uppsi[u]-lambda[u]} of Step~2 that, for all $j=1\,,\ldots\,,r\,,$
\be\label{cond a, ck-foc, preuve}
    \overline{a}c_{j}\,+\,\sum_{k=1}^{r}\,\frac{\overline{c_k}\,c_{j}}{1-\overline{p_k}\,p_j}=m_j\,. 
\ee
Indeed, by applying $\Pi$ to \eqref{|u|2=zpzlog uppsi[u]-lambda[u]}\,,
\[
    \Pi(\va{u}^2)=\Pi(z\p_z\log\uppsi_u -\nu_u)\,,
\]
Observe, on the one hand,
\[
    \Pi(z\p_z\log\uppsi_u -\nu_u)=\sum_{j=1}^r \,\frac{m_j}{1-p_jz}+m_0-\nu_u\,.
\]
And on the other hand,
\begin{align*}
    \Pi(\va{u}^2)=&\,\Pi\left(a+\sum_{j=1}^r\frac{c_j}{1-p_jz}\right)
    \\
    =&\,\va{a}^2+\sum_{j=1}^r\bar{a}c_j+\,\bar{a}\sum_{j=1}^r \frac{c_j}{1-p_jz}+\sum_{j=1}^r\sum_{k=1}^r\frac{c_j\overline{c_k}}{(1-p_j\overline{p_k})(1-p_jz)}\,.
\end{align*}
Therefore, for all $j=1,\ldots,r\,,$ the request conditions~\eqref{cond a, ck-foc, preuve} and 
$$
    \nu_u=m_0-\va{a}^2-\sum_{j=1}^r\bar{a}c_j\,.
$$

\vskip0.25cm
\noindent
\textbf{Step~5.}
We prove the converse. For $N\in\N$\,, let $u=Cz^N,$ $C\in\C^*,$ or 
% $u$ be a rational function 
    \[
    u(z)=z^{m_0}\prod_{j=1}^r\left(\frac{z-\pjbar}{1-p_jz}\right)^{m_j-1}\left(a+\sum_{j=1}^r\frac{c_j}{1-p_jz}\right)\,, \quad p_j\in\D^*\,,\ p_k\neq p_j\,,\ k\neq j\,,
    \]
    where  $m_0\in\llbracket0,N-1\rrbracket\,,\,$ $m_1, \ldots,m_r\in\llbracket 1, N\rrbracket$\,, such that $
    m_0+\sum_{j=1}^rm_j=N\,,
$  and
   $\ (a,c_1,\ldots c_r)\in \C\times \C^{r}\,,$ satisfy
\be\label{cond a, ck-foc, preuve Step 5}
\overline{a}c_{j}\,+\,\sum_{k=1}^{r}\,\frac{\overline{c_k}\,c_{j}}{1-\overline{p_k}\,p_j}=m_j\,,
\ee
with $a\neq0$ if $m\neq 0\,.$
Our aim is to prove that $u\in\U_N$\,, that is
\begin{itemize}
    \item $\e\, \uppsi\in\mathscr{B}_N$ such that $L_u\, S^k \uppsi=(\mu+k)S^k \uppsi$ for all $k\in\Nzero\,,$  where $\mu$ is a real constant.
    \item $\uppsi$ is of minimal degree, i.e. there does not exist  $\upphi\in\mathscr{B}_\ell$  with $\ell<N\,,$  such that $\upphi$ satisfies $L_u\, S^k \upphi=(\mu_1+k)S^k \upphi$ for all $k\in\Nzero\,.$ 
\end{itemize}
\vskip0.2cm
\noindent
For the moment, let us deal with the more complicated case, i.e. $u$ is a rational function.
We start by  proving the first point. Let
\be\label{uppsi}
    \uppsi:=\eee^{i\theta}z^{m_0}\prod_{j=1}^{r}\left(\frac{z-\pjbar}{1-p_jz}\right)^{m_j}\in\mathscr{B}_N\,, \quad \theta \in\R\,, \ p_k\in\D^*\,.
\ee
Observe that 
$\bar{u}\uppsi$ extends as a holomorphic function on $\D$ as $p_k\in\D\,.$ Then, by definition of $L_u\,,$
\[
    L_u\uppsi=z\p_z\uppsi-\va{u}^2\uppsi\,,
\]
where 
\[
    z\p_z\uppsi=m_0\uppsi+\sum_{k=1}^r\left(\frac{\pkbar}{z-\pkbar}+\frac{1}{1-p_kz}\right)\uppsi\,,
\]
and thanks to \eqref{cond a, ck-foc, preuve Step 5}\,, 
% Computing $\va{u}^2$ we find thanks to 
\begin{align}\label{u L2 pour u in U[N]}
    \va{u}^2
    =&\,\va{a}^2+\overline{a}\sum_{k=1}^r\frac{c_k}{1-p_kz}+a\sum_{k=1}^r\frac{\overline{c_k}\, z}{z-\pkbar}+\sum_{k=1}^r\sum_{j=1}^r\frac{c_k\overline{c_j}\, z}{(1-p_kz)(z-\overline{p_j})}\notag
    \\
    =&\,\va{a}^2+a\sum_{k=1}^r\,\overline{c_k}\,+\,\sum_{k=1}^r\left(\overline{a}c_k+\sum_{j=1}^r\frac{c_k\overline{c_j}}{1-p_k\overline{p_j}}\right)\frac{1}{1-p_kz}\notag
    \\
    &\hskip2.8cm+\,\sum_{j=1}^r\left(\overline{a}c_j+\sum_{k=1}^r\frac{c_k\overline{c_j}}{1-p_k\overline{p_j}}\right)\frac{\overline{p_j}}{z-\overline{p_j}}\notag
    \\
    =&\,\va{a}^2+a\sum_{k=1}^r\,\overline{c_k}\,+\,\sum_{k=1}^r\frac{m_k}{1-p_kz}+\sum_{j=1}^r \frac{m_j\,\overline{p_j}}{z-\overline{p_j}}\,.
\end{align}
Therefore,
\[
    L_u\uppsi=\Big(m_0-\va{a}^2-a\sum_{k=1}^{r}\,\overline{c_k}\Big) \uppsi
\]
Additionally, observe that 
for all $k\in\N\,,$ $\ps{S^k\uppsi}{u}=0\,.$
% \[
%     \ps{S^k\uppsi}{u}=\eee^{i\theta}\intc z^{k+m} \prod_{k=1}^{N-m}\left(\frac{z-\pkbar}{1-p_kz}\right)\cdot \frac{1}{z^m}\left(\overline{a}+\sumkNm\frac{\overline{c_k}\,z}{z-\pkbar}\right) \dz=0\,.
% \]
Hence, by applying the commutator identity \eqref{commutateur [Lu,s]}\,, we deduce 
\be\label{Lu Sk uppsi=(mu+k) Sk uppsi}
    L_u\, S^k \uppsi=(\mu+k)S^k \uppsi\,, \qquad \qlq k\in\Nzero\,,
\ee
where 
 $\mu:=m_0-\va{a}^2-a\sum_{k=1}^{r}\,\overline{c_k}\,.$ 

\vskip0.2cm
\noindent
It remains to prove that $\uppsi$ is of degree minimal. Suppose for the seek of contradiction that there exists $\upphi\in\mathscr{B}_\ell$\,, with  $\ell<N\,,$ such that $L_u\, S^{j} \upphi=(\mu_1+k)S^{j} \upphi$ for all $j\in\Nzero\,.$ By comparing the latter identity to \eqref{Lu Sk uppsi=(mu+k) Sk uppsi}, and thanks to \eqref{liminf} we infer that there exists $k'\,, \,k\in\N$ such that 
$S^k\uppsi=S^{k'}\upphi\,,$ i.e.
\[
     \upphi:=\eee^{i\tilde{\theta}}z^{m_0+k-k'}\prod_{j=1}^{r}\left(\frac{z-\pjbar}{1-p_jz}\right)^{m_j}\,, \quad  m_0+k-k'+\sum_{j=1}^rm_j=\ell<N.
\]
Therefore, by repeating Step~1 to Step~4, we infer that $u$ must be of the form
\[
    u(z)=z^{m_0+k-k'}\prod_{j=1}^r\left(\frac{z-\pjbar}{1-p_jz}\right)^{m_j-1}\left(a+\sum_{j=1}^r\frac{c_j}{1-p_jz}\right)\,, \
\]
which is a contradiction.
\end{proof}

\begin{corollary}\label{L2 norm of a finite gap} 
    Given $N\in \N$\,, let $u\in\U_N\,.$ Then,
    \begin{enumerate}[label=(\roman*),itemsep=2pt]
    \item In the focusing case, $\|u\|_{L^2}^2=N-\nu_u\,,$
    \item In the defocusing case,
        $\|u\|_{L^2}^2=\la_u-N\,,$
\end{enumerate}
where $\nu_u$ is the eigenvalue introduced in \eqref{LuSk uppsi[u]} and $\la_u$ is the corresponding one in the defocusing case.
% and $\la_u$ are respectively the corresponding eigenvalues of the Blaschke product of degree $N\,,$ satisfying \eqref{} for $L_u$ and $\Lutilde\,.$
\end{corollary}

\begin{Rq}
    Based on the previous  statement, one can conclude that for any potential $u\in\U_N$, we have
    \be\label{parametre nu}
    \begin{cases}
        \nu_u<N \text{ (focusing case) }
        \\
        \la_u>N \text{ (defocusing case) }
    \end{cases}\,.
    % \nu_u<N\,,\qquad \text{and} \qquad \la_u>N
    \ee
    \end{Rq}
    \begin{proof}
    Let $u\in\U_N\,.$ Then in light of the previous theorem, we have either $u=C\eee^{iNx}\,,$ $C\in\C^*$ or $u$ is the rational function~\eqref{u rational function th characterization of finite gap}\,. Thus, if $u=C\eee^{iNx}$ then the results follow easily by Remark~\ref{rq valeur propre finite gap}. Now, if $u$ is the rational function ~\eqref{u rational function th characterization of finite gap}\,, then
        by computing the $L^2$--norm of $u$ in the focusing case, we infer via \eqref{u L2 pour u in U[N]}\,,
    \begin{align*}
        \|u\|_{L^2}^2
        =&\,\intc\left(\va{a}^2+a\sum_{k=1}^{r}\,\overline{c_k}\,+\,\sum_{k=1}^{r}\frac{1}{1-p_kz}+\sum_{k=1}^{r} \frac{\overline{p_k}}{z-\overline{p_k}}\right)\dz
        \\
        =& \,\va{a}^2+a\sum_{k=1}^{r}\,\overline{c_k}+N-m\,,
    \end{align*}
    which is equal to $-\nu_u+N$ by \ref{nu[u]} of Remark~\ref{rq valeur propre finite gap}\,. For the defocusing case, we shall have
    \begin{align*}
        \|u\|_{L^2}^2
        =&\,\intc\left(\va{a}^2+a\sum_{k=1}^{r}\,\overline{c_k}\,-\,\sum_{k=1}^{r}\frac{1}{1-p_kz}-\sum_{k=1}^{r} \frac{\overline{p_k}}{z-\overline{p_k}}\right)\dz
        \\
        =&\, \va{a}^2+a\sum_{k=1}^{r}\,\overline{c_k}-N+m\,,
    \end{align*}
     which is equal to $\la_u-N$ by \ref{lambda[u]} of Remark~\ref{rq valeur propre finite gap}\,.
     \end{proof}

\begin{prop}\label{U[N] invariant by the evolution}
    For any $N\in \Nzero$\,, the set of finite gap potential  $\mathcal{U}_N$ is conserved along the flow of the \eqref{CS-DNLS}--equation.
\end{prop}

\begin{proof}
    Let $u_0$ be a finite gap potential in $\U_N\,,$ that is there exists $\uppsi_{u_0}\in\mathscr{B}_N$ of minimal degree $N$ satisfying 
    \be\label{Lu0 Sk uppsi=}
        L_{u_0}\,S^k\uppsi_{u_0}=(\nu_{u_0}+k)S^k\uppsi_{u_0}\,, \qquad \qlq k\geq 0\,.
    \ee
    Our aim is to prove that there exists $\varrho(t)\in\mathscr{B}_N$ of minimal degree~\footnote{In the sense, that there does not exist  $\upphi(t)\in\mathscr{B}_\ell$  with $\ell<N\,,$  such that $\upphi(t)$ satisfies $L_u\, S^k \upphi(t)=(\mu_1+k)S^k \upphi(t)$ for all $k\in\Nzero\,.$ } such that 
    \[
        L_{u(t)}\, S^k\varrho(t)=(\nu_{u_0}+k) S^k \varrho(t)\,, \qquad \qlq k\geq 0\,.
    \]
    Let $\varrho(t)$ be a solution of the Cauchy problem
     $$
        \begin{cases}\p_t\,\varrho(t) =B_{u(t)} \,\varrho(t)
        \\
        \varrho(0) =\displaystyle \uppsi_{u_0}
        % :=\prod_{k=1}^N\frac{z-\pkbar(0)}{1-p_k(0)\,z}\;\,,
        \end{cases}\,.
    $$
    Hence, by Remark~\ref{Rq gnt vect. propre+ evolution<u|fn>}\,, 
    \be\label{Lu rho=}
        L_{u(t)}\, \varrho(t)=\nu_{u_0}\, \varrho(t)\,.
    \ee 
    In addition, recall by Lemma~\ref{evolution in the gnt-foc}\,,
    \[
        \ps{S\varrho(t)}{u(t)}\,=\ps{S\uppsi_0}{u_0}\eee^{-i\nu_{u_0}^2t}\,,
    \]
    where here $\ps{S\uppsi_0}{u_0}$ vanishes after combining the commutator identity~\eqref{commutateur [Lu,s]} and equation~\eqref{Lu0 Sk uppsi=}\,. Therefore, by \eqref{commutateur [Lu,s]}\,,
    \be\label{Lu S rho}
        L_{u(t)}\, S\varrho(t)=(\nu_{u_0}+1)S\rho(t)\,.
    \ee
    This yields to
    % implies that $S\varrho$ satisfies the Cauchy problem 
    $$
        \begin{cases}\p_t\,S\varrho(t) =B_{u(t)} \,S\varrho(t)
        \\
        S\varrho(0) =\displaystyle S\uppsi_{u_0}
        % :=\prod_{k=1}^N\frac{z-\pkbar(0)}{1-p_k(0)\,z}\;\,,
        \end{cases}\,.
    $$
Indeed, by the commutator identity \eqref{commutateur [Bu,s]}\,,
\begin{align*}
    \p_t\,S\varrho(t)=
    &\,S B_{u(t)}\varrho(t)
    \\
    =&\,B_{u(t)}S\varrho(t)-i\Big( \Lutilde^2 S \,-\, S(\Lutilde+\Id)^2 \Big)\varrho(t)\,,
\end{align*}
which is equal to $\p_t\,S\varrho(t)=B_{u(t)}S\varrho(t)$ thanks to \eqref{Lu rho=} and \eqref{Lu S rho}\,.
Consequently, by repeating the same procedure, we obtain for all $k\in\Nzero\,,$
\[
    L_{u(t)}\, S^k\varrho(t)=(\nu_{u_0}+k) S^k \varrho(t)\,, 
\]
with 
$$
        \begin{cases}\p_t\,S^k\varrho(t) =B_{u(t)} \,S^k\varrho(t)
        \\
        S^k\varrho(0) =\displaystyle S^k\uppsi_{u_0}
        % :=\prod_{k=1}^N\frac{z-\pkbar(0)}{1-p_k(0)\,z}\;\,,
        \end{cases}\,.
    $$
Besides, observe that $\varrho(t)\in \mathscr{B}_N$\,. Indeed, by applying  Lemma~\ref{evolution in the gnt-foc}\,,
    $$
        \ps{S^k\varrho(t)}{\varrho(t)}=\ps{S^k\uppsi_{u_0}}{\uppsi_{u_0}} \,\eee^{i((\nu_{u_0}+k)^2-\nu_{u_0}^2)\,t}\,=0\,,\quad \qlq k\in\N\,,
    $$
    leading to
    $$
        \ps{\eee^{ikx}}{\va{\varrho(t)}^2}=0\,,\quad k\in \Z\bk\lracc{0}.
    $$
    Thus, following the same lines of  the proof of Proposition~\ref{produit-de-Blashke}, we deduce that $\varrho(t)$ is a finite Blaschke product.
     To infer that the degree of this finite Blaschke product is $N$\,, we should notice that each of $\uppsi_{u_0}$ and $\varrho(t)$ enjoys an inverse spectral formula (Remark~\ref{formule d'inversion f_n})
    \begin{gather*}
        \uppsi_{u_0}=\ps{(\Id-zM_{\leq N}(u_0))^{-1} \mathds{1}_N}{Y_{\leq N}(u_0)}\;,
            \\
        \varrho(t)=\ps{(\Id-zM_{\leq N}(u(t)))^{-1} \mathds{1}_N}{Y_{\leq N}(u(t))} \,.
    \end{gather*}
    where $M_{\leq N}(u_0)$ and  $M_{\leq N}(u(t))$  are  the finite matrix of order $(N+1)\times (N+1)$ obtained respectively from the representation matrix of $S^*$ in the $L^2$ basis $(h_k)_{k=0}^{N-1}\cup (S^k\uppsi_{u_0})_{k\geq 0}$ constituted of the eigenfunctions of $L_{u_0}$ at $t=0$\,, and from the eigenfunctions $(e_k(t))_{k=0}^{N-1}\cup (S^k\varrho(t))_{k\geq 0}$ of $L_{u(t)}$ at any time $t$\,. Therefore, in view of the fourth identity  of Lemma~\ref{evolution in the gnt-foc}\,, we infer
    $$
       M_{\leq N}(u(t))=\mrm{Diag}(\eee^{-i(\nu_n+1)^2t})\,M_{\leq N}(u_0)\,\mrm{Diag}(\eee^{-i\nu_n^2t})\,. 
    $$
    That is, 
    $$
        \big\lvert\det\big(M_{\leq N}(u_0)\big)\big\rvert=\big\lvert\det\big(M_{\leq N}(u(t))\big)\big\rvert
    $$
    and so, 
    \be\label{deg}
        \deg\big(\det(\Id-zM_{\leq N}(u(t)))\big)=\deg\big(\det(\Id-zM_{\leq N}(u_0))\big)=N\,.
    \ee
    As a result, $u(t)\in \U_n$ with $n\leq N$\,.
    It remains to show that $u(t)\notin \U_n$ with $n<N\,.$ Suppose that there exists $\upphi(t)\in \mathscr{B}_n$ with $n<N$ such that 
    $$
    L_{u(t)}\, S^k\upphi(t)=(\nu_{u}+k) S^k \upphi(t)\,,
    $$
    then applying the same above procedure, we infer that  $\upphi(0)\in \mathscr{B}_n$ with $n<N$ and 
    $$
    L_{u_0}\, S^k\upphi(0)=(\la_{u}+k) S^k \upphi(0)\,,
    $$
    leading to $u_0\in \U_n$\,, $n<N$ which is a contradiction. 
     \vskip0.2cm
    Note that the same proof works in the defocusing case.
    
\end{proof}

%------------------------------------------------------------------------------
%------------------------------------------------------------------------------

\section{\textbf{Remark on the regularity of \texorpdfstring{$u$}{u}}}\label{Rk on the regularity}

Recall that in the beginning of Section~\ref{spectral property of the Lax operator}\,, we have supposed for more convenience that $u$ is a function with enough regularity, typically in $\Htwo\,.$ However, the same strategy adopted to derive the traveling waves of the Calogero--Sutherland DNLS equation \eqref{CS-DNLS} and to characterize the finite gap potentials can be extended to less regularity spaces. \textit{ In this section,
 we discuss some remarks that allow the extension of the main results to the critical regularity $\Ltwo\,$.}

\vskip0.25cm
First, we recall from \cite{Ba23} the following Theorem.

\begin{theorem*}[\cite{Ba23}]
 For any $0\leq s\leq 2$, let $u_0\in H^s_+(\T)$\,. Then,
    there exists a unique potential $u\in\mathcal C(\R, H^s_+(\T))$ solution of \eqref{CS-defocusing} such that,
     for any sequence $(u_0^\varepsilon)\subseteq H^2_+(\T)\,,$
    $$
        \|u_0^\eps-u_0\|_{H^s}\underset{\eps\to0}{\longrightarrow} 0\,,
    $$
     we have for all $T>0\,,$
     $$
        \sup_{t\in[-T,T]}\|u^\eps(t)-u(t)\|_{H^s}\to 0\,, \quad \eps\to 0\,.
    $$
      Moreover, the $L^2$--norm of the limit potential $u$ is conserved in time :
     \bes
        \|u(t)\|_{L^2}=\|u_0\|_{L^2}\,,\quad \qlq t\in \R.
    \ees
    Furthermore, the same holds for \eqref{CS-focusing} under the additional condition $\|u_0\|_{L^2}<1$\,.
\end{theorem*}

At a second stage, recall that Lemma~\ref{evolution in the gnt}\,, Proposition~\ref{tauct fn} and Corollary~\ref{evolution theta[n]} have been the keys to characterize the traveling waves for the defocusing equation~\eqref{CS-defocusing}\,, 
and Lemma~\ref{evolution in the gnt-foc}\,, Proposition~\ref{tauct fn foc} and Corollary~\ref{evolution theta[n] foc} for the focusing equation~\eqref{CS-focusing}\,. As a result, we need to extend these Proposition/lemma/Corollary to less regular potentials $u$\,.
% Therefore, we need to understand how the eigenvalues  and the eigenfunctions 
Hence, 
we recall from \cite[Corollary~3.12]{Ba23} the following result.

% give a meaning to the eigenfunction of the Lax operators $L_u$ and $\Lutilde$ for $u\in\Ltwo\,.$ This was established in \cite[Appendix A]{GL22}

% recall a consequence of the previous the, we have obtained 

\begin{corollary*}[Corollary~3.12 of \cite{Ba23}]
     For any $0\leq s\leq 2$, let $u_0\in H^s_+(\T)$\,. There exists an orthonormal basis $(g_n^t) $ of $L^2_+(\T)$ constituted from the eigenfunctions of $L_{u(t)}$\,,
		such that for all $n\in \Nzero\,,$ 
		% $L_{u(t)}f_n(t)=\l_nf_n(t)$ and
		\be\label{coordonnees de u dans la base des vect.propre fn}
		\ps{u(t)}{\gnt}=\ps{u_0}{\fnuzero}\eee^{-it \la_n^2(u_0)}\,, \quad\qlq t\in\R\,,
		\ee
  where $u(t)$ is the solution of \eqref{CS-defocusing} starting at $u_0$ at $t=0\,.$
  Furthermore, the same holds for \eqref{CS-focusing} under the additional condition $\|u_0\|_{L^2}<1$\,.
	\end{corollary*}

 \begin{Rq}
     Note that there is a point hidden in the previous corollary, namely, the fact that $L_u$ is well--defined with $u\in\Ltwo\,.$ We refer the readers to \cite[Appendix A]{GL22} for the construction of this operator and to \cite[Corollary 3.2]{Ba23} for a way to identify its spectrum.
 \end{Rq}
 By repeating the same analysis of the  proof of \cite[Corollary 3.12]{Ba23}, one can establish  the existence of an orthonormal basis $(g_n^t) $ of $L^2_+(\T)$ satisfying 
     \begin{align*}
        \ps{1}{\gnt}=&\,\ps{1}{\fnuzero}\,\eee^{-i\la_n^2t}\ ,
        \\
        \ps{S\gpt}{\gnt}=&\,\ps{S\fpuzero }{\fnuzero}\,\eee^{i((\la_p+1)^2-\la_n^2)\,t}\,.
        % \ps{S\gpt}{u}=&\,\ps{Sf_p^{\, u_0}}{u_0}\eee^{i(\la_p+1)^2\,t}\,.
        \notag
    \end{align*}

\vskip0.25cm
Finally, in Section~\ref{finite gap potentials}\,, more precisely  in \eqref{in L1}\,, we made use  of the fact that the domain of the Lax operator $L_u$ with $u\in \Htwo$ is $H^1_+(\T)$  in order to infer that $\Pi(\bar{u}\uppsi)\in L^2\,.$ However, it should be noted that  the lax operator $L_u$ with $u\in\Ltwo$ has its domain a subset of $H^{\frac12}_+(\T)$ \cite[Appendix A]{GL22}\,. Hence, we need the following lemma to infer the result.

	\begin{lemma*}[Lemma~2.7 of \cite{Ba23}]
		Let $h\in H^{\frac12}_{+}(\T)\,$, $u\in\Ltwo\,,$
		\be\label{inegalite-Tu-bar-h}
		\Ldeux{T_{\bar{u}}h}^2\leq \left(\ps{Dh}{h}+\Ldeux{h}^2\right)\Ldeux{u}^2\,,
		\ee
  where we recall $T_u$ was defined in \eqref{Toeplitz operator}\,.
	\end{lemma*}

%------------------------------------------------------------------------------
%------------------------------------------------------------------------------
\section{\textbf{Open problems}}
\label{Open Problems}
1. The full characterization of the traveling waves $u_0(x-ct)$ of \eqref{CS-focusing} is still an open problem. 
\vskip0.25cm
2.
Note that along this paper, we have treated the case where the traveling waves of the Calogero--Sutherland DNLS equation \eqref{CS-DNLS} are of the form 
\[
    u(t,x):=u_0(x-ct)\,, \qquad c\in\R\,.
\]
But, one may wonder if there exist traveling wave solutions with a phase factor, such as
\be\label{traveling waves with a phase}
    u(t,x):=\eee^{i\varphi(t) }u_0(x-ct)\,, \qquad \varphi(t)\,,\,c\in\R\,.
\ee
% or in a more general form, as $u(t,x)= U(x-ct)$\,, where $c\in\R\,.$
However, let us underline the following feature : observe that the mean $\ps{u}{1}$ is conserved along the flow of the Calogero--Sutherland DNLS equation \eqref{CS-DNLS}\,, for any solution $u$ in the Hardy space of the circle $\T\,$. Indeed, by applying an integration by parts and since $u$ is in the Hardy space, then
\[
    i\p_t\ps{u}{1}=-\ps{\p_x^2u}{1}\pm 2 \ps{D\Pi(\va{u}^2)}{\bar{u}}=0\,.
\]
Therefore,
\begin{itemize}
    \item If $\ps{u_0}{1}\neq 0$, then  $\varphi(t)$ in \eqref{traveling waves with a phase} must be a constant in time.
    % we do not have traveling wave solutions of the form \eqref{traveling waves with a phase} for the Calogero-Sutherland DNLS equation \eqref{CS-DNLS} when $\ps{u_0}{1}\neq 0$\,.
    \item In regard to the case where $\ps{u_0}{1}= 0$\,, the question of the existence of traveling waves of \eqref{CS-DNLS} of the form \eqref{traveling waves with a phase} remains an open problem. However, one can easily prove that $(\varphi(t),c)$ are related via the following identity 
    \[
    \varphi'(t)  -Nc =-N^2 \,,
    \]
    where $N$ is the positive integer appearing after rewriting $u_0$ as $u_0=S^Nv_0$ with $\ps{v_0}{1}\neq 0\,,$ as $\ps{u_0}{1}=0$\,. Indeed, 
    % since $\ps{u}{1}$ is conserved along the flow of \eqref{CS-DNLS} then 
    by writing the solution $u(t,x)$ as
    \begin{align*}
        u(t,x)=&\,\eee^{i\varphi(t)}u_0(x-ct)
        \\
        =&\,\eee^{i\varphi(t)} \eee^{iN(x-ct)}v_0(x-ct)\,,
    \end{align*}
    one observes that if $u$ satisfies \eqref{CS-DNLS}\,, then
    \be\label{eqt en v}
    \hskip1cm
    \begin{cases}
        -(\varphi'(t)-Nc)\,v_0 -N^2v_0 +P(\p_xv_0\,,\,\p_x^2v_0)\mp2i\p_x\Pi(\va{v_0}^2)v_0=0\,.
        \\
        P(w,\tilde{w}):=(2N-c)i\,w+\tilde{w}
    \end{cases}
    \ee
    We conclude by taking the inner product of the last identity with $1\,,$ that
    \[
            \varphi'(t)  -Nc =-N^2 \,.
    \]
%     \vskip3cm
%     \be\label{u=e v}
%         u(t,x)=\eee^{iNx}v(t,x)\,.
%     \ee
%     Thus, if $u$ satisfies \eqref{CS-DNLS}\,, then $v$ is solution to the following  equation
%     \be\label{eqt en v}
%     \begin{cases}
%         i\p_tv -N^2v +P(\p_xv,\p_x^2v)\mp2i\p_x\Pi(\va{v}^2)v=0\,.
%         \\
%         P(v_1,v_2):=2iN\,v_1+v_2
%     \end{cases}
% \ee
% By taking the inner product of \eqref{eqt en v} with $1$\,, we infer
% $$
%     i\p_t\ps{v}{1}=N^2\ps{v}{1}\,.
% $$
% As a result,
% $
%     \ps{v}{1}=\ps{v_0}{1}\eee^{-iN^2t}\,,
% $
% and so
% \bes
%     \varphi(t)-Nct=-N^2t\,,
% \ees
% thanks to   \eqref{u=e v}\,, \eqref{traveling waves with a phase} and the fact that  $u_0(x-ct)=\eee^{iN(x-ct)}v_0(x-ct)\,.$
% %  Therefore, rewriting \eqref{u=e v} as
% %  $
% %     \eee^{i\omega t}u_0(x-ct)=\eee^{iNx}v(t,x)\,,
% %  $
% %  or as 
% % \[
% %     \eee^{i\omega t}\eee^{-ict}v_0(x-ct)=v(t,x) 
% % \] 
% % since 
% % then one infer 

% %  by taking the inner product with $1$ and comparing the obtained result with \eqref{<v|1>=<v0|1> e} 
\end{itemize}

\section*{\textbf{Appendix}} \label{Appendix}
\begin{enumerate}[wide=0pt, leftmargin=0pt, labelwidth=0pt, align=left, itemsep=5pt]
    \item \label{Appendix 1} \textit{The following counterexample illustrate the necessity of the condition $\nu_n\neq 0$ in order to obtain the first point in Proposition~\ref{gap=0 and eigenspaces}}\,.
    \vskip0.2cm
    \noindent
    Consider the 0--gap potential (i.e. a potential satisfying $\gamma_n(u)=0$ for all $n\in\Nzero$\,, where $\gamma_n(u)$ is defined in~\eqref{gap gamma n})
    \[
        u(z)=\frac{\sqrt{1-\va{p}^2}}{1-pz}\,, \qquad p\in\D\,.
    \]
    One can easily check that $L_uf_0=-f_0$ for 
    \[
        f_0(z):=\frac{\sqrt{1-\va{p}^2}}{1-pz}\,,
    \]
    and that, for all $k\in\Nzero\,,$ 
    $L_uS^k\,\uppsi=k\,S^k\uppsi$
    where
    \[
        \uppsi(z):=\frac{z-\overline{p}}{1-pz}\,.
    \]
    Therefore, the spectrum of $L_u$ is given by
    $$
    \sigma(L_u)=\{-1\,<\,0\,<\,1\,<\,2\,<\,\ldots\,<\,n\,<\,n+1\,<\,\ldots\}\ ,
$$  
where notice 
$\nu_1=\nu_0+1$ and $Sf_0 \not \equiv \uppsi\,.$
\item \label{Appendix 2}
\textit{In this part of the Appendix we prove that the two integers $N_1$ and $N_2$  appearing in Corollary~\ref{<u|f[n]> <=> gamma[n]=0} are not necessarily equal.}
\vskip0.2cm
\noindent
% For that, we revisit the same example of Appendix~\ref{Appendix 1} by assuming in addition that  $0<p_1<1$ and $0<p_2<1\,.$
% Observe that we have obtained $\nu_2=\nu_1+1\,,$ and on the other side $\ps{u}{f_1}\neq 0\,.$
% Indeed,  recall 
% $$  
%     f_1=\frac{a_1}{1-p_1z}+ \frac{b_1}{1-p_2z}\,,
% $$
% where $(a_1,b_1)$ are defined in \eqref{a1,b1}\,. By computing 
% \begin{align*}
%     \ps{u}{f_1}
%     =&\,\intc\left(\frac{c_1}{1-p_1z}+\frac{c_2}{1-p_2z}\right)\left(\frac{\bar{a}_1z}{z-\bar{p}_1}+\frac{\bar{b}_1z}{z-\bar{p}_2}\right)\dz
%     \\
%     =&\;\frac{c_1\bar{a}_1}{1-\va{p_1}^2}+\frac{c_1\bar{b}_1}{1-p_1\bar{p}_2}+\frac{c_2\bar{a}_1}{1-p_2\bar{p}_1}+\frac{c_2\bar{b}_1}{1-\va{p_2}^2}
%     \\
%     =&\;\frac{\bar{a}_1}{\bar{c}_1}+\frac{\bar{b}_1}{\bar{c}_2}\,,
% \end{align*}
% where the last identity holds, thanks to \eqref{conditions_ci}
% Now, if $\ps{u}{f_1}=0$\,, then
% \[
%     a_1c_2=-c_1b_1\,.
% \]
% that is 
% \[
%     \frac{ \va{c_{1}}^{2}c_2}{1-\overline{p_1}p_2}
%     +
%     \frac{c_{1}\va{c_2}^2}{1-\va{p_2}^2}
%     =
%     \frac{c_{2}\va{c_{1}}^2}{1-\va{p_1}^2}
%     +
%     \frac{c_1\va{c_{2}}^{2}}{1-p_{1}\overline{p_{2}}}
% \]

Let 
\[
    u(z):=\frac{\sqrt{2(1-\va{p}^4)}\,z}{1-p^2z^2}\,, \qquad p\in\D^*\,.
\]
For such $u$\,, one can check  that  $$\displaystyle\uppsi_u:=\frac{(z-\bar{p})(z+\bar{p})}{(1-pz)(1+pz)}$$ is an eigenfunction of $L_u$ associated with the eigenvalue $0\,.$
Additionally, for all $k\in\Nzero\,,$
\begin{align*}
    \ps{S^k\uppsi_u}{u}=\intc \frac{z^k (z^2-\bar{p}^2)}{1-p^2z^2}\frac{\sqrt{2(1-\va{p}^4)}z}{z^2-\bar{p}^2}\dz=0\,,
\end{align*}
leading, by \eqref{commutateur [Lu,s]}\,, to 
$
    L_uS^k\uppsi_u=k\, S^k \uppsi_u\,,$ for all $  k\in\Nzero\,.
$
Note that  $\deg \uppsi_u=2\,$, then it remains to find two eigenvectors of $L_u\,,$ generating the model space~\footnote{~\cite[Corollary 5.18]{GMR16}} $(\uppsi_u\Ltwo)^\perp$\,.
First, we have $L_u1=0$ as $L_u 1=-\ps{1}{u}1$ and $\ps{u}{1}=0$\,.
Second, by taking 
\[
    f_0=\frac{\sqrt{1-\va{p}^4}\,z}{1-p^2z^2}\,,
\]
one has, $L_u f_0=-f_0\,.$ Therefore, by denoting for all $k\in\Nzero$\,, $\ f_{2+k}:=S^k\uppsi_u\,$ and $f_1:=1\,,$
we have $\ps{u}{f_n}=0$ for all $n\geq 1$\,.   But, on the other side,
$
    \nu_2-\nu_1-1=0-0-1\neq 0\,.
$
\end{enumerate}

\vskip0.5cm

%%%%%%%%%%%%%%%%%%%%%%%%%%%%%%%%%%%%%%%%%%%%%%%%%%%%%%%%
%%%%%%%%%%%%%%%%%%%%%% Bibliographie %%%%%%%%%%%%%%%%%%%
%%%%%%%%%%%%%%%%%%%%%%%%%%%%%%%%%%%%%%%%%%%%%%%%%%%%%%%%

\end{document}